\documentclass[10pt,reqno]{amsart}

\usepackage{amssymb,latexsym,amsmath,amsthm,amsfonts}
\usepackage{tikz, tikz-cd}
\usetikzlibrary{shapes.geometric}
\usepackage{hyperref}

\usepackage{color}
\usepackage{comment} 
\usepackage{float}
\usepackage{lscape}

\newtheorem{theorem}[equation]{Theorem}
\newtheorem*{theorem*}{Theorem}
\newtheorem{lemma}[equation]{Lemma}
\newtheorem{corollary}[equation]{Corollary}
\newtheorem{definition}[equation]{Definition}
\newtheorem{example}[equation]{Example}

\newtheorem{proposition}[equation]{Proposition}

\theoremstyle{remark}
\newtheorem{remark}[equation]{Remark}

\numberwithin{equation}{section}

\newcommand{\C}{\mathbb C}
\newcommand{\Q}{\mathbb Q}
\newcommand{\R}{\mathbb R}
\newcommand{\Z}{\mathbb Z}

\newcommand{\A}{\mathbb A}

\renewcommand{\a}{\mathfrak a}

\newcommand{\F}{{\mathcal F}}

\newcommand{\cP}{{\mathcal P}}

\newcommand{\cO}{{\mathcal O}}
\newcommand{\cV}{{\mathcal V}}

\newcommand{\hC}{\widehat{C}}

\newcommand{\SL}{\mathrm{SL}}

\newcommand{\U}{\mathrm{U}}

\newcommand{\Log}{\operatorname{Log}}
\newcommand{\vol}{\operatorname{vol}}
\newcommand{\Spec}{\operatorname{Spec}}
\newcommand{\Span}{\operatorname{Span}}
\newcommand{\Hom}{\operatorname{Hom}}
\newcommand{\codim}{\operatorname{codim}}

 \newcommand{\adots}{\mathinner{\mkern1mu\raise1pt\hbox{.} 
\mkern2mu\raise4pt\hbox{.} 
\mkern2mu\raise7pt\hbox{.}\mkern1mu}}

\begin{document}

\title[Combinatorics of Trace Formula]{On Combinatorics of the Arthur Trace Formula, Convex Polytopes, and Toric Varieties}

\author[M. Asgari \& K. Kaveh]{Mahdi Asgari \ \ and \ \  Kiumars Kaveh}
\address{Oklahoma State University\\ 
Department of Mathematics \\
Stillwater, OK \\
USA}
\email{asgari@math.okstate.edu}

\address{Department of Mathematics, University of Pittsburgh,
Pittsburgh, PA, USA}
\email{kaveh@pitt.edu}

\begin{abstract}
We explicate the combinatorial/geometric ingredients of Arthur's proof of the convergence and polynomiality, in a truncation parameter, 
of his non-invariant trace formula.  Starting with a fan in a real, finite dimensional, vector space and a collection of functions, one for 
each cone in the fan, we introduce a combinatorial truncated function with respect to a polytope normal to the fan and prove the 
analogues of Arthur's results on the convergence and polynomiality of the integral of this truncated function over the vector space.  
The convergence statements clarify the important role of certain combinatorial subsets that appear in Arthur's work and provide 
a crucial partition that amounts to a so-called nearest face partition.  
The polynomiality statements can be thought of as far reaching extensions of the Ehrhart polynomial. Our proof of polynomiality relies on 
the Lawrence-Varchenko conical decomposition and readily implies an extension of the well-known combinatorial lemma of Langlands. 
The Khovanskii-Pukhlikov virtual polytopes are an important ingredient here. 
Finally, we give some geometric interpretations of our combinatorial truncation on toric varieties as a measure and a Lefschetz number.  
\end{abstract}

\maketitle

\section{Introduction} 
The Arthur Trace Formula is a vast generalization of the Selberg Trace Formula to arbitrary rank reductive groups.  
The first incarnation of Arthur's trace formula, the non-invariant trace formula, relies on two crucial ingredients: the integral of a {\it truncated} 
kernel (of a compactly supported test function) is absolutely convergent and the integral depends polynomially on the truncation 
parameter (which he has to assume is sufficiently regular). 
The purpose of this work is to prove two general, purely combinatorial, statements about polytopes, 
one on convergence and the other on polynomiality of certain integrals. 
These statements essentially capture, 
and generalize, the combinatorial aspects of Arthur's corresponding results (cf. \cite{arthur-duke} and \cite{arthur-annals}), 
isolating them from the analytic aspects that use reduction theory and other techniques.  
The long term hope for our project, of which this work is a first step, is to aim at applications of the 
Arthur Trace Formula to more general test functions \cite{Hoffmann,Finis-Lapid,Finis-Lapid-Mueller,Finis-Lapid-Vietnam}.

We also give interpretations of our combinatorial results in terms of the geometry of toric varieties. 
We hope the present paper would shed light on the combinatorics behind Arthur's trace formula and 
its similarity with certain concepts appearing in toric geometry.  The connection between polyhedral combinatorics 
appearing in Arthur’s trace formula and in toric varieties is not quite transparent yet. In this regard  
we mention the articles of Kottwitz \cite{Kottwitz} and Finis-Lapid \cite{Finis-Lapid} which may be relevant.

We now briefly recall the trace formula before explaining a summary of our results and proofs.

\subsection*{Arthur's non-invariant trace formula} 
For a finite group $G$ the character of a representation of $G$ (or any conjugation invariant function on $G$ for that matter) 
can be written uniquely as a linear combination of characteristic functions of different conjugacy classes, as well as, 
a linear combination of traces of irreducible representations. The equality of these two decompositions is a special case of the 
Frobenius Reciprocity which  plays an important role in representation theory of finite groups. 
This is the prototype of many trace formulas in representation theory.

Arthur gave a far reaching trace formula for arbitrary reductive groups defined over number fields.  
A main problem is that in this generality the integral representing the trace diverges. 
Arthur introduces an operation of {\it truncation} to modify this integral so that it becomes convergent.

The (non-invariant) Arthur Trace Formula (ATF) is an equality of two distributions 
\begin{equation}\label{atf} 
J_{\mathrm{geom}}(f) = J_{\mathrm{spec}}(f), \quad f \in C_c^\infty(G(\A)^1).   
\end{equation} 
Here $G$ is a connected reductive linear algebraic group defined over $\Q$ (or any number field) 
whose ring of adeles we denote by $\A$ 
and $G(\A)^1$ consists of those $x \in G(\A)$ satisfying $|\chi(x)|_\A = 1$ for all rational characters $\chi$ of $G$. 
Both the {\it geometric} and the {\it spectral} distributions on the two sides of \eqref{atf} are equal to the integral 
over $G(\Q) \backslash G(\A)^1$ of a {\it modified kernel} $k^T(x) = k^T(x,f)$ at a certain value $T=T_0$ of a 
suitably regular {\it truncation parameter} $T$ belonging to the positive Weyl chamber of $G$ with respect to a fixed 
minimal parabolic subgroup.  The space $G(\Q) \backslash G(\A)^1$ is in general finite volume (with respect to the 
Haar measure on $G$), but only compact 
when $G$ has no proper parabolic subgroups.  While the trace formula in the case of compact quotient was well understood, 
already the development of the trace formula in the case of $\SL(2,\Z) \backslash \SL(2,\R)$ led A. Selberg 
to his celebrated Selberg Trace Formula.  However, Arthur realized that the presence of proper parabolic subgroups 
in a more general group $G$ makes the integral of the kernel function divergent. As a result, he introduced the 
modified kernel $k^T(x)$.  Two major properties of the modified kernel (see \cite{arthur-duke} and \cite{arthur-annals}) are the following:  
\begin{itemize} 
\item[(1)] 
$ 
\int\limits_{G(\Q) \backslash G(\A)^1} |k^T(x)| \, dx < \infty 
$ 
for suitably regular truncation parameter $T$.     

\item[(2)] The function 
$ T \mapsto J^T(f) = \int\limits_{G(\Q) \backslash G(\A)^1} k^T(x) \, dx $ 
is a polynomial function of $T$. 
\end{itemize} 
As the truncation parameter $T$ goes further away from the origin, the integral of $k^T(x)$ gets closer to the (divergent) integral 
representing the trace. Among other things, the proofs involve quite intricate combinatorics of convex polytopes and convex cones.  
Expanding the modified kernel geometrically (via conjugacy classes) and spectrally (via automorphic representations) then 
provides the two sides of (the truncated analog of) the identity \eqref{atf}.

In the function field case one also has an analogue of the ATF and the truncation parameter $T$.  In particular, we mention the work of 
G. Laumon \cite{laumon-vol1,laumon-vol2} where he develops the trace formula for certain class of test functions for which the modified kernel 
$k^T(\cdot)$ turns out to be equal to the usual kernel $k(\cdot)$.  This makes the question of polynomiality obvious since the resulting 
polynomials would simply be constant. However, the convergence question still remains and indeed a similar argument as Arthur's 
in the number field case applies.

\subsection*{Main results} 
We introduce a notion of \emph{combinatorial truncation} and prove two main results on its
convergence and polynomiality.  The idea for our results is to start with a complex-valued function on a finite dimensional 
real vector space whose integral over the vector space is possibly divergent.  We then ``truncate'' this function by 
subtracting some other functions around some neighborhoods of infinity to arrive at a ``truncated function'' whose 
integral over the vector space is absolutely convergent.  The ``neighborhoods of infinity'' are with respect to a 
toric compactification of $V$ (in the sense of Sections \ref{subsec-toric-background} and \ref{subsec-toric-positive}) 
whose data is encoded in a polytope and its normal fan.  We then prove that the integral of the truncated function, 
as a function of the polytope, is indeed a polynomial function.

To explain our results we introduce some notation and refer to Sections \ref{subsec-fan} and \ref{subsec-polytope} 
for further details on convex cones and polytopes.  We first explain our convergence results.

Let $V \cong \R^n$ be an $n$-dimensional real vector space.  
We fix an inner product $\langle \cdot, \cdot \rangle$ on $V$ and use it to identify $V$ with its dual $V^*$.  
Fix a full dimensional, complete, simplicial fan $\Sigma$ in $V$ and fix a polytope $\Delta \in \cP(\Sigma)$, the set of polytopes 
with normal fan $\Sigma$.  There is a one-to-one correspondence between the cones in $\Sigma$ and the faces of $\Delta$. 
For $\sigma \in \Sigma$, we let $T^-_{\Delta, \sigma}$ denote the outward looking tangent cone of $\Delta$ 
at the face corresponding to $\sigma$ (see Section \ref{subsec-polytope} and Figures \ref{fig-TC1}--\ref{fig-TC2}.)

Suppose a function $K_0: V \to \C$ is given with $\int\limits_V K_0(x) \, dx$ possibly divergent.  In fact, let $K_0$ be a member 
of a collection of functions $K_\sigma: V \to \C$, one for each $\sigma \in \Sigma$.  
We will assume that $K_\sigma$ is invariant in the direction of $\Span(\sigma)$, i.e., $K_\sigma(x+y) = K_\sigma(x)$ 
for $x \in V$ and $y \in \Span(\sigma)$.

Associated with the collection 
$(K_\sigma)_{\sigma \in \Sigma}$ and the polytope $\Delta$ we define the \emph{truncated function} 
\begin{equation}   \label{equ-intro-k-Delta}
k_\Delta(x) = \sum\limits_{\sigma \in \Sigma} (-1)^{\dim \sigma} K_{\sigma}(x) {\mathbf 1}_{T^-_{\Delta, \sigma}(x)},
\end{equation}
where ${\mathbf 1}$ denotes the characteristic function. 
We think of $k_\Delta(x)$ as a ``truncation'' of $K_0$ by means of the polytope $\Delta$ and the functions 
$K_\sigma$ for non-zero cones $\sigma \in \Sigma$. 

Note that the function $k_{\Delta}(x)$ and $K_0(x)$ coincide for $x \in \Delta$. In fact, if all the $K_\sigma$ 
are identically equal to $1$, by the classical Brianchon-Gram theorem (cf. Theorem \ref{th-BG}), 
the function $k_\Delta(x)$ coincides with the characteristic function of $\Delta$ (see page \pageref{simple-example}, Another Simple Example). 

\begin{figure}
\includegraphics[height=5cm]{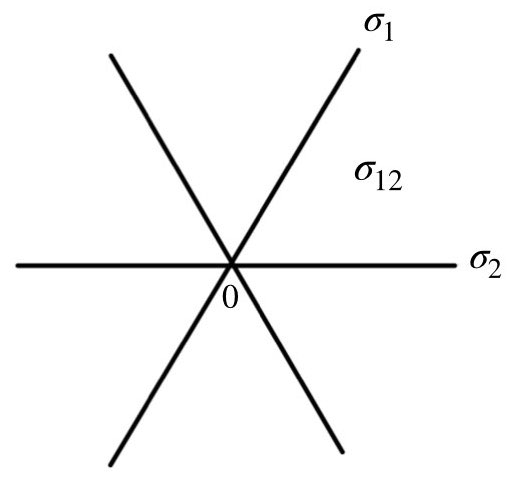}
\includegraphics[height=7cm]{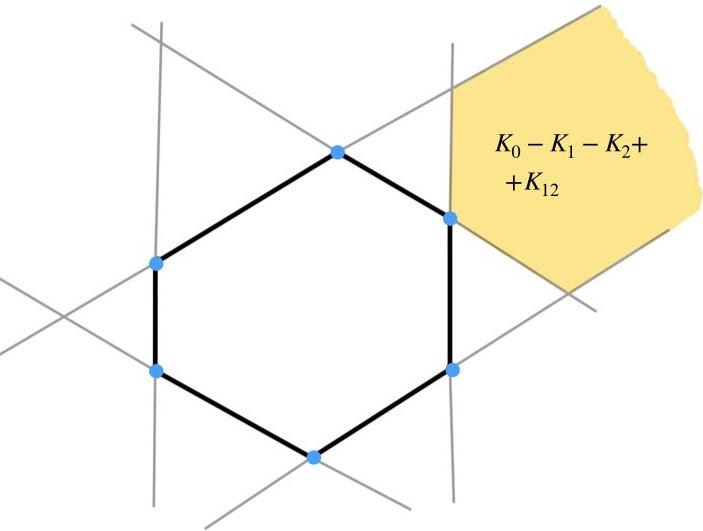}
\caption{Left: a complete simplicial fan in $V=\R^2$; we have labeled three cones in the fan. Right: a polygon normal to the fan and 
regions obtained by drawing the outward face cones; the function $k_\Delta$ in the shaded region is given by $K_0 -K_1-K_2+K_{12}$.}
\label{fig-k-Delta-SL3}
\end{figure}
  
One of our main results gives a sufficient condition for $k_\Delta(x)$ to be absolutely integrable on $V$ 
(see Theorem \ref{th-comb-conv} and also Theorem \ref{th-growth-conv}).

For $\sigma_2 \preceq \sigma_1$ in $\Sigma$, let 
\[ 
K_{\sigma_1, \sigma_2}(x) = \sum_{\sigma_2 \preceq \tau \preceq \sigma_1} (-1)^{\dim \tau} K_{\tau}(x). 
\] 
Also let polyhedral regions $R_{\sigma_1}^{\sigma_2}$ and $S_{\sigma_1}^{\sigma_2}$ be as in Definition \ref{SR-defn}, i.e., 
$S_{\sigma_1}^{\sigma_2}$ is the cone in $\operatorname{Span}(\sigma_1)$ defined via the edge vectors 
and facet normals of $\sigma_1$ and $\sigma_2$ as in Definition \ref{SR-defn}(a) (or equivalently \eqref{S12}) 
and $R_{\sigma_1}^{\sigma_2} = Q_{\sigma_1} + S_{\sigma_1}^{\sigma_2}$, where $Q_{\sigma_1}$ is the 
face of $\Delta$ associated with the cone $\sigma \in \Sigma$.

\begin{theorem*}[Convergence]
Assume that the fan $\Sigma$ above is acute (cf. Definition \ref{acute}).  With the notation as above, 
suppose for any $\sigma_2 \preceq \sigma_1$, the function $K_{\sigma_1, \sigma_2}$ is rapidly decreasing 
on the shifted neighborhoods of $S^{\sigma_1}_{\sigma_2}$. (See Theorem \ref{th-growth-conv} for the precise 
definition.) Then for any polytope $\Delta \in \cP(\Sigma)$, 
the integral  
\[ 
J_\Sigma(\Delta) = \int\limits_V k_\Delta(x) \, dx 
\] 
is absolutely convergent.
\end{theorem*}

We note that the conditions on $K_{\sigma_1, \sigma_2}$ in the theorem are ``local'' with respect to the fan $\Sigma$ 
in the sense that for each $\sigma \in \Sigma$ we only need to check a condition about $\sigma$ and 
the functions $K_\tau$, $\tau \preceq \sigma$ (and independent of other cones in the fan and their associated functions).

We also remark that the assumption that the fan $\Sigma$ is acute is crucial; without it, the convergence 
result may fail as we show in Example \ref{obtuse} where we consider obtuse cones.

Next, we discuss our result on polynomiality.  The set $\cP(\Sigma)$ of polytopes with normal fan $\Sigma$ is 
closed under multiplication by positive scalars and the Minkowski sum. Hence it makes sense to talk about 
a polynomial function on $\cP(\Sigma)$.  In fact, if $\Sigma(1)$ denotes the set of one dimensional cones in $\Sigma$, 
then a polytope $\Delta \in \cP(\Sigma)$ has a unique representation as 
\[ 
\Delta = \{ x \in V : \langle x, v_\rho \rangle \geqslant a_\rho, \forall \rho \in \Sigma(1)\},
\] 
where $v_\rho$ denotes the unit vector along $\rho$. The numbers $(a_\rho)_{\rho \in \Sigma(1)}$ are called 
the \emph{support numbers} of $\Delta$ and can be considered as coordinates on $\cP(\Sigma)$ 
(see Section \ref{subsec-normal-fan}).  Our main polynomiality result (cf. Theorem \ref{th-polynomiality})
states that the integral of $k_\Delta(x)$ depends polynomially on $\Delta \in \cP(\Sigma)$.  

\begin{theorem*}[Polynomiality]
The function $$\Delta \mapsto J_\Sigma(\Delta)$$ is a polynomial on $\cP(\Sigma)$,  
i.e., a polynomial in the support numbers of $\Delta$. 
\end{theorem*}

We remark that if all the $K_\sigma$ are identically equal to $1$, then $J_\Sigma(\Delta)$ coincides with the volume of $\Delta$. 
Thus our Polynomiality Theorem is a vast generalization of the classical fact that $\Delta \mapsto \vol(\Delta)$ is 
a polynomial function.  The assumption that each $K_\sigma$ is invariant in the direction of $\Span(\sigma)$ is 
obviously crucial in the proof of the Polynomiality Theorem.  For example, one can consider examples where $K_\sigma$ 
are not necessarily constant, but rather they are asymptotic to a constant in the direction of $\Span(\sigma)$.  Then 
one can still have convergence of $J_\Sigma(\Delta)$ by our more general Theorem \ref{th-comb-conv} on convergence, while 
$J_\Sigma(\Delta)$ would clearly not be a polynomial function. 

The strategy to prove our Convergence Theorem is as follows. Recall that the truncated function $k_\Delta(x)$ 
in \eqref{equ-intro-k-Delta} is defined as an alternating sum over various outward tangent cones $T^-_{\Delta, \sigma}$.  
In Lemma \ref{S-partition} we prove a certain double partition of the tangent cones $T^-_{\Delta, \sigma}$ 
in terms of certain natural subsets that appear, associated with pairs of cones in $\Sigma$, with the smaller cone being 
a face of $\sigma$ and the large one having $\sigma$ as a face.  In the double partition the inner partition essentially 
amounts to the special case where $\sigma$ is a full dimensional cone in $\Sigma$ while the outer partition amounts to 
a ``nearest face partition'' (cf. Section \ref{subsec-nearest-face}).  This allows us to repackage the various terms 
appearing in $k_\Delta$ into a sum of certain alternating sums $K_{\sigma_1,\sigma_2}$ associated with pairs of cones 
$\sigma_2 \preceq \sigma_1$ in $\Sigma$.  As a result, we reduce the question of the absolute convergence of the 
integral of $k_\Delta(x)$  over $V$ to that of absolute convergence of $K_{\sigma_1,\sigma_2}$ on the sets we obtain 
out of the partition.  This already gives our first, and more general, convergence result (cf. Theorem \ref{th-comb-conv}). 
We then go on to show that the two conditions in the above convergence theorem guarantee the convergence 
of the integral of $K_{\sigma_1,\sigma_2}$ on the required sets.

The regions we mentioned above seem to show up naturally in any treatment of convergence results, including Arthur's 
original proof of convergence of his (non-invariant) trace formula.  When $\sigma_1$ is full dimensional (corresponding 
to a maximal parabolic subgroup in Arthur's setting) and $\sigma_2$ is the origin, the region simply becomes 
the cone $\sigma_1$ shifted to the vertex of $\Delta$ corresponding to $\sigma_1$.  When $\sigma_2$ is a non-zero 
face of $\sigma_1$, then the region is again another cone shifted to the vertex.  This type of cone is precisely what 
Arthur has, for example, in \cite[Figure 8.5]{arthur-clay}.  For more general $\sigma_1$ the regions are a sum (as a set) 
of a compact face of $\Delta$ and a somewhat simpler cone.  For example, when $\dim V = 2$ these regions look like 
stripes.

A key step in the proof of the Polynomiality Theorem is Lemma \ref{lem-Gamma-Delta-sigma}, which is a statement 
concerning the polytope $\Delta$ and a cone $\sigma \in \Sigma$.  As far as we know this lemma is new and does not 
appear in Arthur's work.  It simplifies and streamlines some of the combinatorial arguments in \cite{arthur-duke, arthur-annals}.  
As a special case when $\Delta = \{0\}$, Lemma \ref{lem-Gamma-Delta-sigma} also implies 
the \emph{Langlands combinatorial lemma} (see \cite[Eqs. (8.10)--(8.11)]{arthur-clay} and \cite[Appendix]{gkm}).

When $\sigma$ is full dimensional and the vertex of $\Delta$ corresponding to $\sigma$ lies in $\sigma$, 
Lemma \ref{lem-Gamma-Delta-sigma} gives a decomposition of the characteristic function of the polytope 
$\Delta \cap \sigma$ in terms of certain cones with apexes at the vertices of this polytope. 
We obtain Lemma \ref{lem-Gamma-Delta-sigma} as a corollary of the Lawrence-Varchenko conical decomposition 
of a polytope (Theorem \ref{th-LV}). In fact, we need a more general version of this decomposition that applies to 
virtual polytopes (Theorem \ref{th-LV-virtual}). The arguments in this section rely on some key concepts and results 
from \cite{Khov-Pukh1, Khov-Pukh2} (which we review in Section \ref{sec-convex-chains}). 
We would like to point out that the proof of polynomiality shows that $J_\Sigma(\Delta)$ is a linear combination of 
volumes of certain virtual polytopes $\Gamma_{\Delta, \sigma}$, $\sigma \in \Sigma$.  

In the interest of making the connections with poset theory and M\"obius inversion more transparent, we show that 
the Langlands combinatorial lemma can be interpreted as a formula for the inverse of a certain element in 
the \emph{incidence algebra} of the poset of faces of a polyhedral cone (see Corollary \ref{cor-Langlands-comb-lem}).

Finally, we point out that Arthur's truncation parameter $T$ determines a polytope which is the convex hull of the 
Weyl group orbit of $T$.  Thus, Arthur's combinatorics is concerned with Weyl group invariant polytopes with a vertex 
in each Weyl chamber. In this paper we generalize the combinatorics to arbitrary simple polytopes.

It follows from the the proof of polynomiality that 
\[ 
J_\Sigma(0) = \sum_{\sigma_2 \preceq \sigma_2, \dim \sigma_1 = n} \int_{S_{\sigma_1}^{\sigma_2}} K_{\sigma_1, \sigma_2}(x) dx.
\] 
and that, in the case of a Weyl fan $\Sigma$ and a Weyl group invariant $\Delta$, 
the top degree homogeneous part of the polynomial $J_\Sigma(\Delta)$ is a constant multiple of the volume of $\Delta$.

\subsection*{The simplest example}
Let $\Sigma$ be the complete fan in $V = \R$ consisting of the origin $\sigma_0 = \{0\}$, the negative half-line $\sigma_-$ and the positive half-line $\sigma_+$. Let $\Delta \subset V^* \cong V = \R$ be the line segment $[a, b]$. Let $K_0$, $K_-$, $K_+$ be functions on $V$ corresponding to $\sigma_0$, $\sigma_-$, $\sigma_+$ respectively. From definition one computes that the truncated function $k_\Delta(x)$ is given by 
\[ 
k_\Delta(x) = \begin{cases} 
K_0 - K_-,  & x < a, \\ 
K_0, & a \leq x \leq b, \\ 
K_0 - K_+, & x > b. \\ 
\end{cases}
\] 
The assumption in Theorem \ref{th-polynomiality} that $K_\sigma$ is constant along $\textup{Span}(\sigma)$ implies that $K_-$ and $K_+$ 
are constant functions. Moreover, the condition that $\int\limits_{V} k_\Delta(x) dx$ is absolutely convergent means that 
$|K_0(x) - K_-|$ and $|K_0(x) - K_+|$ are integrable.  
We have 
\[ 
J_\Sigma(\Delta) = \int_\R k_\Delta(x) \, dx 
= \int_{-\infty}^0 (K_0(x) - K_-) \, dx 
+ \int_0^\infty (K_0(x) - K_+) \, dx 
+ \int_a^0 K_- \, dx
+ \int_0^b K_+ \, dx. 
\] 
Note that $\int_{-\infty}^0 (K_0(x) - K_-) \, dx$ and $\int_0^\infty (K_0(x) - K_+) \, dx$ are constants independent of $a$ and $b$ 
(whose sum we denote by the constant $c$) and $K_-$ and $K_+$ are constants.  
Hence, $J_\Sigma(\Delta) = c + (-a) \, K_- + b \, K_+$, a polynomial of degree $1$ in $a$ and $b$.

\begin{figure} [H]
\includegraphics[height=2cm]{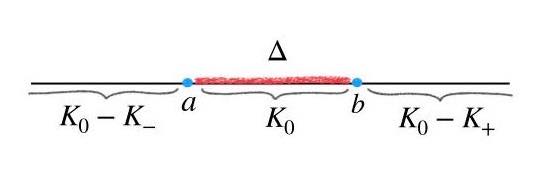}
\caption{}
\label{fig-line-segment}
\end{figure}
It is easy to see that if $K_+$ or $K_-$ is not a constant function, then the resulting $J_\Sigma(a, b)$ may not be 
a polynomial in $a$ and $b$. For example let $K_0 = K_+ = K_- = e^{x}$. Then $K_0 - K_+ = K_0 - K_+- = 0$, 
so the conditions of convergence are satisfied and in fact we have 
$J_\Sigma(a, b) = \int_a^b e^{x}dx = e^b - e^a,$ which is clearly not a polynomial in $a$ and $b$.

\subsection*{Another simple example: Brianchon-Gram} \label{simple-example}
If $K_\sigma \equiv 1$ for all the cones $\sigma$, then 
$k_\Delta$ becomes the characteristic function of the polytope $\Delta$ by the Brianchon-Gram Theorem 
(cf. Theorem \ref{th-BG}) and, as we mentioned earlier, our polynomiality result recovers the fact that the volume function 
$\Delta \mapsto \vol(\Delta)$ is a polynomial function. See Example \ref{ex-BG} for details.

\subsection*{Discrete versions of the results}

Replacing integration with summation, we obtain discrete versions of the above theorems.  
Given free abelian groups $M$ and $N$ of rank $n$ with a perfect $\Z$-pairing to identify them, 
we let $V = N_\R =  N \otimes_\Z \R$ and $V^* = M_\R = M \otimes_\Z \R$.  
Then $V$ and $V^*$ are a pair of dual $n$-dimensional real vector spaces as above.

We take a fan $\Sigma$ in $V=N_\R$ which is \emph{rational}, i.e. all its cones are generated by rational vectors 
with respect to $N \subset N_\R$.  We denote by $\cP(\Sigma, M)$ the set of polytopes with normal fan $\Sigma$ 
whose vertices lie in $M$.  It is closed under the Minkowski sum.  The discrete version of our convergence and polynomiality 
results (cf. Theorems \ref{th-growth-conv-discrete} and \ref{th-polynomiality-discrete}) are as follows. 

\begin{theorem*}[Convergence, discrete version]    
With notation as above, suppose for any $\sigma_2 \preceq \sigma_1$ in $\Sigma$, the function 
$K_{\sigma_1, \sigma_2}$ is rapidly decreasing on any shifted neighborhood of the cone $S^{\sigma_1}_{\sigma_2}$. Then for any polytope $\Delta \in \cP(\Sigma, M)$, 
the series 
\[ 
S_\Sigma(\Delta, M) = \sum_{x \in M} k_\Delta(x)dx 
\] 
is absolutely convergent.
\end{theorem*}

\begin{theorem*}[Polynomiality, discrete version]   
The function 
\[ 
\Delta \mapsto S_\Sigma(\Delta, M) 
\] 
is a polynomial on $\cP(\Sigma, M)$.
\end{theorem*}

We remark that if $K_\sigma \equiv 1$ for all nonzero cones $\sigma$ in $\Sigma$, then $S_\Sigma(\Delta, M)$ coincides 
with the number of lattice points in $\Delta$.  Thus the above theorem is a far reaching generalization 
of the classical fact that $\Delta \mapsto |\Delta \cap M|$ is a polynomial function (Ehrhart polynomial, see Theorem \ref{th-Ehrhart-poly}). 
It is interesting to explore whether some well-known polynomials appearing in combinatorics and representation theory, 
e.g. in the theory of symmetric polynomials, are instances of the polynomial $J_\Sigma(\Delta)$ or $S_\Sigma(\Delta, M)$.

\subsection*{Relation with toric varieties}

Convex lattice polytopes are well-studied in combinatorial algebraic geometry in relation to the geometry of toric varieties.  
In particular, there is a dictionary between algebraic geometric notions on toric varieties and convex geometric notions 
about lattice polytopes (see \cite{Fulton, cls}). For example the Riemann-Roch theorem for toric varieties gives 
beautiful formulas relating the number of lattice points in a polytope and its volume as well as volumes of its faces 
(see \cite{Khov-Pukh1, Khov-Pukh2} and \cite{Brion-Vergne}). 

A complete (rational) fan $\Sigma$ in $N_\R$ determines a complete toric variety $X_\Sigma$ over $\C$.  
It is an equivariant compactification of the algebraic torus $T_N \cong (\C^*)^n$. The polytope $\Delta \in \cP(\Sigma)$ 
determines a $T_N$-linearized ample line bundle $\mathcal{L}_\Delta$ on $X_\Sigma$ (see Section \ref{sec-toric-var}).

In Section \ref{subsec-equiv-Euler-char} we recall the well-known fact that the Brianchon-Gram Theorem can be regarded 
as the computation of the equivariant Euler characteristic of an ample toric line bundle. 

In Section \ref{sec-geo-comb-truncation} we give two interpretations of the function $k_\Delta(x)$ in terms of 
the toric variety $X_\Sigma$. In Section 
\ref{sec-measure} 
we interpret it as a ``truncated'' measure 
on the toric variety $X_\Sigma$ obtained by truncating a measure $\omega_0$ on the open torus orbit $X_0 \subset X_\Sigma$ 
using the measures $\omega_\sigma$ on the torus orbits $O_\sigma \subset X_\Sigma$ (at infinity).  
Each tangent cone $T^-_{\Delta, \sigma}$ determines an open neighborhood $\tilde{U}_{\Delta, \sigma}$ 
of the torus orbit closure $\overline{O}_\sigma$.  
The interpretation of the tangent cones $T^-_{\Delta, \sigma}$ as neighborhoods $\tilde{U}_{\Delta, \sigma}$ justifies 
the assumption that the fan is acute: under the acute assumption the neighborhood $\tilde{U}_{\Delta, \sigma}$ contains 
the orbit closure $\overline{O}_\sigma$.

In Section \ref{sec-Lefschetz} we observe that computation of equivariant Euler characteristic of an ample toric line bundle 
has uncanny resemblances to the definition of truncated function $k_\Delta(x)$ and hence to Arthur’s construction of 
the modified kernel $k^T(x)$. This leads to an interpretation of our combinatorial truncation as a Lefschetz number for computing 
the trace of the induced linear map of a morphism on the sheaf cohomologies of a toric variety.  

We point out that the similarity between the definition of $k^T(x)$ and the Brianchon-Gram theorem 
about polytopes has been observed by Bill Casselman in \cite{casselman}.

The polynomiality of the number of lattice points in a polytope is related to the polynomiality of the Euler characteristic 
which is an immediate consequence of the Riemann-Roch Theorem.  From this point of view, it is probable that 
our Polynomiality Theorem (Theorem \ref{th-polynomiality-discrete}) is a special case of a more general Riemann-Roch type theorem.

\subsection*{Relation with Arthur's work}

As we mentioned above, Arthur's development of his non-invariant trace formula is based on the two crucial results 
that the integral of $k^T(x) = k^T(x,f)$ on $G(\Q) \backslash G(\A)^1$ is absolutely convergent for $T \in \a_P^+$ 
sufficiently regular and $f \in C_c^\infty\left(G(\A)^1\right)$ and it is a polynomial of $T$.  We recall that 
\begin{equation} \label{equ-intro-k-T}
k^T(x,f) = \sum\limits_P (-1)^{\dim (A_P/A_G)} 
\sum\limits_{\delta \in P(\Q) \backslash G(\Q)} 
K_P\left(\delta x, \delta x \right) \, 
\widehat{\tau}_P\left(H_P(\delta x) - T \right). 
\end{equation} 
Here, the outer sum is over the standard parabolic subgroups $P$ of $G$ (containing a fixed minimal parabolic subgroup $P_0$), 
$H_P : G(\A) \longrightarrow \a_P$ is the Harish-Chandra map, and $\widehat{\tau}_P(\cdot)$ is the characteristic function of 
$\left\{t \in \a_P : \varpi(t) > 0, \varpi \in \widehat{\Delta}_P \right\}$, where $\widehat{\Delta}_P$ consists of weights $\varpi_\alpha$ 
for simple roots $\alpha$ corresponding to $P$.  (We refer to \cite{arthur-clay} for any unexplained notation.)

If we take $\Sigma$ to be the Weyl fan of the group $G$, then the parabolic subgroups of $G$ correspond to the 
cones in $\Sigma$ and the choice of a minimal parabolic subgroup corresponds to a choice of a full dimensional cone 
in $\Sigma$ with the standard parabolic subgroups corresponding to the faces of this full dimensional cone. The other 
cones in $\Sigma$ then correspond to the Weyl conjugates of the standard parabolic subgroups and this correspondence 
between cones and parabolic subgroups is order reversing with respect to inclusion.

The Weyl fan $\Sigma$ is a full dimensional, complete, simplicial fan that satisfies the acute assumption. 
The toric variety $X_\Sigma$ of the fan $\Sigma$ is a compactification of an algebraic torus by adding strata (orbits) at infinity for each cone $\sigma \in \Sigma$.  
The combinatorial truncation is an alternating sum of the $K_\sigma$ times the characteristic functions of certain neighborhoods of the strata at infinity.

Similarly, one has a compactification (Mumford's \emph{toroidal compactification}) of a reductive group $G$ by adding strata $X_P$ at infinity corresponding to rational parabolic subgroups $P$ (see \cite[Chapter IV, \S 1]{Mumford}). Arthur's truncation can be interpreted as an alternating sum of the $K_P$ times characteristic functions of certain neighborhoods of the strata $X_P$ at infinity.

The similarity between \eqref{equ-intro-k-Delta} and \eqref{equ-intro-k-T} is clear. 
This suggests that there is a corresponding family of functions $(K_\sigma)_{\sigma\in\Sigma}$ defined using the $K_P$ functions.  
We believe that our combinatorial arguments, or a variant thereof, can be used to give convergence and polynomiality results of Arthur as follows. 
One would use the analytic arguments already in Arthur's work to verify the assumptions of (the variant of) our convergence and polynomiality 
theorems. As a consequence, one would recover Arthur's results making the combinatorial/geometric ingredients of his 
proofs more streamlined, at least in our view.

We expect that one can extend the geometric interpretations of truncation (e.g. as a Lefschetz number) in 
Section \ref{sec-geo-comb-truncation} to Arthur's set up by replacing the toric variety $X_\Sigma$ by 
Mumford's toroidal compactification of a reductive algebraic group $G$.
We hope to write the details, using Reduction Theory, in our next paper on this 
subject.

\subsection*{Acknowledgements}

We would like to thank James Arthur, William Casselman, Mark Goresky, Thomas Hales, Erez Lapid, Werner M\"uller, and Tian An Wong 
for useful correspondences and conversations. The second author was partially supported by NSF grant DMS-1601303 and a Simons Collaboration Grant. We also thank the referee for a careful reading of the manuscript. 

\section{Preliminaries} \label{sec-prelim}

We review some basic notions from the theory of polyhedral cones and fix some notations along the way.  
We refer to \cite[\S 1.2]{cls} for further details.

\subsection{Cones and fans}   \label{subsec-fan}

Let $V$ be a finite dimensional real vector space of dimension $n$ and let $V^*$ denote its dual. 
Recall that a \textit{(closed convex) polyhedral cone} in $V$ is a set of the form 
\[ 
\sigma = \operatorname{Cone}(W) = \left\{ \sum\limits_{w \in W} a_w w : a_w \geqslant 0 \right\} \subseteq V
\] 
with $W$ a finite subset of $V$. 
Equivalently, there is a finite subset $B$ of $V^*$ such that  
\[ 
\sigma = \bigcap\limits_{b \in B} \left\{x \in V : b(x) \geqslant 0 \right\}. 
\] 
We say that $\sigma$ is \textit{generated} by $W$.  
Also, we write $\operatorname{Cone}(\emptyset) = \{0\}$.  
The dimension of $\sigma$ is the dimension of its linear span. 
The dual cone $\sigma^\vee$ is defined as 
\[ 
\sigma^\vee := \left\{ y \in V^* : y(x) \geqslant 0 \mbox{ for all } x \in \sigma\right\}.   
\] 
Dual cones enjoy the property that if $\sigma$ is a polyhedral cone in $V$, 
then $\sigma^\vee$ is a polyhedral cone in $V^*$ and $\sigma^{\vee\vee} = \sigma$.

For a face $\tau$ of $\sigma$ (denoted $\tau \preceq \sigma$) define its dual face 
\begin{eqnarray*} 
\tau^* &=& \left\{y \in \sigma^\vee : y(x) = 0 \mbox{ for all } x\in\tau \right\} \\
	&=& \sigma^\vee \cap \tau^\perp. 
\end{eqnarray*}
Then $\tau^*$ is a face of $\sigma^\vee$, $\tau^{**} = \tau$, 
$\tau \leftrightarrow \tau^*$ is an inclusion-reversing bijection 
between faces of $\sigma$ and those of $\sigma^\vee$, and $\dim \tau + \dim \tau^* = n$.  
One dimensional cones, i.e. half-lines,  are called \textit{rays}.  
A face $\tau$ of $\sigma$ is called a \textit{facet} if 
$\dim \tau = \dim \sigma - 1$ and its linear span is referred to as a \textit{wall} of $\sigma$. 
An \textit{edge} is a face of dimension 1.

Define the \textit{relative interior} $\sigma^\circ$ of $\sigma$ to be 
the interior of $\sigma$ in its span. One then checks that $x \in \sigma^\circ$ if and only if 
$y(x) > 0$ for all $y \in \sigma^\vee \setminus \sigma^\perp$. 
A polyhedral cone $\sigma$ in $V$ is \textit{strongly convex} if the origin is a face. 
This is the case 
if and only if $\sigma$ contains no positive dimensional subspace of $V$ if and only if 
$\sigma \cap (-\sigma) = \{0\}$ if and only if $\dim \sigma^\vee = n$. 
A strongly convex polyhedral cone $\sigma \subseteq V$ is called \textit{simplicial} 
if it is generated by linearly independent vectors.
We note that the dual of a simplicial cone of maximal dimension is again simplicial.

For $y \in V^*$, we set   
\[ 
H_y := \left\{x \in V : y(x) = 0 \right\} \subseteq V    
\] 
and define the closed, resp. open, spaces 
\[ 
H^+_y := \left\{x \in V : y(x) \geqslant 0 \right\} \subseteq V   
\quad \mbox{ and } \quad 
H^-_y := \left\{x \in V : y(x) < 0 \right\} \subseteq V.  
\] 
When $y \not= 0,$ $H_y$ is a hyperplane and $H^+_y$ and $H^-_y$ are half-spaces in $V$. 
When $y=0$, we have $H_y = H^+_y = V$ while $H_y^-$ is empty. 
If $\sigma \subseteq H^+_y$ for $y \not= 0$, we say 
$H_y$ is a \textit{supporting hyperplane} and $H^+_y$, resp. $H^-_y$, is an \textit{inward, resp. outward, 
supporting half-space} of $\sigma$.  (When $y=0$, we automatically have $\sigma \subseteq H^+_0 = H_0 = V$.) 
Note that $H_y$ is a supporting hyperplane of $\sigma$ if and only if $y \in \sigma^\vee \setminus \{0\}$. 
If $y_1, y_2, \dots, y_r$ generate $\sigma^\vee,$ then $\sigma = H^+_{y_1} \cap \cdots \cap H^+_{y_r}$.  
Thus, every polyhedral cone is an intersection of finitely many closed half-spaces. 

A \textit{fan} $\Sigma$ in $V$ is a finite collection of cones $\sigma \subseteq V$ satisfying the following 
three properties: 
(a) every $\sigma \in \Sigma$ is a strongly convex polyhedral cone,     
(b) for all $\sigma \in \Sigma,$ each face of $\sigma$ also belongs to $\Sigma$, and  
(c) for all $\sigma_1,  \sigma_2 \in \Sigma$, the intersection $\sigma_1 \cap \sigma_2$ is a face of each. 
The set of $r$-dimensional cones of $\Sigma$ is denoted by $\Sigma(r)$. 
The \textit{support} of $\Sigma$ is defined by 
\[ 
|\Sigma| := \bigcup_{\sigma \in \Sigma} \, \sigma \subseteq V.  
\] 
If $|\Sigma| = V$, then $\Sigma$ is called a \textit{complete fan}. 
A \textit{simplicial fan} is a fan all whose cones are simplicial.
Every fan can be refined into a simplicial fan.

Finally, for $\sigma \in \Sigma$ we let $\Sigma / \sigma$ denote the fan in $V / \Span(\sigma)$ consisting of all the images of the cones $\sigma' \succeq \sigma$.
If we fix an inner product on $V$ then $V / \Span(\sigma)$ can be identified with $\sigma^\perp$ and  
$\Sigma / \sigma$ consists of projections of $\sigma' \succeq \sigma$ onto $\sigma^\perp$.

\subsection{Polytopes}    \label{subsec-polytope}

A \textit{polytope} is a set in $V^*$ of the form 
\[ 
P = \operatorname{Conv}(S) = \left\{ \sum\limits_{u \in S} \lambda_u u : \lambda_u \geqslant 0, \sum_{u \in S} \lambda_u 
= 1 \right\}, 
\] 
where $S$ is a finite subset of $V^*$.  We say $P$ is the \textit{convex hull} of $S$.  The dimension, 
$\dim P$, of a polytope $P$ is the dimension of the smallest affine subspace of $V^*$ containing $P$. 
Given $x \in V \setminus \{0\}$ and $r \in \R$ we have the affine hyperplane 
\[ 
H_{x,r} := \left\{ y \in V^* : y(x) = r \right\}
\]
and the closed, resp. open, half-spaces 
\[ 
H_{x,r}^{+} := \left\{ y \in V^* : y(x) \geqslant r \right\} 
\quad \mbox{ and } \quad 
H_{x,r}^{-} := \left\{ y \in V^* : y(x) < r \right\}.  
\] 
A subset $Q \subseteq P$ is a \textit{face} of $P$, denoted by $Q \preceq P$, if there is  
$x \in V \setminus \{0\}$ 
and there is $r \in \R$ with 
\[ 
Q = H_{x,r} \cap P \quad\mbox{ and }\quad P \subseteq H_{x,r}^+. 
\] 
We then say that $H_{x,r}$ is a \textit{supporting affine hyperplane}.  
The polytope $P$ is regarded as a face of itself 
and faces of $P$ of dimensions $0$, $1$, and $(\dim P - 1)$ 
are called 
\textit{vertices}, 
\textit{edges}, 
and 
\textit{facets}, 
respectively.

A polytope $P \subseteq V^*$ can be written as a finite intersection of closed half-spaces and an 
intersection 
\[ 
P = \bigcap\limits_{i=1}^s H_{x_i,r_i}^+ 
\] 
is a polytope provided that it is bounded. In general, an intersection of finitely many closed half-spaces is called 
a \emph{polyhedron} and could be unbounded. 
When $\dim P = \dim V^*$ (i.e., full dimensional polytope) 
for each facet $F$ we have a \textit{unique} 
supporting affine hyperplane and corresponding closed half-space given by 
\[ 
H_F = H_{u_F^+,a_F} = 
\left\{ y \in V^* : y(u^+_F) = a_F \right\} 
\quad \mbox{ and } \quad 
H_F^{+} = H^+_{u_F^+,a_F} = 
\left\{ y \in V^* : y(u^+_F) \geqslant a_F \right\},   
\] 
where $(u^+_F, a_F) \in V \times \R$ is unique up to multiplication by a positive real number. 
We call $u^+_F$ an \textit{inward-pointing facet normal} of the facet $F$. Hence, 
\begin{equation}\label{aF} 
P = \bigcap\limits_{F \text{ facet }} H_F^+ 
= 
\left\{ y \in V^* : y(u^+_F) \geqslant a_F \mbox{ for all proper facets } F \prec P
\right\}. 
\end{equation}
This is the so-called facet representation of $P$. We also have a similar representation with 
\textit{outward-pointing facet normals} $u^-_F = - u^+_F$.  When the facet normals $u^{\pm}_F$ 
are assumed to be unit vectors, we may call the $a_F$ the \textit{support numbers} of $P$.

Let $Q$ be a face of $P$ and define the \textit{inward, resp. outward, tangent cone} 
$T^+_{P, Q}$, resp. $T^-_{P, Q}$, via  
\begin{eqnarray} \label{inward-face-cone}
T^+_{P, Q} &:=& \left\{ y \in V^* : y(u^+_F) \geqslant a_F \mbox{ for all facets } F \supset Q \right\}, 
\\
\text{resp. } \label{outward-face-cone}
T^-_{P, Q} &:=& \left\{ y \in V^* : y(u^+_F) < a_F \mbox{ for all facets } F \supset Q \right\} \\ 
&=& \left\{ y \in V^* : y(u^-_F) > a_F \mbox{ for all facets } F \supset Q \right\}.  
\nonumber
\end{eqnarray}
See Figures \ref{fig-TC1} and \ref{fig-TC2} for illustrations of inward and outward tangent cones 
of a quadrilateral at a vertex and at an edge, respectively. 

\begin{figure}    
\includegraphics[height=5cm]{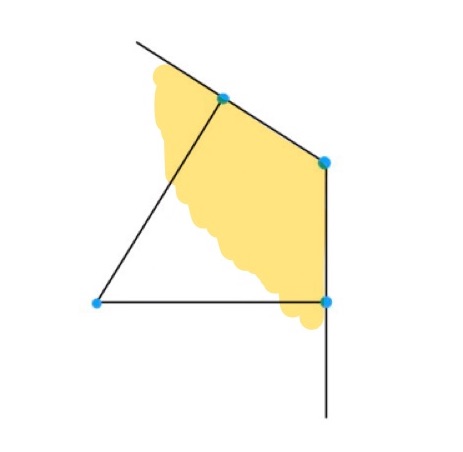}
\includegraphics[height=5cm]{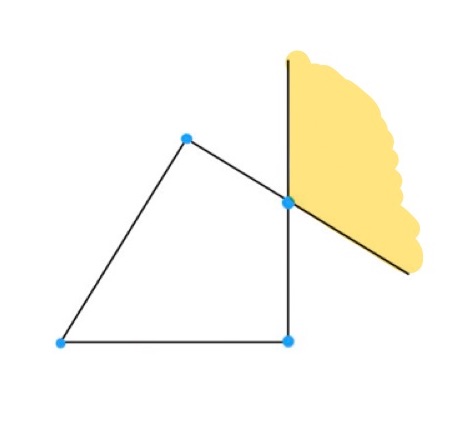}
\caption{Inward and outward tangent cones at a vertex (left inward, right outward)}
\label{fig-TC1}
\end{figure}

\begin{figure}  
\includegraphics[height=5cm]{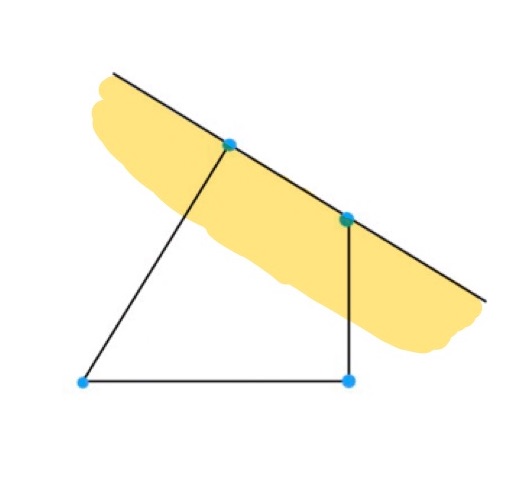}
\includegraphics[height=5cm]{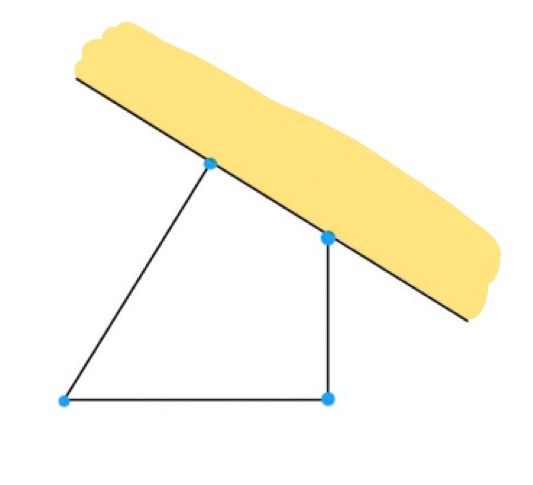}
\caption{Inward and outward tangent cones at an edge (left inward, right outward)}
\label{fig-TC2}
\end{figure}

A polytope $P \subseteq V^*$ of dimension $d$ is called a \textit{$d$-simplex} (or just a \textit{simplex}) 
if it has $d+1$ vertices, \textit{simplicial} if every facet is a simplex, and \textit{simple} if every vertex 
is the intersection of precisely $d$ facets.  

Given a polytope $P = \operatorname{Conv}(S)$, its multiple $rP = \operatorname{Conv}(rS)$ 
is also a polytope for any $r \geqslant 0$.  The Minkowski sum $P_1 + P_2 =  \{y_1+y_2 : y_i \in P_i\}$ of two polytopes 
$P_1 = \operatorname{Conv}(S_1)$ and $P_2 = \operatorname{Conv}(S_2)$ is again a 
polytope and we have the distributive law $rP + sP = (r+s) P$. 
The set $\mathcal{P}(V^*)$ of polytopes in $V^*$ together with Minkowski sum is a cancellative semigroup. 
The following theorem is originally due to Minkowski. 
\begin{theorem}[Volume polynomial]   \label{th-vol-poly}
The map $P \mapsto \vol_n(P)$ is a polynomial function on $\cP(V^*)$ in the following sense: let $P_1, \ldots, P_r$ 
be polytopes in $V^*$. For any $\lambda_1, \ldots, \lambda_r \geqslant 0$ we can form the polytope $\sum_i \lambda_i P_i$.  
Then the function $(\lambda_1, \ldots, \lambda_r) \mapsto \vol_n(\sum_i \lambda_i P_i)$ is the restriction of a homogeneous 
polynomial on $\R^r$ to the positive orthant $\R_{\geqslant 0}^r$.
\end{theorem}

There is also a discrete analogue of Theorem \ref{th-vol-poly} which is harder and more subtle to prove. 
It is a generalization of the notion of Ehrhart polynomial. Let  $M \cong \Z^n$ be a 
full rank lattice in $V^* \cong \R^n$. Let $\cP(M)$ denote the collection of \emph{lattice polytopes with respect to $M$}, 
that is, all polytopes in $V^*$ whose vertices belong to $M$. The set $\cP(M)$ is closed under the Minkowski sum and 
multiplication by positive integers. 
\begin{theorem} [Ehrhart polynomial]  \label{th-Ehrhart-poly}
The map $P  \mapsto |P \cap M|$ is a polynomial map on $\cP(M)$. 
\end{theorem} 

More generally, the polynomiality property holds for any \emph{valuation} (also called \emph{finitely additive measure}).  
A function $\Phi: \cP(M) \to \R_{\geqslant 0}$ 
is called a \emph{valuation} if for all $P_1, P_2 \in \cP(M)$ the following hold: 
\begin{itemize} 
\item[(1)] 
$\Phi$ is monotone with respect to inclusion, i.e. $\Phi(P_1) \leq \Phi(P_2)$ provided that 
$P_1 \subset P_2$;  
\item[(2)] 
$\Phi(P_1 \cup P_2) = \Phi(P_1) + \Phi(P_2) - \Phi(P_1 \cap P_2)$. 
\end{itemize} 
We say $\Phi$ is \emph{$\Z^n$-invariant} 
if $\Phi(m+P) = \Phi(P)$ for all $P \in \cP(M)$ and $m \in M$. 
The following is a beautiful result of P. McMullen \cite{McMullen}. It generalizes Theorem \ref{th-Ehrhart-poly}.
\begin{theorem}  \label{th-McMullen}
Let $\Phi$ be a $\Z^n$-invariant valuation on $\cP(M)$. Then $\Phi$ is a polynomial function. 
\end{theorem}
\begin{remark}   \label{rem-McMullen}
When $\Phi(P) = |P \cap M|$ one recovers Theorem \ref{th-Ehrhart-poly}.  
Fix a point $a \in V^*$.  Theorem \ref{th-McMullen}, in particular, implies that the function defined by
$\Phi_a(P) = |P \cap (a+M)|$ is also a polynomial.  
\end{remark} 

\subsection{Normal fan}      \label{subsec-normal-fan}

For $Q \preceq P$, let 
\[ 
\sigma_Q := \operatorname{Cone}\left( u^-_F : \text{ facets } F \supset Q \right). 
\] 
Given a full dimensional 
polytope $P \subseteq V^*$, the cones $\sigma_Q$ fit together to form the \textit{normal fan} of $P$ 
in $V$ given by 
\[ 
\Sigma_P = \left\{ \sigma_Q : Q \preceq  P \right\}. 
\] 
Note that we have used \emph{outward} facet normals $u^-_F$ to define the normal fan. 
(Some authors use inward facet normals $u^+_F$ instead.) 

\begin{figure}
\includegraphics[height=6cm]{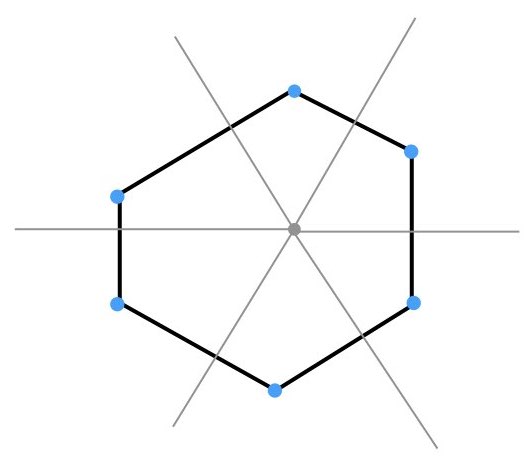}
\caption{A polygon and its normal fan. Note that in our convention we use \emph{outward} facet normals to define the cones in the normal fan.}
\label{fig-normal-fan}
\end{figure}

Let $\mathcal{P}(\Sigma)$ be the collection of all convex polytopes whose normal fan is $\Sigma$. 
This set is closed under Minkowski sum of polytopes and multiplication by positive scalars. 
For $P \in \mathcal{P}(\Sigma)$ we have an inclusion-reversing bijection 
\begin{equation}\label{face-cone-bijection} 
Q = Q_\sigma \longleftrightarrow \sigma = \sigma_Q 
\end{equation}   
between the set of faces of $P$ and the set of cones in the normal fan $\Sigma$. 
In particular, the facets $F$ of $P$ correspond to rays $\rho \in \Sigma(1)$.  For a ray $\rho \in \Sigma(1)$ we set 
$a_\rho = a_F$, where $F$ is the facet corresponding to $\rho$ and $a_F$ are the support numbers of $P$, 
see \eqref{aF}. 
The map $P \mapsto (a_\rho)_{\rho \in \Sigma(1)}$ gives an embedding of $\mathcal{P}(\Sigma)$ into 
$\R^s$, where $s=|\Sigma(1)|$.  The image is a full dimensional (open) convex polyhedral cone.

Let $P$ be a full dimensional polytope with normal fan $\Sigma_P$. Let $Q \preceq P$ be a face with corresponding cone $\sigma_Q \in \Sigma_P$. 
Then the normal fan $\Sigma_Q$ (of the polytope $Q$) is the fan $\Sigma_P / \sigma_Q$ (defined at the end of Section \ref{subsec-fan}). It consists of the 
images of the cones $\sigma' \succeq \sigma_Q$ in the quotient vector space $V / \Span(\sigma_Q)$. 

\subsection{Nearest face partition}   \label{subsec-nearest-face}
Fix an inner product $\langle \cdot, \cdot \rangle$ on $V$.  Let $P \subset V$ be a convex polyhedron.  
To $P$ we can associate a partition of $V$ into polyhedral regions $V_P^Q$ as follows. For each face $Q \preceq P$ let 
$$V_P^Q  = \left\{ x \in V : \textup{the minimum distance from } x \textup{ to } P \textup{ is attained at a point in the relative interior of } Q \right\}.$$
The following is straightforward to verify.
\begin{proposition}   \label{prop-nearest-face} 
\begin{itemize}
\item[(1)] For each face $Q \preceq P$, the set $V_P^Q$ is a polyhedron. 
\item[(2)] We have a disjoint union \[ V = \bigsqcup\limits_{Q \preceq P} V_P^Q. \]
\end{itemize}
\end{proposition}
We can modify the $V_P^Q$ to obtain a slightly different partition $\left\{W_P^Q : Q \preceq P \right\}$. For each face $Q \preceq P$ let
$$W_P^Q = \overline{V_P^Q} \setminus (\bigcup_{Q' \succneqq Q} \overline{V_P^{Q'}}),$$ where $\overline{V_P^Q}$ denotes the closure of $V_P^Q$.
The polyhedra $W_P^Q$ and $V_P^Q$ have the same relative interior but they are different on the boundary. 

We refer to both $\left\{V_P^Q : Q \preceq P \right\}$ and $\left\{W_P^Q : Q \preceq P \right\}$ as the \emph{nearest face partition} of $V$ with respect to the polyhedron $P$. We note that if in particular $P = \sigma$ is a cone (with apex at the origin) then the closure of the parts in the partition 
with respect to $\sigma$ in fact form a complete fan in $V$. In practice we will also use the nearest face partition to partition 
a polyhedron inside $V$. 

\begin{figure}
\includegraphics[height=6.5cm]{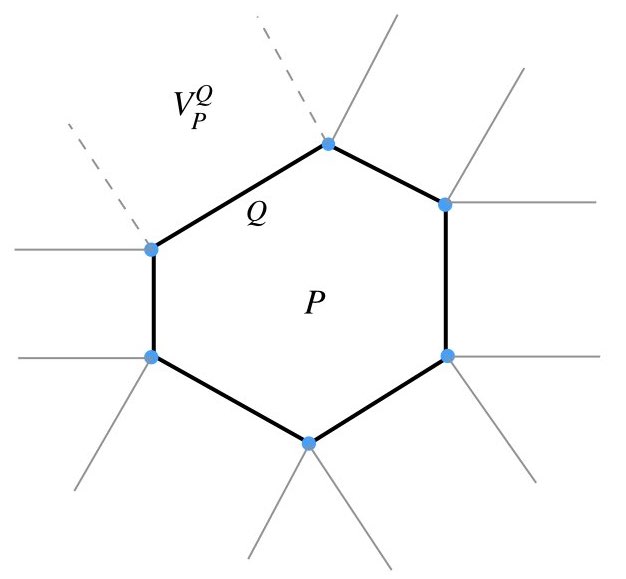}
\includegraphics[height=6.5cm]{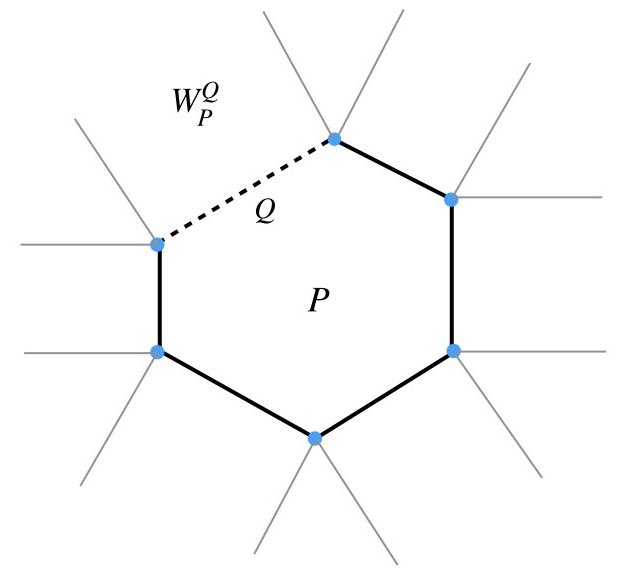}
\caption{Nearest face partition for a polygon illustrating polyhedral regions $V_P^Q$ and $W_P^Q$ corresponding to an edge $Q$.}
\label{fig-nearest-face}
\end{figure}

\subsection{Conical decomposition theorems}     \label{sec-conical-decompose}

We end this section by recalling two beautiful formulas which represent the characteristic function of a polytope as an alternating sum of 
characteristic functions of cones. For a nice overview of these decompositions and related topics we refer the reader to \cite{Beck-Haase-Sottile}.

\subsubsection*{Brianchon-Gram theorem}

The first conical decomposition theorem we discuss 
is the Brianchon-Gram theorem. 
It is named after C. J. Brianchon and J. P. Gram who independently proved the $n = 3$ case in 1837 and 1874, respectively (\cite{Brianchon} and \cite{Gram}). 
It is the mother of all cone decompositions!  See \cite[Section 1.1]{Haase} and references therein. Also see \cite{Agapito}. 
\begin{theorem}[Brianchon-Gram]   \label{th-BG}
Let $P$ be a polytope in $V^*$. We have the following equality, where ${\mathbf 1}$ denotes characteristic function 
\begin{equation}   \label{equ-BG}
{{\mathbf 1}}_P = \sum_{Q \preceq P} (-1)^{\dim Q} {{\mathbf 1}}_{T^+_{P, Q}}.
\end{equation}
\end{theorem}

\begin{proof} 
For a point $y \in P$, the right hand side computes the Euler characteristic of $P$ and hence is equal to $1$ since $P$ is contractible. 
For $y \notin P$, we have to subtract the Euler characteristic of the sub-complex that is visible from $y$ which is again contractible.  
\end{proof}

Alternatively, one can formulate Brianchon-Gram in terms of outward looking tangent cones. 
\begin{theorem}[Brianchon-Gram, alternative version]   \label{th-BG-alt}
Let $P$ be a polytope in $V^*$.  We have the following equality:
\begin{equation} \label{equ-BG-alt}
{{\mathbf 1}}_{P} = \sum_{Q \preceq P} (-1)^{n - \dim Q} {{\mathbf 1}}_{T^-_{P, Q}}. 
\end{equation}
\end{theorem}
The above version of the Brianchon-Gram formula looks similar to Arthur's definition of the modified kernel $k^T(x)$, 
as was observed in \cite{casselman}.  
See Figures \ref{fig-BG-v1} and \ref{fig-BG-v2} for illustrations of \eqref{equ-BG} and \eqref{equ-BG-alt}. 
\begin{figure}[H] 
\includegraphics[height=6cm]{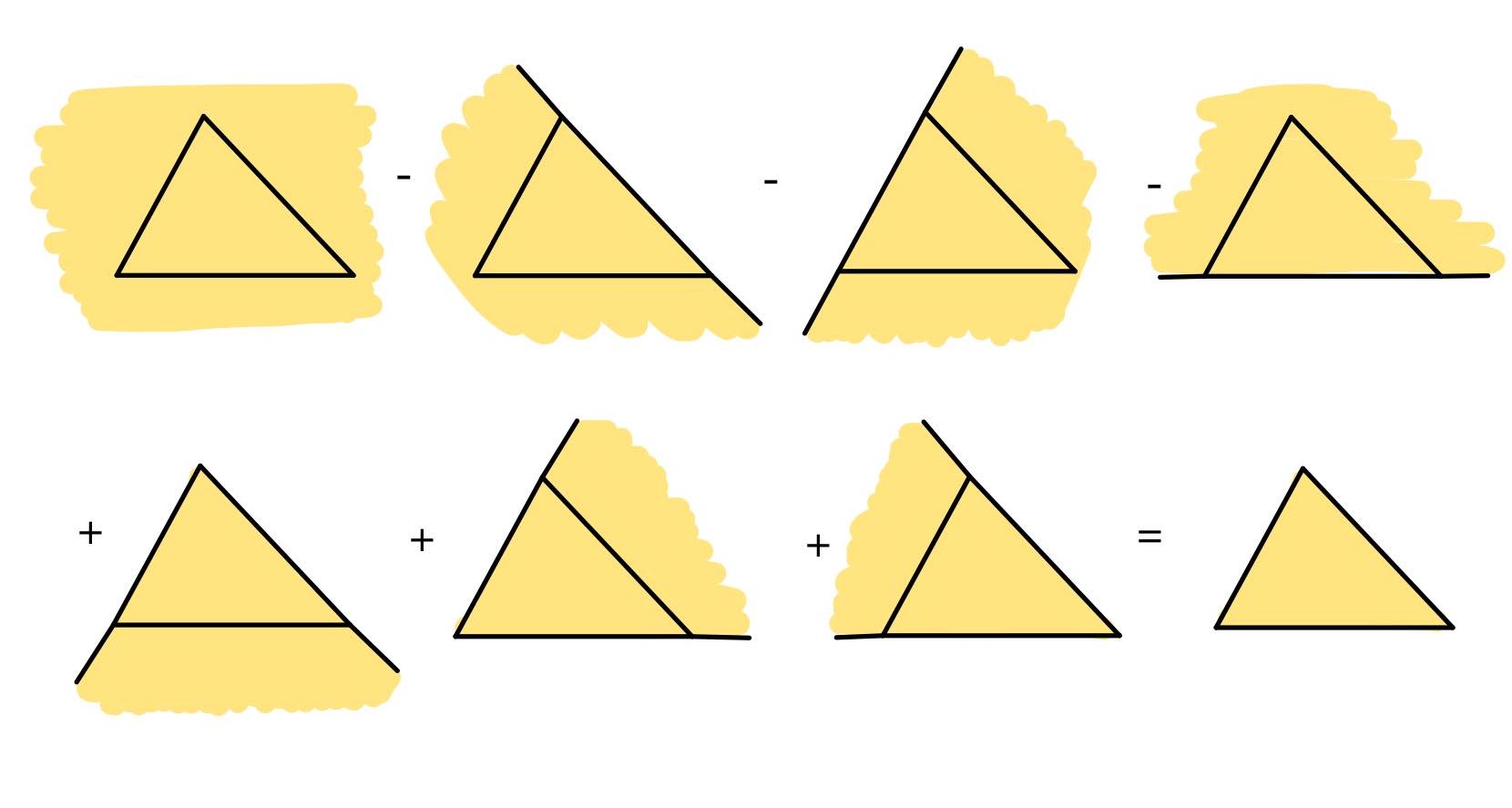}
\caption{Illustration of the Brianchon-Gram theorem (inward looking tangent cones) for a triangle}
\label{fig-BG-v1}
\end{figure}
\begin{figure}[H]  
\includegraphics[height=6cm]{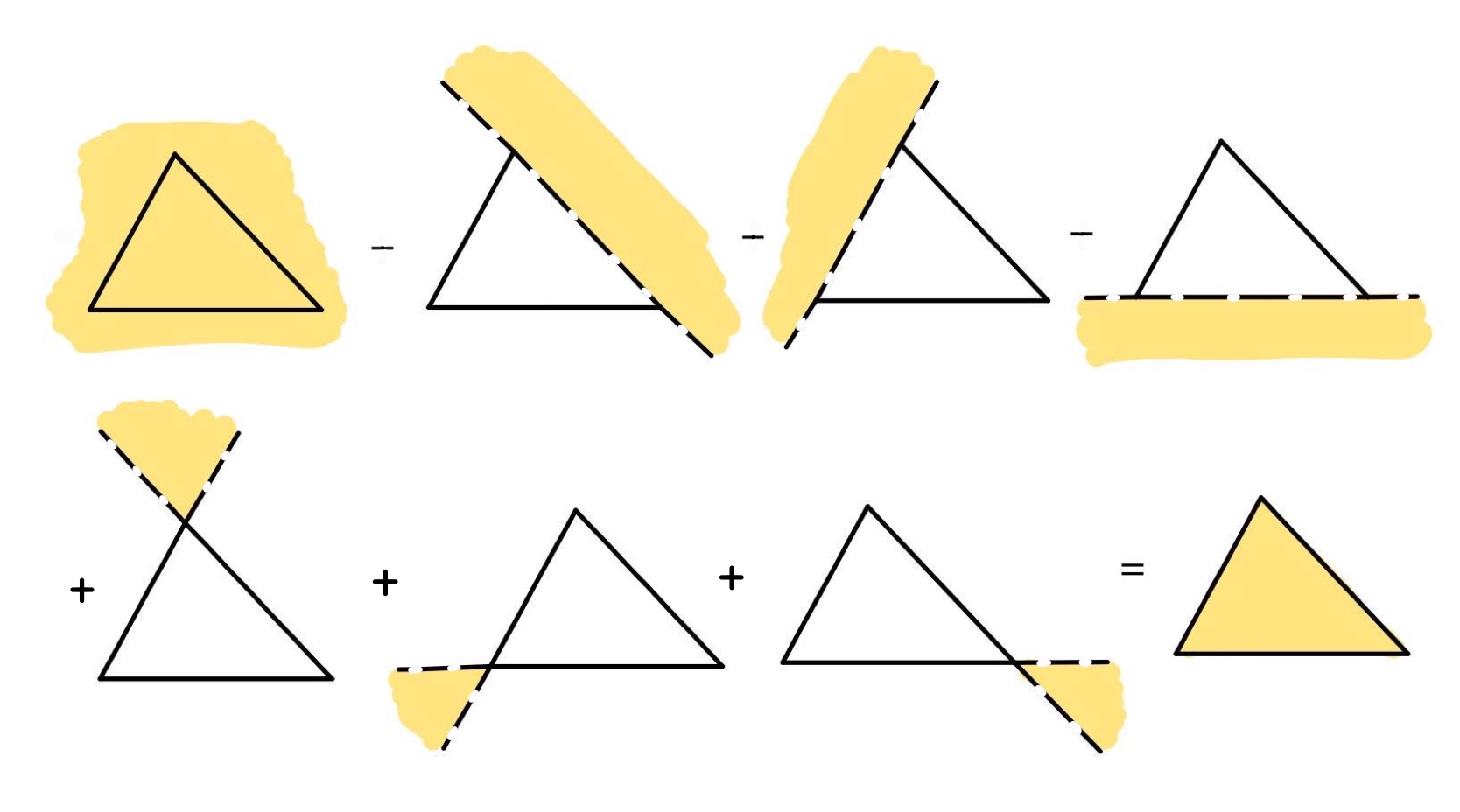}
\caption{Illustration of the Brianchon-Gram theorem (alternative version, outward looking tangent cones) for a triangle}
\label{fig-BG-v2}
\end{figure}

\subsubsection*{Lawrence-Varchenko theorem}

The second conical decomposition due to Lawrence \cite{Lawrence} and Varchenko \cite{Varchenko} represents the 
characteristic function of a polytope as an alternating sum of characteristic functions of certain cones associated to vertices 
of the polytope. It is a predecessor to the work of Khovanskii-Pukhlikov \cite{Khov-Pukh1, Khov-Pukh2} and 
Brion-Vergne \cite{Brion-Vergne}. It is related to Morse theory on polytopes as well as equivariant cohomology of toric varieties.  
The Lawrence-Varchenko theorem follows immediately from Khovanskii-Pukhlikov results as well, see \cite[Section 3.2]{Khov-Pukh2}.

Let $P \subset V$ be a simple polytope and let $v$ be a vertex of $P$. Let $w_1, \ldots w_r$ be edge vectors of $P$ 
at the vertex $v$. Fix a dual vector $\xi \in V^*$ such that $\langle w_i, \xi \rangle \neq 0$, for all $i$. We define vectors 
$w'_1, \ldots, w'_r$ as follows: 
\begin{equation*}  
w'_i = 
\begin{cases} 
 w_i, & \mbox{ if } \langle w_i, \xi \rangle > 0, \\ 
- w_i, & \mbox{ otherwise. } 
\end{cases} 
\end{equation*} 
Finally, define the \emph{polarized tangent cone} $T^\xi_{P, v}$ with apex at $v$ by 
\[ T^\xi_{P, v} = \left\{ 
\sum_{i=1}^r \lambda_i w'_i : 
\begin{array}{ll} 
\lambda_i \geqslant 0 \textup{ if } w'_i = w_i \\ 
\lambda_i > 0  \textup{ if } w'_i = -w_i 
\end{array} 
\right\}. \]

\begin{theorem}[Lawrence-Varchenko]   \label{th-LV}
With notation as above, we have the following:
\begin{equation} \label{LV-eqn}
{\mathbf 1}_P = \sum_{v} (-1)^{n_v} {\mathbf 1}_{T^\xi_{P, v}}, 
\end{equation} 
where the sum is over all the vertices $v$ of $P$, and $n_v = |\{i : w'_i = -w_i \}|$. 
\end{theorem}
See Figure \ref{fig-LV} for an illustration of \eqref{LV-eqn}

\begin{figure}[H] 
\includegraphics[height=8cm]{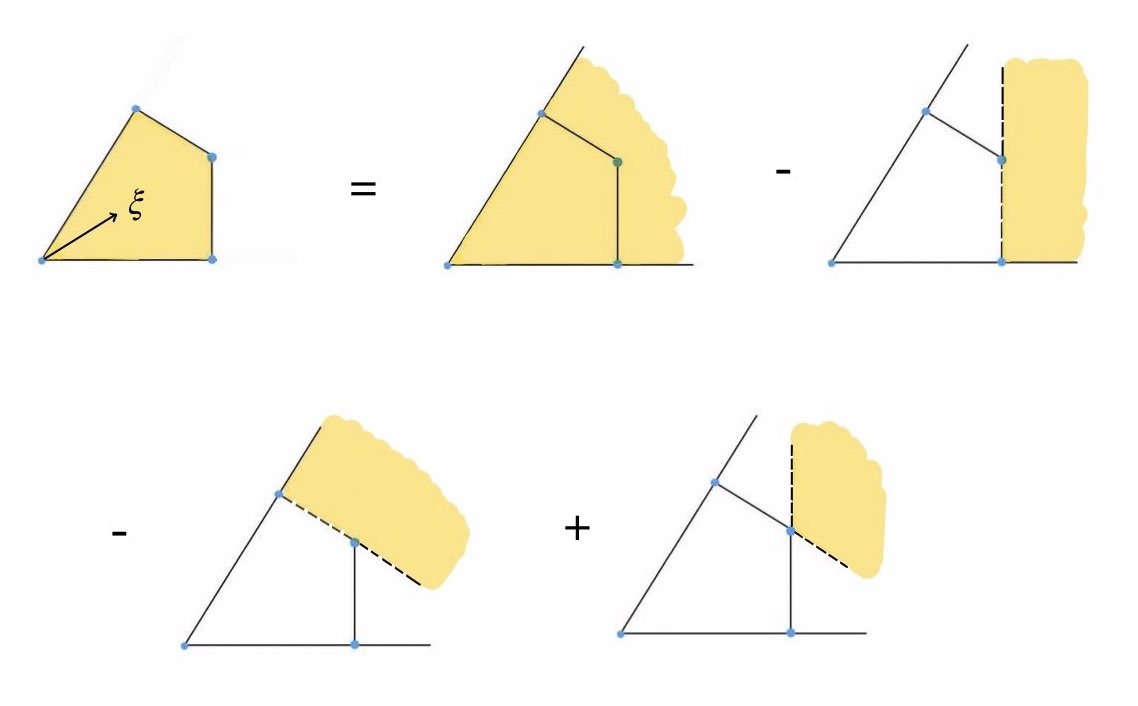}
\caption{Illustration of the Lawrence-Varchenko theorem for a quadrangle}
\label{fig-LV}
\end{figure}

\subsection{Khovanskii-Pukhlikov virtual polytopes and convex chains}     \label{sec-convex-chains}

This is a summary of some ideas and results from \cite{Khov-Pukh1, Khov-Pukh2} that we will need later.
As before $V \cong \R^n$ denotes an $n$-dimensional real vector space. 

Recall that $\cP(V^*)$ denotes the set of polytopes in the dual space $V^*$. The set $\mathcal{P}(V^*)$ 
is equipped with the operations of Minkowski sum and multiplication by positive scalars. One knows that 
$\cP(V^*)$ together with the Minkowski sum is a cancellative semigroup and hence it can be  extended to 
a real vector space $\mathcal{V}(V^*)$ consisting of formal differences $P_1 - P_2$, $P_i \in \mathcal{P}(V^*)$, 
where for polytopes $P_1, P_2, P'_1, P'_2$, we have $P_1 - P_2 = P'_1 - P'_2$ if and only if $P_1 +  P'_2 = P'_1 + P_2$.

\begin{definition}[Virtual polytope]
The elements of $\mathcal{V}(V^*)$ are called \textit{virtual polytopes} (see \cite{Khov-Pukh1}).
\end{definition}
We note that $\cV(V^*)$ is an infinite dimensional vector space.

Let $\Sigma$ be a complete fan in $V$. Recall that $\cP(\Sigma)$ denotes the set of all polytopes in $V^*$ 
whose normal fan is $\Sigma$. The set $\cP(\Sigma)$ is closed under Minkowski sum and multiplication by positive scalars.  
We denote by $\mathcal{V}(\Sigma)$ the subspace of $\mathcal{V}(V^*)$ spanned by $\mathcal{P}(\Sigma)$.  
The elements of $\mathcal{V}(\Sigma)$ are called \textit{virtual polytopes with normal fan $\Sigma$}.  
Generalizing the facet representation of a polytope $P \in \cP(\Sigma)$, i.e. representation as an intersection of half-spaces 
$H^+_{u^+_\rho, a_\rho}$, $\rho \in \Sigma(1)$, each virtual polytope in $\cV(\Sigma)$ is represented by a collection of 
oriented hyperplanes $H_{u_\rho, a_\rho}$, $\rho \in \Sigma(1)$. Note that any choice of the support numbers $a_\rho$ 
yields a virtual polytope (even if the intersection of the corresponding half-spaces is empty).  
See Figures \ref{fig-quadrangle} and \ref{fig-virt-quadrangle} for illustrations of a usual and virtual quadrangle with the same normal fan.

\begin{figure}[H] 
\includegraphics[height=4cm]{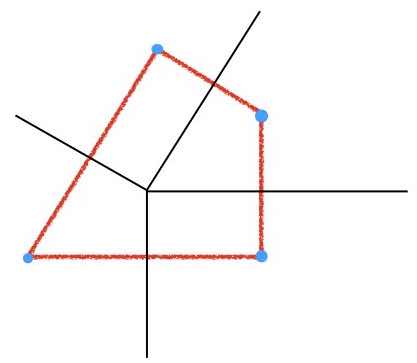}
\caption{A usual quadrangle with its normal fan}
\label{fig-quadrangle}
\end{figure}

\begin{figure}[H]
\includegraphics[height=4cm]{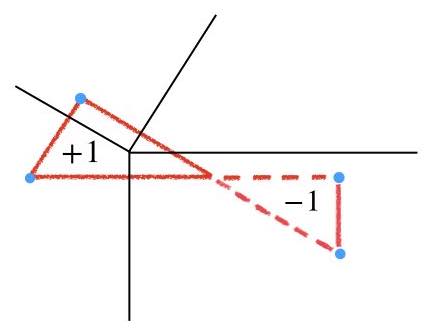}
\caption{A virtual quadrangle with the same normal fan}
\label{fig-virt-quadrangle}
\end{figure}

\begin{remark}  \label{rem-vol-virt-poly}
The notion of volume of a polytope extends to virtual polytopes via Theorem \ref{th-vol-poly}.
For a  virtual polytope $P \in \cV(V^*)$ we defined $\vol_n(P)$ to be the value of the volume polynomial at $P$. 
Similarly, the notion of the number of lattice points in a polytope extends to virtual polytopes as well.  
Let $M \subset V^*$ be a  full rank lattice. Let $\cV(M)$ denote the collection 
\emph{lattice virtual polytopes with respect to $M$}, i.e., all virtual polytopes whose vertices are in $M$.  
In other words, $\cV(M)$ is the subgroup of $\cV(V^*)$ generated by lattice polytopes in $\cP(M)$.  
By Theorem \ref{th-Ehrhart-poly} there exists a (unique) polynomial $F$ on $\cV(V^*)$ such that 
for any lattice polytope $P \in \cP(V^*)$ we have $F(P) =  |P \cap M|$.  
For a virtual lattice polytope $P \in \cV(M)$ we define the number of lattice points in $P$ to be $F(P)$.  
The same applies to any valuation on the space of polytopes (see \cite{Khov-Pukh1};  also see 
Theorem \ref{th-McMullen} and the paragraph before it for the definition of a valuation). 
\end{remark} 

Each polytope $P \in \cP(V^*)$ is determined by its characteristic function ${\mathbf 1}_P: V^* \to \{0, 1\}$.  
We would like to extend the assignment 
$P \mapsto {\mathbf 1}_P$ to virtual polytopes. The natural extension of the set of characteristic functions 
of convex polytopes (to a vector space) is the set of convex chains (defined by Khovanskii and Pukhlikov).

\begin{definition}[Convex chain]    \label{def-convex-chain}
A \emph{convex chain} $Z$ is a finite linear combination (with real coefficients) of characteristic functions 
of convex polytopes in $V^*$, that is, $Z = \sum_i \lambda_i {\mathbf 1}_{\Delta_i}$, where the $\Delta_i$ 
are convex polytopes in $V^*$ and $\lambda_i \in \R$.  We denote the set of convex chains by 
$\mathcal{Z}(V^*)$.  It is an infinite dimensional vector space with addition and scalar multiplication of functions. 
\end{definition}

Moreover, generally one can consider the characteristic functions of convex polyhedral cones.

\begin{definition}[Conical convex chain]   \label{def-conical-convex-chain}
A \emph{conical convex chain} $C$ is a finite linear combination (with real coefficients) of characteristic 
functions of shifted convex cones in $V^*$, that is, $C = \sum_i \lambda_i {\mathbf 1}_{a_i + C_i}$, where 
the $C_i$ are convex polyhedral cones in $V^*$ (with apex at the origin), $a_i \in V^*$ and $\lambda_i \in \R$.  
We denote the set of convex conical chains by $\mathcal{C}\mathcal{Z}(V^*)$.  
\end{definition}

A remarkable construction in \cite{Khov-Pukh1} is a ``convolution'' operation $*$ on $\mathcal{Z}(V^*)$ which 
makes it a commutative algebra (together with addition and scalar multiplication of functions). It has the property 
that for any two polytopes $P_1$, $P_2$ we have 
\[ 
{\mathbf 1}_{P_1} * {\mathbf 1}_{P_2} = {\mathbf 1}_{P_1 + P_2}. 
\] 
In particular, the identity element for the $*$ operation is ${\mathbf 1}_{\{0\}}$, the characteristic function of the origin.

For a polytope $P$, it is shown in \cite{Khov-Pukh1} that the inverse (with respect to $*$) of ${\mathbf 1}_P$ 
is the convex chain $(-1)^{\dim P} {\mathbf 1}_{P^\circ}$, where $P^\circ$ denotes the relative interior of $P$.  
In other words, 
\[ 
{\mathbf 1}_P * (-1)^{\dim P} {\mathbf 1}_{P^\circ} = {\mathbf 1}_{\{0\}}. 
\] 
One verifies that 
\[ 
(-1)^{\dim P} {\mathbf 1}_{P^\circ} = \sum_{Q \preceq P} (-1)^{\dim Q} {\mathbf 1}_Q 
\] 
and hence $(-1)^{\dim P} {\mathbf 1}_{P^\circ}$ is indeed a convex chain.  
It follows that 
\begin{equation}   \label{equ-char-function-virtual}
\iota: P_1 - P_2 \mapsto {\mathbf 1}_{P_1} * (-1)^{\dim P_2} {\mathbf 1}_{P_2^\circ} 
= \sum_{Q \preceq P_2} (-1)^{\dim Q} {\mathbf 1}_{P_1 + Q} 
\end{equation}
defines a natural embedding of the group of virtual polytopes (with Minkowski sum) into the semigroup of 
convex chains (with convolution $*$).  We refer to the righthand side of \eqref{equ-char-function-virtual} 
as the \emph{convex chain associated to} or \emph{characteristic function of} the virtual polytope $P_1 - P_2$.  
In fact, it is shown in \cite{Khov-Pukh1} that the image of $\iota$ coincides with the set of $*$-invertible convex chains.

We can talk about vertices of a virtual polytope. For a virtual polytope $P \in \cV(\Sigma)$, the vertices are 
in one-to-one correspondence with the full dimensional cones in $\Sigma$. Similarly, the notion of a tangent cone 
of a polytope extends to virtual polytopes. The tangent cones of $P \in \cV(\Sigma)$ are in one-to-one correspondence 
with $\sigma \in \Sigma$.

There is a generalization of the Brianchon-Gram theorem to convex chains (see \cite[\S 4, Proposition 2]{Khov-Pukh1}). 
The Lawrence-Varchenko theorem also extends to simple virtual polytopes.

\begin{theorem}[Lawrence-Varchenko for virtual polytopes]    \label{th-LV-virtual}
Let $P$ be a virtual polytope in $V^*$ and let 
$\pi: V^* \to \R$ the corresponding convex chain. Then 
\begin{equation} \label{LV-vir}
\pi = \sum_{v} (-1)^{n_v} {\mathbf 1}_{T^\xi_{P, v}}, 
\end{equation} 
where the sum is over all the vertices $v$ of $P$ and $T^\xi_{P, v}$ and $n_v$ are as in Theorem \ref{th-LV}.  
\end{theorem} 
See Figure \ref{fig-LV-virt-polytope} for an illustration of \eqref{LV-vir}.

\begin{figure}[H]
\includegraphics[height=8.5cm]{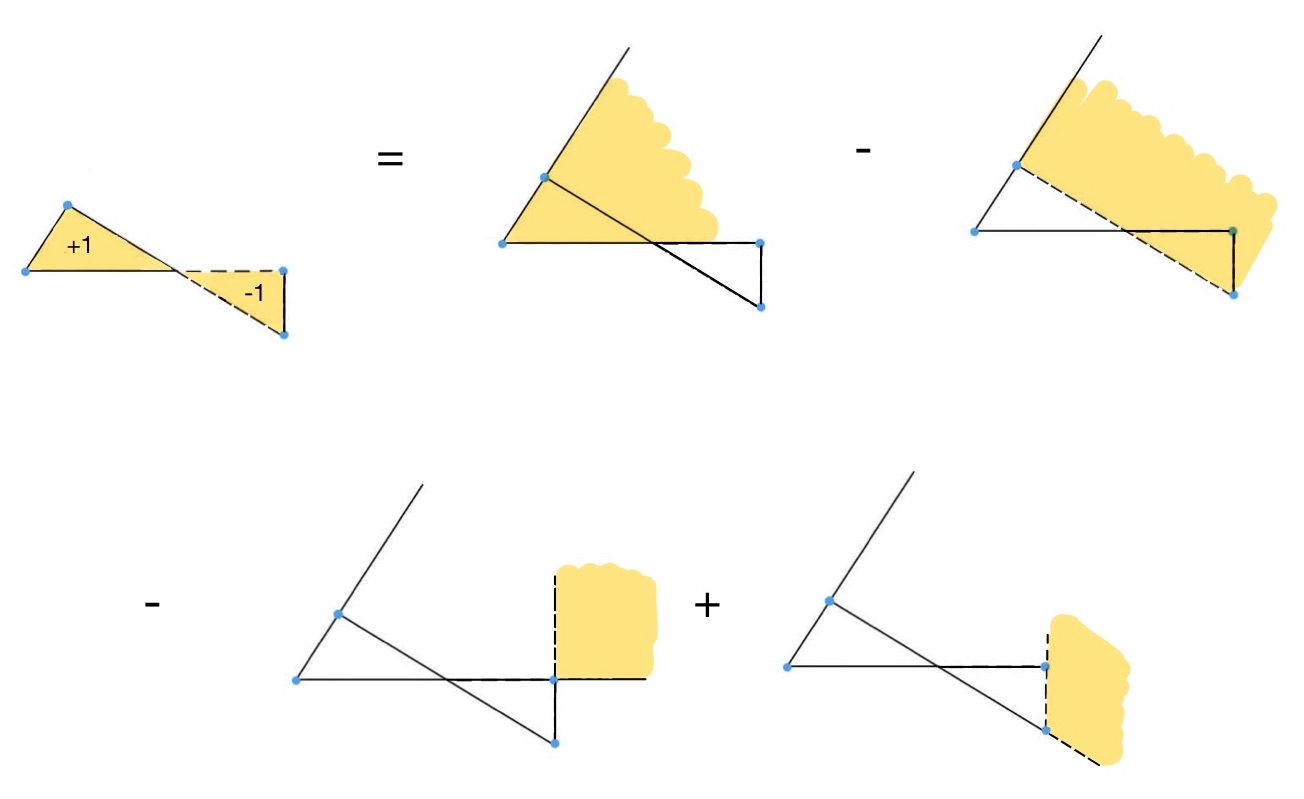}
\caption{Illustration of the Lawrence-Varchenko theorem for a virtual quadrangle}
\label{fig-LV-virt-polytope}
\end{figure}

\subsection{Incidence algebra of a poset and M\"obius inversion}   \label{sec-incidence-alg}
For a nice reference about incidence algebra and M\"obius inversion see \cite[Sections 3.6 and 3.7]{Stanley}.
Let $\mathcal{P}$ be a finite poset with partial oder $\prec$. Let $R$ be a commutative ring with $1$ which 
we take as the ring of scalars. 
Let $\tilde{\mathcal{P}} = \{(\tau, \sigma) : \tau \preceq \sigma \} \subset \mathcal{P} \times \mathcal{P}$ be 
the collection of all intervals in $\mathcal{P}$. Let $I(\mathcal{P}) = \{F: \tilde{\mathcal{P}} \to R \}$ be 
the set of functions from $\tilde{\mathcal{P}}$ to $R$. Clearly $I = I(\mathcal{P})$ is an abelian group with addition of functions. 
One defines a \emph{convolution} operation $*$ on $I$ as follows. For $F, G \in I$ define $F * G \in I$ by 
\[ 
(F * G)(\tau, \sigma) = \sum_{\tau \preceq \tau' \preceq \sigma} F(\tau, \tau') G(\tau', \sigma). 
\] 
It can be verified that $(I, +, *)$ is an algebra over $R$, called \emph{incidence algebra of the poset $\mathcal{P}$}.  
In general, $I(\mathcal{P})$ is not commutative.

The identity (for the convolution operation $*$) is the function $\delta$ defined by 
\[ 
\delta(\tau, \sigma) = \begin{cases} 1 \quad \tau = \sigma \\ 0 \quad \tau \neq \sigma \end{cases}. 
\]

A distinguished element of the incidence algebra is the constant function $\zeta(\tau, \sigma) = 1$, 
for any interval $\tau \preceq \sigma$. The \emph{M\"obius inversion formula} states that the function $\zeta$ 
is invertible and its inverse is the \emph{M\"obius} function $\mu$. For general poset $\mathcal{P}$ the M\"obius 
function is constructed/defined inductively but in specific examples it can be defined/computed explicitly.

\begin{example}[Poset of subsets of a finite set]   \label{ex-subsets-poset}
Let $\mathcal{P}$ be the poset of all subset of $\{1, \ldots, d\}$ ordered by inclusion.  
It can be shown that M\"obius function in this case is given by 
\[ 
\mu(I, J) =  (-1)^{|I| - |J|}, \quad J \subset I, 
\] 
and the M\"obius inversion formula recovers the inclusion-exclusion principle.   
\end{example}

The following is the main example of a poset that we will be concerned with in the paper.
\begin{example}[Poset of faces of a convex polyhedral cone]   \label{ex-faces-poset}
Let $\mathcal{P}$ be the poset of all faces of a given convex polyhedral cone $C \subset \R^n$.  
If $\sigma$ is simplicial of dimension $d$, then this poset is the same as the poset of all subsets of 
$\{1, \ldots, d\}$ above. It can be shown that the M\"obius function in this case is given by  
$$\mu(\tau, \sigma) = (-1)^{\dim \sigma - \dim \tau}, \quad \tau \preceq \sigma.$$
\end{example}

\section{Convergence}

In this section we give some combinatorial/geometric results that contain the combinatorial ingredients of Arthur's 
result on the convergence and polynomiality (in a truncation parameter $T$) of his truncated trace $J^T(f)$ in 
his non-invariant trace formula. See \cite[\S7]{arthur-duke} and \cite[\S2]{arthur-annals} as well as the survey 
\cite[\S\S 8--9]{arthur-clay}.

We continue to denote the $n$-dimensional real vector space we fixed in Section \ref{sec-prelim}  by $V$. 
We choose an inner product $\langle \cdot , \cdot \rangle$ on $V$ and use it to identify $V$ with its dual.  
Our results in this section depend on the choice of this inner product.  In particular, we view the dual 
cone $\sigma^\vee$ as a subset of $V$ itself, 
\[ 
\sigma^\vee := \left\{ x \in V : \langle x , y \rangle \geqslant 0, \text{for all } y \in \sigma\right\}.   
\]

Our starting point is a full dimensional, complete, simplicial, fan $\Sigma$ in $V$.  
Let $\Delta \in \mathcal{P}(\Sigma)$ be a convex polytope whose normal fan is $\Sigma$.  
Suppose that we are given a collection of continuous functions  
\begin{equation}  \label{K_sigma}
K_\sigma: V \longrightarrow \C, \quad \sigma \in \Sigma.  
\end{equation} 
To this data we  associate the \textit{truncated function} $k_\Delta : V \longrightarrow \C$ defined by 
\begin{equation}  \label{equ-comb-truncation}
k_\Delta(x) = \sum_{\sigma \in \Sigma}  (-1)^{\dim\sigma}~K_\sigma(x) ~{\mathbf 1}_{T^-_{\Delta, \sigma}}(x),    
\end{equation} 
where $T^-_{\Delta, \sigma} = T^-_{\Delta, Q_\sigma}$ is the outward tangent cone, as in \eqref{outward-face-cone}, of 
the face $Q_\sigma$ of $\Delta$ that stands in 
bijection with $\sigma$ as in \eqref{face-cone-bijection}.  The main result of this section is to prove that if the 
functions $K_\sigma$ satisfy certain assumptions, then the integral 
of $k_\Delta$ over $V$ is absolutely convergent. In particular, these assumptions hold when the functions $K_\sigma$ 
satisfy certain growth conditions as we explain below.  The latter is the setting in which Arthur's trace formula 
appears.

For a cone $\sigma \in \Sigma$ let $W(\sigma) = \left\{w_i \in V : i \in I\right\}$ 
be a set of unit edge vectors of $\sigma$.  We also let $B(\sigma) = \left\{b_i \in V : i \in I\right\}$ denote 
the set of unit, inward, facet normals in $\operatorname{Span}(\sigma)$ to the facets of $\sigma$. 
Note that the $b_i$ form a basis of $\operatorname{Span}(\sigma)$ dual to the $w_i$, i.e., 
\[ 
\langle w_i, b_j \rangle = \delta_{i,j}, \quad i,j \in I. 
\]  
When $\sigma$ is full dimensional  $B(\sigma)$ is the set of edge vectors of the dual cone $\sigma^\vee$.

\begin{definition}[Acute cone and acute fan] \label{acute}
We say that a convex cone $\sigma$ in $V$ is acute if  $ \sigma \subseteq \sigma^\vee $.  
We call the fan $\Sigma$ acute if all its cones are acute. 
\end{definition}

Notice that our definition of acute allows for right angles. 
We also remark that the notion of acute depends on the inner product we have chosen in $V$.  
Indeed the acute assumption will be crucial for the convergence results below to hold 
as Example \ref{obtuse} shows.

Observe that 
\begin{equation} \label{acute-conditions} 
\sigma \mbox{ is acute} \quad \Longleftrightarrow \quad \langle w_i, w_j \rangle \geqslant 0, \quad i,j \in I. 
\end{equation}   
It follows from Definition \ref{acute} that if $\sigma$ is acute, then for $x \in \operatorname{Span}(\sigma)$ 
\begin{equation} \label{acute-ineqs} 
\langle x , b_i \rangle > 0 \mbox{ for all } i \in I \quad \Longrightarrow \quad \langle x , w_i \rangle > 0 \mbox{ for all } i \in I. 
\end{equation}

Next, fix a pair of cones $\sigma_2 \preceq \sigma_1$ in $\Sigma$.  
Write $W(\sigma_1) = \left\{w_i \in V : i \in I_1\right\}$ and 
$B(\sigma_1) = \left\{b_i \in V : i \in I_1\right\}$ as above. 
Then $W(\sigma_2) = \left\{w_i : i \in I_2 \right\}$ for some $I_2 \subseteq I_1$ 
and the set $\left\{b_j : j \in I_1 \setminus I_2 \right\}$ consists of vectors normal to $\sigma_2$. 
(However, $B(\sigma_2)$ is not $\left\{b_j : j \in I_2 \right\}$ as the latter depends on $\sigma_1$.)

Define 
\begin{equation} \label{T} 
C_{\sigma_1} = C_{\sigma_1}^0 := 
\left\{x \in 
\operatorname{Span}(\sigma_1) : \langle x, b_j \rangle > 0,  \mbox{ for all } j \in I_1 \right\} 
\end{equation} 
and similarly define 
\begin{equation} \label{hatT} 
{\hC}_{\sigma_1} = \hC_{\sigma_1}^0 := 
\left\{x \in 
\operatorname{Span}(\sigma_1) : \langle x, w_i\rangle > 0,  \mbox{ for all } i \in I_1 \right\}.  
\end{equation} 
More generally, we define 
\begin{equation}  \label{T12}
C_{\sigma_1}^{\sigma_2} := 
\left\{x \in \operatorname{Span}(\sigma_1) : \langle x, b_j \rangle > 0, \mbox{ for all } j \in I_1 \setminus I_2 \right\}. 
\end{equation} 
and 
\begin{equation}  \label{hatT12}
\hC_{\sigma_1}^{\sigma_2} := 
\left\{x \in \operatorname{Span}(\sigma_1) : \langle x, w_i\rangle > 0, \mbox{ for all }  i \in I_1 \setminus I_2 \right\}.  
\end{equation} 

Next, we define the following subsets of $V$ which play a crucial role in our results.  

\begin{definition} \label{SR-defn}
Let $\Sigma$ be a full dimensional, complete, simplicial, acute fan in $V$. 
Assume that $\sigma_2 \preceq \sigma_1$ are two cones in $\Sigma$ with unit edge vectors indexed 
by $I_2 \subset I_1$ as above. 
\begin{itemize}
\item[(a)] 
Define $S_{\sigma_1}^{\sigma_2}$ to be the set of 
$x \in \operatorname{Span}(\sigma_1) \cap \sigma_1^\vee$ such that the face of $\sigma_1$ that 
is nearest to $x$ is the cone generated by $\left\{w_i : i \in I_1 \setminus I_2 \right\}$.  Also 
let ${\mathbf 1}_{S_{\sigma_1}^{\sigma_2}}$ denote its characteristic function.  
(See Section \ref{subsec-nearest-face}.) 
\item[(b)] 
Define the ``shifted'' subset   
\begin{equation}\label{R12} 
R_{\sigma_1}^{\sigma_2} := 
Q_{\sigma_1} + S_{\sigma_1}^{\sigma_2} = 
\left\{x_0 + x \in V : x_0 \in Q_{\sigma_1} \mbox { and } x \in S_{\sigma_1}^{\sigma_2} 
\right\}.   
\end{equation} 
\end{itemize}
\end{definition}

We also note that while the subsets $S_{\sigma_1}^{\sigma_2}$ 
may have smaller dimensions, the subsets $R_{\sigma_1}^{\sigma_2}$, when non-empty, are always 
full dimensional because the dimension of $Q_{\sigma_1}$ (as an affine space) and that of 
$S_{\sigma_1}^{\sigma_2}$ add up to $n = \dim V$.

As Lemma \ref{arthur-lemma} below shows the $S_{\sigma_1}^{\sigma_2}$ are the analogues of the subsets appearing 
in \cite[LEMMA 6.1]{arthur-duke} which also appear to play a similar crucial role in 
Arthur's results on convergence and polynomiality.

\begin{lemma} \label{arthur-lemma}
With $\sigma_2 \preceq \sigma_1$ in $\Sigma$, the vectors $w_i$ and $b_i$, and $I_2 \subset I_1$ as above we have 
\begin{equation} \label{S12} 
S_{\sigma_1}^{\sigma_2} = 
\left\{x \in \operatorname{Span}(\sigma_1) : 
\begin{array}{ll} 
\langle x, b_j \rangle > 0,  & j \in I_1 \setminus I_2 \\ 
\langle x, b_j \rangle \leqslant 0, & j \in I_2 \\ 
\langle x, w_i \rangle > 0,  & i \in I_1 \\ 
\end{array} 
\right\} 
\end{equation} 
\end{lemma} 

\begin{proof} 
Write $\tau = \operatorname{Cone}\left(w_i : i \in I_1 \setminus I_2 \right)$.  Fix $x \in \operatorname{Span}(\sigma_1) \cap \sigma^\vee_1$.  
Now $x$ belongs to $S_{\sigma_1}^{\sigma_2}$ if and only if among all the faces of $\sigma_1$ the face $\tau$ 
is the unique face that is nearest to $x$.  Note that the distances to the faces of $\sigma_1$ are controlled by the normal 
vectors $b_j$ and for $\tau$ to be the unique nearest face, we must have 
$\langle x, b_j \rangle > 0$ for $j \in I_1 \setminus I_2$ while 
$\langle x, b_j \rangle \leqslant 0$ for $j \in I_2$.  This implies that $x \in  S_{\sigma_1}^{\sigma_2}$ satisfies 
the first two sets of inequalities on the right hand side of \eqref{S12}. Also, $x$ satisfies the third set of inequalities 
on the right hand side of \eqref{S12} by \eqref{acute-ineqs} because $x \in \sigma^\vee_1$, a cone whose edge vectors are the $b_i$'s. 

Next, assume that $x$ belongs to the right hand side of \eqref{S12}.  The first two sets of inequalities imply that $\sigma_2$ 
is the unique nearest face of $\sigma_1$ to $x$ and the third set of inequalities means that $x \in \sigma_1^\vee$. 
\end{proof} 

\begin{remark} 
Even though we start with simplicial cones $\sigma_2 \preceq \sigma_1$ the cone 
$S_{\sigma_1}^{\sigma_2}$ may not be simplicial.  As an example, consider $V = \R^3$ 
and let $w_1 = e_1, w_2 = e_2, w_3=e_1+e_2+e_3$.  
Take $\sigma_2 = \operatorname{Cone}(w_3) \preceq \sigma_1 = \operatorname{Cone}(w_1, w_2, w_3)$.  
We then have $b_1 = e_1 - e_3, b_2 = e_2 - e_3, b_3=e_3$. 
A simple calculation then shows that  
$S_{\sigma_1}^{\sigma_2} = \operatorname{Cone}(w_1,w_2,b_1,b_2)$ which is not simplicial. 
\end{remark} 

The following is a type of double nearest face partition that will help us prove our convergence results.

\begin{lemma} \label{S-partition} 
Let $\Sigma$ be a full dimensional, complete, simplicial, fan in $V$ which is assumed to be acute. 
Let $\Delta \in \mathcal{P}(\Sigma)$ be a convex polytope whose normal fan is $\Sigma$. 
Then for any $\sigma \in \Sigma$ the outward tangent cone $T^-_{\Delta, \sigma}$ has the partition  
\begin{equation}  \label{partition}
T^-_{\Delta, \sigma} = 
\bigsqcup\limits_{\left\{\sigma_1\in\Sigma \, : \, \sigma\preceq\sigma_1\right\}}~
\bigsqcup\limits_{\left\{\sigma_2 \in \Sigma : \sigma_2\preceq\sigma\right\}}
R_{\sigma_1}^{\sigma_2}. 
\end{equation}
 \end{lemma}

\begin{proof} 
Consider the inner disjoint union in \eqref{partition} first.   
Fix $\sigma_1$ in $\Sigma$ with $\sigma \preceq \sigma_1$. 
Write $W(\sigma_1) = \left\{w_i \in V : i \in I_1 \right\}$ and assume that $I_2 \subseteq I \subseteq I_1$ 
are such that 
$W(\sigma) = \left\{w_i \in V : i \in I \right\}$ and similarly for $W(\sigma_2)$. 
Also, write $B(\sigma_1) = \left\{b_j \in V : j \in I_1 \right\}$. 
Notice that 
$b_j$ is normal to $\sigma$ for $j \in I_1 \setminus I$ and 
$b_j$ is normal to $\sigma_2$ for $j \in I_1 \setminus I_2$.

Simply considering all the subsets of $I$ we see that 
\[ 
A_{\sigma_1}^\sigma := 
\bigsqcup\limits_{\sigma_2 : \sigma_2\preceq\sigma \preceq \sigma_1}
R_{\sigma_1}^{\sigma_2} 
= 
\left\{ 
x \in V : 
\begin{array}{ll} 
\langle x-q , b_i \rangle > 0, & i \in I_1 \setminus I, \\
\langle x-q , w_i \rangle > 0 &  i \in I_1, \\ 
\end{array} 
\mbox{ for some } q \in Q_{\sigma_1}
\right\}.  
\] 
This is because for $q \in Q_{\sigma_1}$ the set $q+ S_{\sigma_1}^{\sigma_2}$ is, 
by \eqref{S12}, given by  
\[ 
\begin{array}{ll} 
\langle x-q , b_i \rangle > 0,  & i \in I_1 \setminus I_2 = (I_1 \setminus I) \sqcup (I \setminus I_2), \\ 
\langle x-q , b_i \rangle \leqslant0, & i \in I_2, \\ 
\langle x-q , w_i \rangle > 0,  & i \in I_1. \\ 
\end{array} 
\] 
In the disjoint union over all subsets $I_2$ of $I$ corresponding to the faces $\sigma_2$ of $\sigma$ 
the first set of inequalities for $i \in I_1 \setminus I$ are common for all the subsets $I_2$ 
and the remaining inequalities along with the second set of inequalities cover all possible signs for 
$\langle x-q , b_i \rangle$ for all $i \in I$.  Moreover, we have 
$\langle x-q , w_i \rangle > 0$ for  $i \in I_1$.  This proves our claim about the inner union 
and, in fact, already proves the lemma for the case when 
$\sigma$ is full dimensional since we only have the inner union in that case.  

Next, we consider the outer union. The assertion of the lemma now amounts to a nearest face partition. 
The set $T^-_{\Delta, \sigma}$ consists of $x \in V$ satisfying $\langle x-q , w_i \rangle > 0, i \in I$ for every 
$q \in Q_\sigma$.  Fix one such $x$.  There is a unique face $Q_{\sigma_1}$ of $\Delta$ with $\sigma \preceq \sigma_1$ 
such that the distance from $x$ to $Q_{\sigma_1}$ is smallest among all the faces contained in $Q_\sigma$.  
Note that the distances are controlled by the normal vectors $b_j$ and for the smallest distance to occur for 
the face $Q_{\sigma_1}$ of $Q_\sigma$, we must have 
$\langle x-q , b_j \rangle > 0$ for $j \in I_1 \setminus I$ and 
$\langle x-q , b_j \rangle \leqslant0$ for $j \in I_0 \setminus I_1$ 
for any $I_0 \supset I$ with $\sigma_0 \in \Sigma$ for some $q \in Q_{\sigma_1}$.  
Therefore, among the $A_{\sigma_1'}^\sigma$ with $\sigma \preceq \sigma_1'$, 
only $A_{\sigma_1}^\sigma$ contains $x$. 
Hence, \eqref{partition} holds. 
\end{proof} 

Let us also fix the following notation.  For $\sigma_2 \preceq \sigma_1$ in $\Sigma$, define the functions 
\begin{equation}  \label{K_sigma12}
K_{\sigma_1, \sigma_2}(x) = 
\sum_{\left\{\tau\in\Sigma \, : \, \sigma_2 \preceq \tau \preceq \sigma_1 \right\}} (-1)^{\dim(\tau)} K_\tau(x), 
\quad x \in V.  
\end{equation}

We are now prepared to state our first convergence result. 

\begin{theorem}[Absolute Convergence]   \label{th-comb-conv}
Let $\Sigma$ be a full dimensional, complete, simplicial, fan in $V$ which is assumed to be acute.
Let $\Delta \in \mathcal{P}(\Sigma)$ be a simple full dimensional polytope in $V$ whose normal fan is $\Sigma$. 
Suppose that a collection of functions $(K_\sigma)_{\sigma \in \Sigma}$ is given as in \eqref{K_sigma} and $k_\Delta$ 
is defined as in \eqref{equ-comb-truncation}. 

For each pair $\sigma_2 \preceq \sigma_1$ in $\Sigma$, assume that the function $K_{\sigma_1,\sigma_2}$ is 
absolutely integrable on the set $R_{\sigma_1}^{\sigma_2}$.    
Then  
\begin{equation}  \label{equ-J_Delta}
J_\Sigma(\Delta) := \int\limits_V k_\Delta(x) \, dx 
\end{equation}
is absolutely convergent. 
Recall that $R_{\sigma_1}^{\sigma_2}$ is defined by \eqref{R12} and $K_{\sigma_1,\sigma_2}$ by \eqref{K_sigma12}.  
\end{theorem}

\begin{proof} 
Recall that $k_\Delta(x)$ is defined in terms of outward tangent cones $T^-_{\Delta, \sigma}$. 
It follows from Lemma \ref{S-partition} that 
\begin{eqnarray*}  
k_\Delta(x) & = &  
\sum_{\sigma \in \Sigma}  (-1)^{\dim(\sigma)} K_\sigma(x) ~{\mathbf 1}_{T^-_{\Delta, \sigma}}(x) \\ 
& = & 
\sum\limits_{\sigma \in \Sigma} (-1)^{\dim(\sigma)} K_\sigma(x) 
\left( 
\sum\limits_{\sigma_1 : \sigma \preceq \sigma_1} 
\sum\limits_{\sigma_2 : \sigma_2 \preceq \sigma} 
{\mathbf 1}_{R_{\sigma_1}^{\sigma_2}}(x) 
\right) 
\\ 
& = & 
\sum\limits_{\sigma_2 \preceq \sigma_1} 
K_{\sigma_1,\sigma_2}(x) 
{\mathbf 1}_{R_{\sigma_1}^{\sigma_2}}(x). 
\end{eqnarray*} 
Hence, 
\[ 
\int\limits_V 
\left| k_\Delta(x) \right| \, dx 
\leqslant 
\sum\limits_{\left\{\sigma_1,\sigma_2\in\Sigma \, : \, \sigma_2 \preceq \sigma_1\right\}}~ 
\int\limits_{R_{\sigma_1}^{\sigma_2}}
\left| K_{\sigma_1,\sigma_2}(x) \right| \, dx   
\]
and each of the integrals on the right hand side is finite by assumption. 
Therefore, the integral on the left hand side is finite. 
\end{proof} 

A special case of Theorem \ref{th-comb-conv} is particularly suitable for applications to Arthur's non-invariant trace formula. 
To state it we review the following standard notions of  growth.

Let $\sigma$ be a cone in $V$.  A function $K: V \to \C$ is said to be \textit{of order $N$} 
in $\sigma$ if there is a constant $C = C_{K,N}$ such that 
\[ 
 |K(x)| \leq C \, |x|^N
\] 
for $x$ in $\sigma$ with $|x|$ sufficiently large. In other words, $K(x) = O(|x|^{N})$ as $x$ tends to $\infty$ in $\sigma$. 
We say $K$ is \textit{rapidly decreasing} on $\sigma$ if for every $N>0$ we have $K(x) = O(|x|^{-N})$ 
as  $x$ tends to $\infty$ in $\sigma$.

\begin{theorem} \label{th-growth-conv}
Let $\Sigma$ be a full dimensional, complete, simplicial, fan in $V$ which is assumed to be acute and 
let $(K_\sigma)_{\sigma\in\Sigma}$ be a collection of continuous functions as in \eqref{K_sigma}.  
Assume that the following two assumptions are satisfied: 
\begin{itemize} 
\item[(i)] 
For all $\sigma \in \Sigma$ the function $K_\sigma$ is constant in the direction of $\operatorname{Span}(\sigma)$ 
(i.e., a function on $\sigma^\perp$). 
\item[(ii)] 
For all pairs of cones $\sigma_2 \preceq \sigma_1$ in $\Sigma$ with the subset 
$S_{\sigma_1}^{\sigma_2}$ non-empty, the function  
$K_{\sigma_1,\sigma_2}$ is of order $N = -(n_1+\epsilon)$ for some $\epsilon > 0$ 
in every shifted neighborhood $B(y,\delta) + S_{\sigma_1}^{\sigma_2}$ for all $y \in V$  
where $B(y,\delta)$ is a (small) ball in $V$ of positive radius $\delta$ around $y$, and  $n_1 = \dim \sigma_1$. 
In particular, this condition is satisfied if 
$K_{\sigma_1,\sigma_2}$ is rapidly decreasing on the shifted neighborhoods.  
\end{itemize} 
Then for $\Delta \in \mathcal P(\Sigma)$ 
the integral \eqref{equ-J_Delta} defining $J_\Sigma(\Delta)$ 
converges absolutely. 
\end{theorem}

\begin{proof} 
By Theorem \ref{th-comb-conv} it is enough to prove that the two assumptions in the statement imply that 
\[ 
\int\limits_{R_{\sigma_1}^{\sigma_2}} \left| K_{\sigma_1,\sigma_2}(x) \right| dx < \infty
\] 
for all pairs $\sigma_2 \preceq \sigma_1$ in $\Sigma$.

We may replace the domain of integration by its closure.  Also recall that 
the closure of $R_{\sigma_1}^{\sigma_2}$ is equal to  closure of $Q_{\sigma_1}$, which is compact, 
plus the closure of $S_{\sigma_1}^{\sigma_2}$, which can be given by making all the 
inequalities in \eqref{S12} non-strict. Note that $S_{\sigma_1}^{\sigma_2}$ is a cone, even though 
it may be non-simplicial.

To estimate the integral above, we apply Fubini's theorem to break the integral as three iterated integrals,  
an integral over $Q_{\sigma_1}$, an integral over $A = \sigma_2^\perp \cap \operatorname{Span}(\sigma_1)$, 
and a third integral in the direction of $\sigma_2$.  

Note that $\operatorname{Span}(\sigma_2)$ does not intersect 
$S_{\sigma_1}^{\sigma_2}$ because for any $x \in \operatorname{Span}(\sigma_2)$ the third set of 
inequalities in \eqref{arthur-lemma} for $i \in I_2$ and \eqref{acute-ineqs} imply that $x$ can not satisfy the second 
set of inequalities in \eqref{arthur-lemma}.  This observation and our first assumption imply that the contribution of the integral 
over $\sigma_2$ is bounded, up to a constant, by the product of the integrand with $|x|^{n_2},$ where $n_2 = \dim \sigma_2$. 
Hence, the integral above is bounded, up to a constant, by  
\[ 
\int\limits_{Q_{\sigma_1}} 
\int\limits_{A} 
\left| K_{\sigma_1,\sigma_2}(x) \right| |x|^{n_2} dx.  
\] 
Next, using the second assumption and the fact that $Q_{\sigma_1}$ is compact, we may cover the domain of integration by a finite 
number of shifted neighborhoods. Therefore, up to a constant, the integral over $A$, which is a cone of dimension $n_1 - n_2$ 
is bounded by 
\[ 
\int\limits_{A} 
|x|^{N+n_2} \, dx.  
\] 
The volume element on $A$ involves $|x|^{\dim A - 1}$ and $\dim A = n_1 - n_2$ which implies that the original integral 
is convergent if $N+n_2+(n_1-n_2-1)+1 = - \epsilon > 0$ which is clear. This proves the theorem. 
\end{proof}

We will give several examples of the convergence theorems later in Section \ref{sec-poly}.  
At the moment we mention the following example, which shows that the acute assumption in our convergence results is crucial. 

\begin{example} \label{obtuse} 
Consider the complete fan $\Sigma$ in $V=\R^2$ pictured in Figure \ref{fig-obtuse-fan}.  
In addition to zero, $\Sigma$ contains three one dimensional cones $\sigma_x$, $\sigma_y$, $\sigma_z$,   
as well as three two dimensional cones $\sigma_{xy}$, $\sigma_{xz}$, and $\sigma_{yz}$. 
Also, let $\Delta$ be a polytope whose normal fan is $\Sigma$ as indicated. 

For convenience, let us write $z = x+y$.  Define the collection of functions $(K_\sigma)_{\sigma \in \Sigma}$ 
as follows.  
\begin{itemize} 
\item $K_{xy} = K_{xz} = K_{yz} = 1$;
\item $K_x = K_x(y) = 1 + e^{-|y|}$; $K_y = K_y(x) = 1 + e^{-|x|}$; $K_z = K_z(x,y) = 1 + e^{-|z|}$; 
\item $K_0 = K_0(x,y) = e^{-|z|} + e^{-|x|} + e^{-|y|}$. 
\end{itemize} 
In Figure \ref{fig-obtuse-ex} we have indicated all the non-empty $R_{\sigma_1}^{\sigma_2}$. 
The truncated function $k_\Delta$ is the sum of the functions in the various regions indicated.  
A simple calculation shows that there are four regions where the integral of $|k_\Delta|$ is divergent.  
These regions are precisely those that are not of the form $R_{\sigma_1}^{\sigma_2}$ in this example, 
while on the other regions the hypotheses of Theorem \ref{th-growth-conv} clearly hold.  As it is evident from 
this example, the crucial Lemma \ref{S-partition} fails which leads to the failure of Theorem \ref{th-growth-conv} 
without the acute assumption. 

\begin{figure}[H]
\includegraphics[height=4cm]{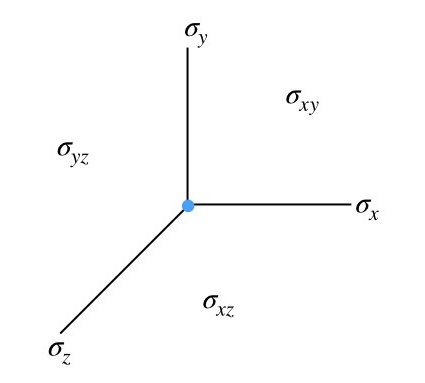}
\caption{An example of an obtuse fan, it is the normal fan of a right triangle.}
\label{fig-obtuse-fan}
\end{figure}

\begin{figure}[H]
\includegraphics[height=6cm]{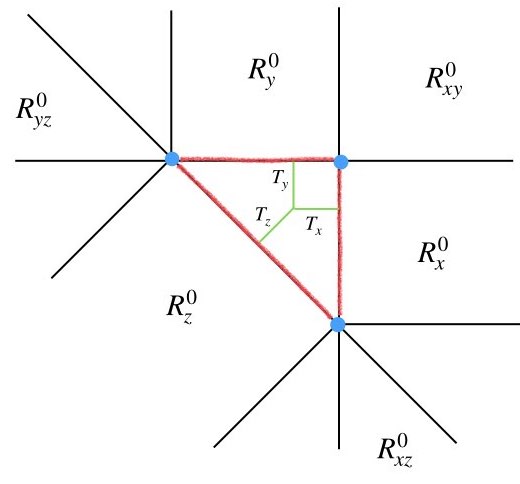}
\includegraphics[height=6cm]{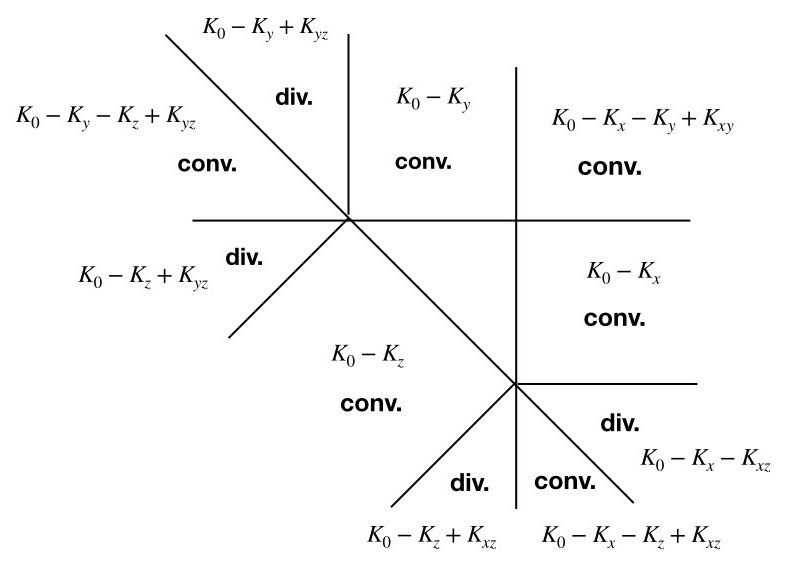}
\vspace{.5cm}
\caption{The regions $R_{\sigma_1}^{\sigma_2}$ and their corresponding $K_{\sigma_1, \sigma_2}$ functions}
\label{fig-obtuse-ex}
\end{figure}

\end{example}

We also prove the following lemma for later use in Section \ref{sec-poly}.  
Let $\tau$ be a cone in $\Sigma$.  
Recall from Section \ref{subsec-normal-fan} that $\Sigma / \tau$ denotes the fan 
consisting of all the images of the cones $\sigma \succeq \tau$ in the quotient vector space $V / \Span(\tau) \cong \tau^\perp$.
For $\sigma \succeq \tau$, let us denote the image of $\sigma$ in $V / \Span(\tau)$ by $\bar{\sigma}$. 
Note that by assumption, for any $\sigma \succeq \tau$, the function $K_\sigma$ is constant along $\Span(\tau)$ 
and hence induces a well-defined function 
$\bar{K}_{\bar{\sigma}}$ on $V / \Span(\tau)$.
\begin{lemma}   \label{lem-conv-Sigma/tau}
Suppose the conditions in Theorem \ref{th-growth-conv} for convergence are satisfied for the $K_\sigma$, $\sigma \in \Sigma$.  
Then for any $\tau \in \Sigma$, these conditions are also satisfied for the $\bar{K}_{\bar{\sigma}}$, $\bar{\sigma} \in \Sigma / \tau$, 
and hence $J_{\Sigma/\tau}(0)$ is convergent as well.
\end{lemma}
\begin{proof}
This is an immediate corollary of the following two observations.
Let $\tau \preceq \sigma_2 \preceq \sigma_1$.  Then we have that (1) the cone $\mathcal{S}^{\bar{\sigma}_1}_{\bar{\sigma}_2}$ 
(as in the proof of Theorem \ref{th-comb-conv}) coincides with the image of $\mathcal{S}^{\sigma_1}_{\sigma_2}$ in $V / \Span(\tau)$;  
(2) the function $K_{\bar{\sigma}_1, \bar{\sigma}_2}$ (as in the statement of Theorem \ref{th-comb-conv}) is 
rapidly decreasing on a shifted neighborhood $\mathcal{S}^{\bar{\sigma}_1}_{\bar{\sigma}_2}$ because 
$K_{\sigma_1, \sigma_2}$ is rapidly decreasing on a shifted neighborhood of $\mathcal{S}^{\sigma_1}_{\sigma_2}$.
\end{proof}

Finally, we give a discrete version of Theorem \ref{th-growth-conv}. As usual let $N$ and $M$ be dual lattices and let $V=N_\R=N \otimes \R$ and 
$V^*=M_\R =M  \otimes  \R$ be the 
corresponding vector spaces respectively. We fix a perfect pairing $N \times N \to \Z$ and use it to identify $N$ and $M$ as well as $N_\R$ and $M_\R$. 
\begin{theorem} \label{th-growth-conv-discrete}
With notation and assumptions as in Theorem \ref{th-growth-conv}, the sum 
$$S_\Sigma(\Delta, M) =  \sum_{m \in M} k_\Delta(m),$$ is absolutely convergent. 
\end{theorem}
\begin{proof}
In the proof of Theorem \ref{th-growth-conv} replace all integrals $\int\limits_A f(x)dx$ with sums $\sum\limits_{m \in A \cap M}  f(m)$.
\end{proof}

We should note that the discrete analogue of Lemma \ref{lem-conv-Sigma/tau} also holds with the same proof.

\section{Polynomiality}  \label{sec-poly}
In this section we prove the following theorems.

\begin{theorem}[Polynomiality] \label{th-polynomiality} 
Let $\Sigma$ be a full dimensional, complete, simplicial, fan in $V$ which is assumed to be acute. 
Let $(K_\sigma)_{\sigma\in\Sigma}$ be a collection of continuous functions 
satisfying assumptions (i) and (ii) in Theorem \ref{th-growth-conv}.  
Then $$J_\Sigma(\Delta) = \int\limits_V k_\Delta(x) dx$$ is a polynomial function on 
$\cP(\Sigma)$, i.e., a polynomial in the support numbers of $\Delta$. 
\end{theorem} 

We also prove a discrete version of the above polynomiality result. 
Let $N$ and $M$ be dual lattices with $V=N_\R$ and $V^*=M_\R$ the corresponding vector spaces. 
We fix a perfect $\Z$-pairing $N \times N \to \Z$ and use it to identify $N$ and $M$. 
Recall that $\cP(\Sigma, M)$ denotes the collection of polytopes with normal fan $\Sigma$ whose vertices lie in $M$.

\begin{theorem}   \label{th-polynomiality-discrete}
Let the notations and assumptions be as in Theorem \ref{th-polynomiality}. Then 
\[ 
S_{\Sigma}(\Delta) = \sum_{m \in  M}  k_\Delta(m) 
\] 
is a polynomial function on $\cP(\Sigma, M)$.  
\end{theorem}

A key step in the proof of Theorem \ref{th-polynomiality}  is a combinatorial lemma (Lemma \ref{lem-Gamma-Delta-sigma}) 
which we deduce as a corollary of the Lawrence-Varchenko conical decomposition (Theorem \ref{th-LV-virtual}).  
The notion of a virtual polytope naturally appears here (see Section \ref{sec-convex-chains}).  
The proof of Theorem \ref{th-polynomiality-discrete} is slight modification of the proof of Theorem \ref{th-polynomiality}.  
We give the proofs in Section \ref{ProofPoly} below after some preparation.  Let us give some examples first.

\begin{example}[Brianchon-Gram] \label{ex-BG}
Let $\Sigma$ be a simplicial fan in  $V$ with $\Delta \in  \mathcal{P}(\Sigma)$ a polytope normal to $\Sigma$.  
Let $K_\sigma \equiv 1$, $\forall \sigma \in \Sigma$. The  combinatorial truncation $k_\Delta$ in this case is given by 
$$k_\Delta = \sum_{\sigma \in \Sigma}  (-1)^{\dim(\sigma)} {\mathbf 1}_{T^-_{\Delta, \sigma}}.$$
By the Brianchon-Gram theorem (Theorem \ref{th-BG-alt}), we have  $$k_\Delta = {\mathbf 1}_{\Delta}.$$

For any pair of cones $\sigma_1 \preceq \sigma_2$ in $\Sigma$ we have 
$$K_{\sigma_1, \sigma_2} = \sum_{\left\{\tau\in\Sigma \, : \, \sigma_2 \subseteq \tau \subseteq \sigma_1 \right\}} (-1)^{\dim(\tau)} = 0 $$
by the binomial identity $\sum_{k=0}^n (-1)^k {n \choose k} = 0$. Thus, the conditions in Theorem  \ref{th-growth-conv} are satisfied.  
Moreover, the $K_\sigma$ are constant and hence the assumptions in the polynomiality theorem are also satisfied.  
Thus we recover the polynomiality of the volume function $\Delta  \mapsto \vol(\Delta)$ (see Theorem \ref{th-vol-poly}).  
\end{example}

\begin{example}[Rectangle]  \label{ex-rectangle}
We consider the fan $\Sigma$ in $V =\R^2$ as in Figure \ref{fig-rectangle}, consisting of one dimensional cones 
$\sigma_x$ and $\sigma_y$ and their opposites, as well as the two dimensional cone $\sigma_{xy}$ and its counterparts 
for the other three quadrants. 
We also have the cone $\{0\}$. 
The fan $\Sigma$ is normal to the rectangle $\Delta$ with support numbers $T_1, T_2, T'_1, T'_2$ as indicated.
\begin{figure}[H]
\includegraphics[height=5cm]{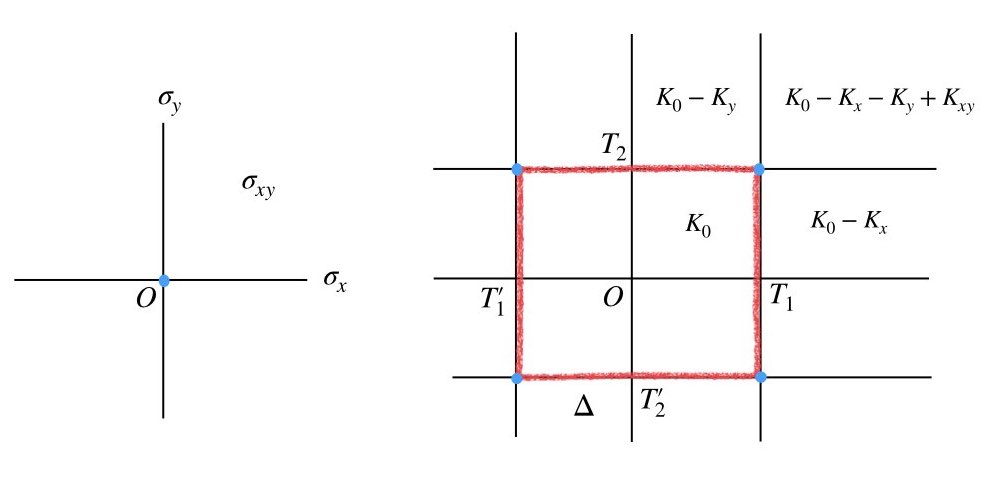}
\caption{}
\label{fig-rectangle}
\end{figure}
Let $f(x,y)$ be an absolutely integrable function on $\R^2$ with $f_{++}$ denoting the value of its integral over the first quadrant. 
Similarly let the values of its integral over the other quadrants be denoted by $f_{+-}$, $f_{-+}$, $f_{--}$. 
Also, let $g(x)$ and $h(y)$ be absolutely integrable functions on $\R$ with their integrals over $[0,\infty)$ denoted by $g_+$ and $h_+$ 
and their integrals over $(-\infty,0]$ denoted by $g_-$ and $h_-$, respectively.  Finally, let $k$ denote a constant. 
   
We assign the following functions to the cones in $\Sigma$: 
\begin{itemize}
\item $K_0(x, y) = f(x,y) + g(x) + h(y) + k$,
\item $K_{\sigma_{\pm x}}(x,y) = h(y) + k$,
\item $K_{\sigma_{\pm y}}(x,y) = g(x) + k$, and 
\item $K_\sigma(x,y) = k$ for all two dimensional cones $\sigma$ in $\Sigma$.
\end{itemize}
Notice that the conditions (i) and (ii) of Theorem \ref{th-growth-conv} are clearly satisfied.  

Let us calculate $J_\Sigma(\Delta)$. Because of the symmetry in this example, it is enough to consider a quarter of the picture. 
We have 
\begin{align*}
& \int_0^{T_1} \int_0^{T_2} \left(f(x,y) + g(x) + h(y) + k\right) \, dy \, dx + \int_0^{T_1} \int_{T_2}^\infty \left(f(x,y)+h(y)\right) \, dy \, dx \\ 
& + \int_{T_1}^\infty \int_{0}^{T_2} \left(f(x,y) + g(x)\right) \, dy \, dx + \int_{T_1}^\infty \int_{T_2}^\infty f(x,y) \, dy \, dx \\ 
&= f_{++} + g_+ T_2 + h_+ T_1 + k T_1 T_2,\\ 
\end{align*}
which is a polynomial of degree $2$ in $T_1$ and $T_2$.  Adding similar contributions from the other three quadrants we arrive at 
\[ 
J_\Sigma(\Delta) = 
k \, (T_1 + T'_1) (T_2 + T'_2) + h_+ \, T_1 + g_- \, T'_1 + g_+ \, T_2 + g_- \, T'_2 + (f_{++} + f_{+-} + f_{-+} + f_{--}). 
\] 
\end{example}

\subsection{An extension of Langlands combinatorial lemma}
As before $V$ is an $n$-dimensional real vector space. We fix an inner product $\langle \cdot, \cdot \rangle$ on $V$ 
and identify $V$ with its dual space $V^*$.  Let $\Sigma$ be a full dimensional, complete, simplicial, fan in $V$ 
and let $\Delta \in \mathcal{P}(\Sigma)$ be a full dimensional simple polytope with normal fan $\Sigma$.  
Since we identified $V$ and $V^*$, we take both $\Sigma$ and $\Delta$ to lie in $V$.

Let $\sigma \in \Sigma$ be a cone. First we consider the case where $\sigma$ is full dimensional.  
Let $v_\sigma$ be the corresponding vertex of $\Delta$. 
Let $W = \{w_1, \ldots, w_n\}$ (respectively $B = \{b_1, \ldots, b_n\}$) be the set of edge vectors of $\sigma$ 
(respectively, of $\sigma^\vee$).  Then the $b_i$ (respectively, the $w_j$) are the inward facet normals to 
$\sigma$ (respectively $\sigma^\vee$), 
and the cone $\sigma$ is given by inequalities as 
\[ 
\sigma = \left\{ x : \langle x, b_i \rangle \geqslant0,~i=1, \ldots, n \right\}. 
\] 
Also the inward-looking tangent cone $T^+_{\Delta, \sigma}$ at the vertex $v_\sigma$ is given by
\[ 
T^+_{\Delta, \sigma} = \left\{ x : \langle x, w_i \rangle \leqslant \langle v_\sigma, w_i \rangle,~ i=1, \ldots, n \right\}. 
\] 
We consider the oriented hyperplanes corresponding to the union of these two sets of inequalities: 
\begin{align} 
H_{b_i, 0} &= \left\{x : \langle x, b_i \rangle = 0 \right\}, & i=1, \ldots, n,  \nonumber  \\
H_{w_i, \langle v_\sigma, w \rangle} &= 
\left\{x : \langle x, w_i \rangle = \langle v_\sigma, w \rangle \right\}, &  i=1, \ldots, n.  
\label{equ-Gamma-Delta-sigma}
\end{align}
If $v_\sigma$ lies in $\sigma$ then the hyperplanes in \eqref{equ-Gamma-Delta-sigma} are the facets of the polytope 
$\Delta \cap \sigma$ oriented outward. In general, $v_\sigma$ may not lie in $\sigma$.

\begin{definition}   \label{def-Gamma-Delta-sigma}
We denote the virtual polytope in $V$ determined by the oriented hyperplanes in \eqref{equ-Gamma-Delta-sigma} by 
$\Gamma_{\Delta, \sigma}$. 
We denote the convex chain corresponding to $\Gamma_{\Delta, \sigma}$ by $\gamma_{\Delta, \sigma}$.
\end{definition}

See Section \ref{sec-convex-chains} for a review of the notions of virtual polytope and convex chain.  
Also see Figure \ref{fig-cube} for a three dimensional example of $\Gamma_{\Delta,\sigma}$ 
and Figure \ref{fig-ex-Gamma} for a pair of two dimensional examples of the virtual polytope 
$\Gamma_{\Delta, \sigma}$ 
and its convex chain $\gamma_{\Delta, \sigma}$.

In this section we consider the Lawrence-Varchenko conical decomposition for the virtual polytope 
$\Gamma_{\Delta, \sigma}$ (Theorem \ref{th-LV-virtual}). We will see that this recovers and extends 
some of the key combinatorial lemmas appearing in Arthur's work (e.g. \cite{arthur-annals}).  
As a special case we immediately recover the Langlands combinatorial lemma (see \cite[Section I.8, p. 46]{arthur-clay} 
and \cite[Appendix B]{gkm}). In addition, we interpret the Langlands combinatorial lemma as a formula for 
the inverse of a distinguished element in the incidence algebra of poset of faces of 
$\sigma$ (see Section \ref{sec-incidence-alg}).

Recall that for $\tau \preceq \sigma$, the largest face of $\sigma^\vee$ orthogonal to $\tau$ is denoted by 
$\tau^*$ and we have $\dim\tau + \dim\tau^* = n$ (Section \ref{subsec-fan}). It follows that the intersection 
$\Span(\tau) \cap (v_\sigma + \Span(\tau^*))$ is a single point which can be shown to be a vertex 
$v_\tau$ of $\Gamma_{\Delta, \sigma}$. In fact, we will see below that $\tau \mapsto v_\tau$ gives 
a one-to-one correspondence between the faces of $\sigma$ and the vertices of $\Gamma_{\Delta, \sigma}$. 
The vertex corresponding to the zero dimensional face $0$ is $0$ itself. On the other hand, the vertex 
corresponding to the whole $\sigma$ is the vertex $v_\sigma$ of $\Delta$. 

\begin{figure}
\includegraphics[height=4.5cm]{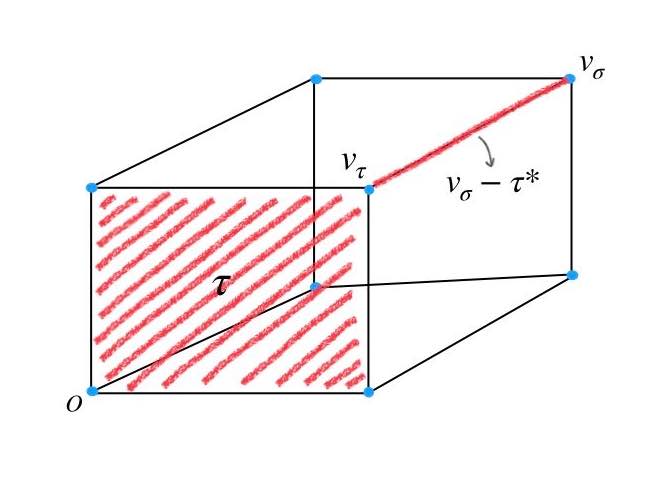}
\caption{A three dimensional example where $\Gamma_{\Delta, \sigma}$ is a cube.  
A face $\tau$ (of $\sigma$) and its corresponding dual face $\tau^*$ (of $\sigma^\vee$) 
and the vertex $v_\tau$ (of $\Gamma_{\Delta, \sigma}$) are illustrated.}
\label{fig-cube}
\end{figure}

For a face $\tau \preceq \sigma$  let $W(\tau) \subset W$ (respectively $B(\tau) \subset B$) 
be the subset of edge vectors of $\tau$ (respectively $\tau^*$). Thus 
$$\tau = \{ x \in \sigma : \langle x, b \rangle = 0,~ b \in B(\tau) \}.$$
The vertex $v_\tau$ is then the unique solution of the system of equations 
\[ 
\begin{cases} 
\langle x, w \rangle = \langle v_\sigma, w \rangle, & \forall w \in W(\tau), \\ 
\langle x, b \rangle = 0, & \forall b \in B(\tau). \\ 
\end{cases} 
\] 
And the inward tangent cone $T^+_{\Gamma_{\Delta, \sigma}, v_\tau}$ at the vertex $v_\tau$ is given by the inequalities 
\[ 
T^+_{\Gamma_{\Delta, \sigma}, v_\tau} =
\left\{x \in V : 
\begin{array}{ll}
\langle x, w \rangle \leqslant \langle v_\sigma, w \rangle, & \forall w \in W(\tau) \\ 
\langle x, b \rangle \geqslant0, & \forall b \in B(\tau)  
\end{array}
\right\}. 
\] 
Thus the set of outward facet normals of $\Gamma_{\Delta, \sigma}$ at $v_\tau$ is $W(\tau) \cup -B(\tau)$.
In other words, the cone in the normal fan of $\Gamma_{\Delta, \sigma}$ corresponding to the vertex 
$v_\tau$ is generated by the set of vectors 
$W(\tau) \cup -B(\tau)$ (Section \ref{subsec-normal-fan}). 

Consider the nearest face partition corresponding to $\sigma$ (Section  \ref{subsec-nearest-face}). 
That is, for each face $\tau$ let $V_\sigma^\tau$ 
be the set of points $x \in V$ whose shortest distance to $\sigma$ is attained at a point in the relative interior of 
$\tau$. Since $\sigma$ is a 
cone, each $V_\sigma^\tau$ is a full dimensional cone. 
Moreover, the closures of the cones $V_\sigma^\tau$, $\tau \preceq \sigma$, are the maximal cones of 
a complete simplicial fan in $V$ which we call the 
\emph{nearest face fan} of $\sigma$. 
The following is straightforward to verify.
\begin{proposition}   \label{prop-nearest-face-fan}
In the nearest face fan of $\sigma$, 
the cone corresponding to a face $\tau \preceq \sigma$ is the convex cone generated by the set of vectors 
$W(\tau) \cup -B(\tau)$.
\end{proposition}

Since the $V_\sigma^\tau$ partition the whole space $V$, the above proposition shows that the union of the cones generated by 
$W(\tau) \cup -B(\tau)$, $\tau \preceq \sigma$, is $V$. This then implies that the normal fan of $\Gamma_{\Delta, \sigma}$ 
coincides with the nearest fan of $\sigma$. In particular, 
the $v_\tau$ are all of the vertices of $\Gamma_{\Delta, \sigma}$. 
In other words, $\tau \mapsto v_\tau$ gives a one-to-one correspondence between the faces of $\sigma$ 
and the vertices of $\Gamma_{\Delta, \sigma}$.

\begin{figure}
\includegraphics[height=5.2cm]{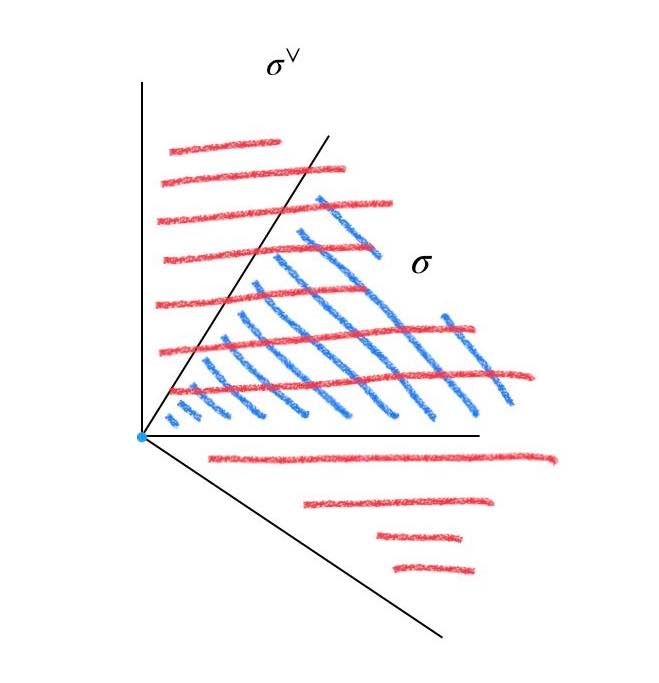}
\includegraphics[height=5cm]{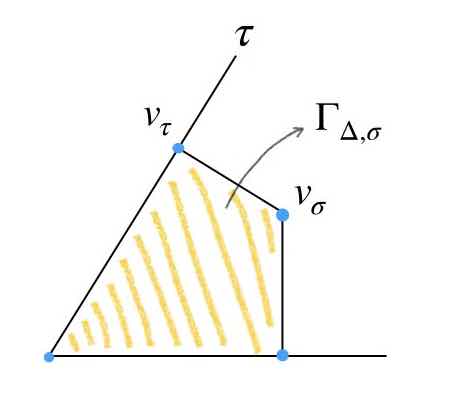}
\includegraphics[height=5.2cm]{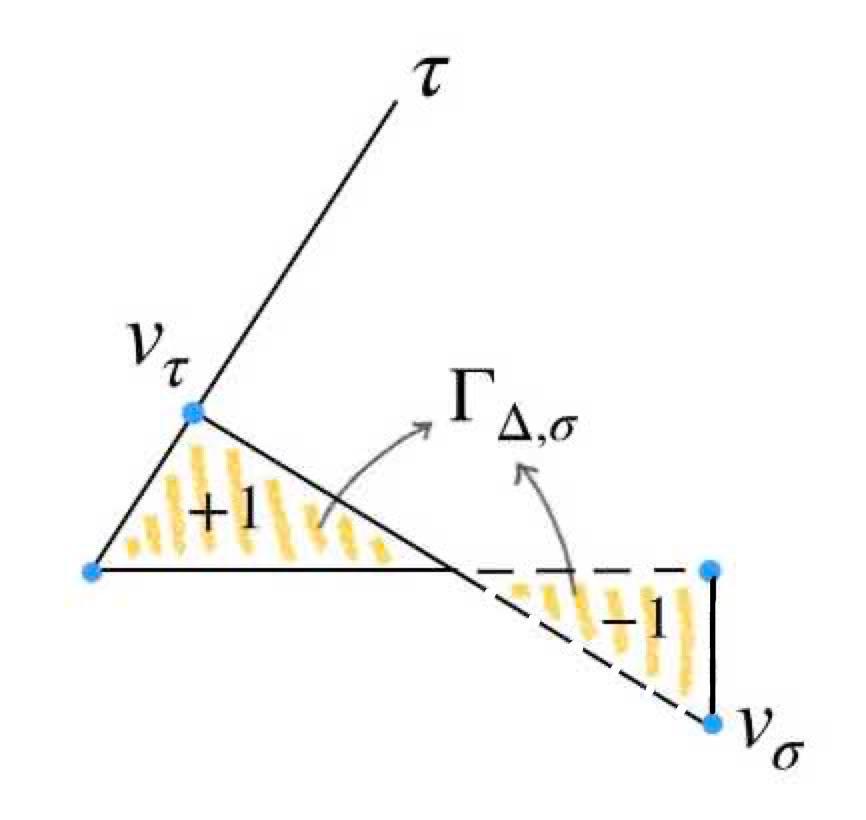}
\caption{Two examples of the virtual polytopes $\Gamma_{\Delta, \sigma}$. In the first example the vertex 
$v_\sigma$ lies in the cone $\sigma$ and $\Gamma_{\Delta, \sigma}$ is an actual polytope (a quadrangle). 
The convex chain $\gamma_{\Delta, \sigma}$ is the characteristic function of the quadrangle. In the second 
example, $v_\sigma$ lies outside $\sigma$ and $\Gamma_{\Delta, \sigma}$ is a virtual quadrangle.  
The convex chain $\gamma_{\Delta, \sigma}$ is the function which has values $1$ and $-1$ in the two 
shaded regions respectively.}
\label{fig-ex-Gamma}
\end{figure}
Now take a vector $\xi$ in $\sigma^\circ \cap (\sigma^\vee)^\circ$, that is, 
$$\langle \xi, b \rangle > 0,~ \forall b \in B,$$
$$\langle \xi, w \rangle > 0,~ \forall w \in W.$$
Note that since $\sigma \neq V$, we know $(\sigma^\circ)^\vee + \sigma^\circ \neq V$ and hence 
$\sigma^\circ \cap (\sigma^\vee)^\circ = ((\sigma^\circ)^\vee + \sigma^\circ)^\vee \neq \emptyset$.

Let $T^\xi_{\Gamma_{\Delta, \sigma}, v_\tau}$ be the polarized tangent cone at the vertex $v_\tau$ appearing in 
the Lawrence-Varchenko decomposition of $\Gamma_{\Delta, \sigma}$ relative to the vector $\xi$ 
(see Section \ref{sec-conical-decompose}). By construction, the edge vectors of $T^\xi_{\Gamma_{\Delta, \sigma}, v_\tau}$ 
are $\pm$ the edge vectors of the tangent cone of $\Gamma_{\Delta, \sigma}$ at $v_\tau$ so that the minimum of 
$\langle \xi, \cdot \rangle$ on $T^\xi_{\Gamma_{\Delta, \sigma}, v_\tau}$ is attained at the vertex $v_\tau$.  
Since the inner product of $\xi$ with any vector in $W \cup B$ is positive, it follows that the set of inward facet normals of 
$T^\xi_{\Gamma_{\Delta, \sigma}, v_\tau}$ is exactly $W(\tau) \cup B(\tau)$. More precisely, 
$T^\xi_{\Gamma_{\Delta, \sigma}, v_\tau}$ is defined  by the inequalities 
\begin{equation} \label{equ-T-xi}
T^\xi_{\Gamma_{\Delta, \sigma}, v_\tau} = 
\left\{ x \in V : 
\begin{array}{ll}
\langle x, w \rangle > \langle v_\sigma, w \rangle, & \forall w \in W(\tau), \\ 
\langle x, b \rangle \geqslant 0, & \forall b \in B(\tau), 
\end{array} 
\right\}.
\end{equation}
On the other hand, 
let $C_\sigma^\tau$ be the inward looking tangent cone of $\sigma$ at $\tau$. It is the cone defined as 
$$C_\sigma^\tau = \{ x \in V : \langle x, b \rangle \geqslant 0,~\forall b \in B(\tau)\}.$$
It follows from \eqref{equ-T-xi} that $T^\xi_{\Gamma_{\Delta, \sigma}, v_\tau}$ can be written as 
$$T^\xi_{\Gamma_{\Delta, \sigma}, v_\tau} = C_\sigma^\tau \cap T^-_{\Delta, \tau}.$$

If $\sigma$ is not full dimensional, we can repeat the above, replacing $\Delta$ with $\Delta \cap \Span(\sigma)$.  
Then $\gamma_{\Delta, \sigma}$ is a convex chain supported on $\Span(\sigma)$. We extend $\gamma_{\Delta, \sigma}$ 
to the whole $V$ by requiring it to be constant along $\sigma^\perp$. 
Now applying the Lawrence-Varchenko theorem to the virtual polytope $\Gamma_{\Delta, \sigma}$ and the vector $\xi$ 
as above we obtain the following conical decomposition for $\Gamma_{\Delta, \sigma}$.

\begin{lemma}  \label{lem-Gamma-Delta-sigma}
With notation as above, let $\gamma_{\Delta, \sigma}$ be the convex chain associated to the virtual polytope 
$\Gamma_{\Delta, \sigma}$. 
We have 
\begin{equation}   \label{equ-Delta-tilde-sigma}
\gamma_{\Delta, \sigma} = \sum_{\tau \preceq \sigma} (-1)^{\dim \tau} {\mathbf 1}_{C_\sigma^\tau} {\mathbf 1}_{T^-_{\Delta, \tau}}.
\end{equation}
\end{lemma}   
\begin{proof}
First, we note that the number $n_{v_\tau}$ of the edges flipped in the polarized tangent cone 
$T^\xi_{\Gamma_{\Delta, \sigma}, v_\tau}$ is equal to $|W(\tau)| = \dim \tau$.
The above discussion then proves the case where $\sigma$ is full dimensional. If $\sigma$ is not full dimensional, 
all the cones considered in the right hand side of \eqref{equ-Delta-tilde-sigma} above should be extended 
in the orthogonal direction $\sigma^\perp$. This finishes the proof.
\end{proof}

Letting $\Delta = \{0\}$ we recover a combinatorial lemma of Langlands.

\begin{corollary}[Langlands combinatorial lemma]   \label{cor-Langlands-comb-lem}
Let $\sigma \subset V$ be a convex polyhedral cone.  
The following identities hold. 
\begin{equation}   \label{equ-8-10}
\sum_{\tau \preceq \tau' \preceq \sigma} (-1)^{\dim \tau + \dim \tau'} {\mathbf 1}_{C_\sigma^{\tau'}} {\mathbf 1}_{C_{\tau^*}^{\tau'^*}} = 
\begin{cases} 
1, & \mbox{ if } \tau = \sigma, \\ 
0 & \mbox{ if }  \tau \neq \sigma 
\end{cases}
\end{equation}

\begin{equation}  \label{equ-8-11}
\sum_{\tau \preceq \tau' \preceq \sigma} (-1)^{\dim \tau' + \dim \tau} {\mathbf 1}_{C_\tau^{\tau'}} {\mathbf 1}_{C_{\sigma^*}^{\tau'^*}} = 
\begin{cases} 
1, & \mbox{ if } \tau = \sigma, \\ 
0, & \mbox{ if } \tau \neq \sigma. 
\end{cases}
\end{equation}

Alternatively, consider the incidence algebra of the poset of faces of $\sigma$ with ring of scalars $R$ being 
the ring of all real-valued functions on $V$ (see Section \ref{sec-incidence-alg} and Example \ref{ex-faces-poset}).  
Define the elements $F$, $G$ of the incidence algebra by 
\begin{align*}
F(\tau, \tau') &= (-1)^{\dim \tau} {\mathbf 1}_{C_\tau^{\tau'}},\\ 
G(\tau, \tau') &= (-1)^{\dim \tau} {\mathbf 1}_{C_{\tau^*}^{\tau'^*}}.
\end{align*}
Equations \eqref{equ-8-10} and \eqref{equ-8-11} state that $F$ and $G$ are inverses of each other in the incidence algebra, that is,  
\begin{equation}  \label{equ-F-G-inverse}
(F * G)(\tau, \sigma) = (G * F)(\tau, \sigma) = \delta(\tau, \sigma). 
\end{equation}
\end{corollary} 

\begin{proof}
Firstly, to prove \eqref{equ-8-10}, we can assume without loss of generality that $\tau=0$. 
Equation \eqref{equ-8-10} is then an immediate consequence of \eqref{equ-Delta-tilde-sigma} 
when we let $\Delta = \{0\}$. To obtain \eqref{equ-8-11}, we apply \eqref{equ-8-10} to $\sigma^\vee$ 
in place of $\sigma$. Finally, \eqref{equ-F-G-inverse} is a rewriting of \eqref{equ-8-10} and \eqref{equ-8-11} 
using the language of incidence algebra.
\end{proof}

\begin{corollary}   \label{cor-Gamma}
With notation as before, we have 
\begin{equation}  \label{equ-equ-Delta-tilde-tangent-cone}
{\mathbf 1}_{T^-_{\Delta, \sigma}} = \sum_{\tau \preceq \sigma} (-1)^{\dim \tau} {\mathbf 1}_{C_{\sigma}^{\tau}} \, \gamma_{\Delta, \tau}.
\end{equation}
\end{corollary}
\begin{proof}
Let $H$ and $L$ be elements of the incidence algebra such that 
$H(0, \tau) = {\mathbf 1}_{T^-_{\Delta, \sigma}}$ and $L(0, \tau) = \gamma_{\Delta, \tau}$, $\forall \tau \preceq \sigma$. 
Then \eqref{equ-Delta-tilde-sigma} states that $L(0, \tau) = (H * F)(0, \tau)$.  
Convolution of both sides from right with $G$ gives 
$(L * G)(0, \tau) = H(0, \tau)$, which is exactly \eqref{equ-equ-Delta-tilde-tangent-cone}.  
\end{proof}

\subsection{Proof of polynomiality} \label{ProofPoly}
\begin{proof}[Proof of Theorem \ref{th-polynomiality}]
In the definition of $J_{\Sigma}(\Delta)$, we use Corollary \ref{cor-Gamma} to write 
$T^-_{\Delta, \sigma}$ as $\sum\limits_{\tau \preceq \sigma} (-1)^{\dim \tau} {\mathbf 1}_{C_{\sigma}^{\tau}} \, \gamma_{\Delta, \tau}$.  
We have 
\begin{align*}
J_\Sigma(\Delta) 
&= \int\limits_{V} \sum_{\sigma \in \Sigma} (-1)^{\dim \sigma} K_\sigma(x) \, 
{\mathbf 1}_{T^-_{\Delta, \sigma}}(x) dx \\
&= \int\limits_{V} \sum_{\sigma \in \Sigma} (-1)^{\dim \sigma} K_\sigma(x) 
\left( \sum_{\tau: \tau \preceq \sigma} (-1)^{\dim \tau} 
{\mathbf 1}_{C_{\sigma}^{\tau}}(x) \, \gamma_{\Delta, \tau}(x) \right) dx \\
&= \sum_{\tau \in \Sigma} (-1)^{\dim \tau} \int\limits_{V} 
\left( \sum_{\sigma: \tau \preceq \sigma} (-1)^{\dim \sigma} K_\sigma(x) \, 
{\mathbf 1}_{C_{\sigma}^{\tau}}(x)\, \gamma_{\Delta, \tau}(x) \right) dx\\
\end{align*}
Now we use the assumption that $K_\sigma(x)$ is invariant along $\sigma$ and $\gamma_{\Delta, \tau}$ 
is invariant along $\tau^\perp$ (by definition of $\gamma_{\Delta, \tau}$) 
to write the above as 
$$\sum_{\tau} (-1)^{\dim \tau} 
\left( \, \int\limits_{\tau^\perp} \sum_{\sigma: \tau \preceq \sigma} (-1)^{\dim \sigma} 
K_\sigma(x_2)\, {\mathbf 1}_{C_{\sigma}^{\tau}}(x_2) dx_2 \right) 
\cdot 
\left( \, \int\limits_{\Span(\tau)} \gamma_{\Delta, \tau}(x_1) dx_1 \right).$$
Here $x=x_1 + x_2$ where $x_1 \in \Span(\tau)$ and $x_2 \in \tau^\perp$, and $dx_1$, $dx_2$ are the Lebesgue measures on 
$\Span(\tau)$, $\tau^\perp$ respectively so that $dx=dx_1 dx_2$.  
By Theorem \ref{th-vol-poly} and Remark \ref{rem-vol-virt-poly} we know that 
$$\vol(\Gamma_{\Delta, \tau}) = \int\limits_{\Span(\tau)} \gamma_{\Delta, \tau}(x_1) \, dx_1$$ 
is a polynomial in the support numbers of $\Gamma_{\Delta, \tau}$ of degree $\dim \tau$.  
By definition (see \eqref{equ-Gamma-Delta-sigma}) these support numbers either correspond  to the $b_i$ in which case 
they are $0$, or they correspond to the $w_i$ in which case they are equal to the $a_i$, the corresponding support numbers of $\Delta$.  
It follows that $\vol(\Gamma_{\Delta, \tau})$ is a polynomial in the support numbers of $\Delta$ of degree $\dim \tau$. 
Recall that the normal fan of the face of $\Delta$ corresponding to $\tau$ is the fan $\Sigma / \tau$ 
consisting of all the images of the cones $\sigma \succeq \tau$ in the quotient vector space $V / \Span(\tau) \cong \tau^\perp$.  
One then observes that 
$\int\limits_{\tau^\perp} \sum\limits_{\sigma: \tau \preceq \sigma} (-1)^{\dim \sigma} K_\sigma(x)\, 
{\mathbf 1}_{C_{\sigma}^{\tau}}(x) \, dx_2$ 
is exactly $J_{\Sigma/\tau}(0)$. 
In summary 
$$J_\Sigma(\Delta) = \sum_{\tau \in \Sigma} (-1)^{\dim \tau} J_{\Sigma/\tau}(0) \vol(\Gamma_{\Delta, \tau}).$$
This shows that $J_\Sigma(\Delta)$ is a linear combination of the polynomials $\vol(\Gamma_{\Delta, \tau})$ 
and hence is a polynomial itself. 
It remains to show that $J_{\Sigma/\tau}(0)$ is convergent. But this is the content of 
Lemma \ref{lem-conv-Sigma/tau} and the proof is finished. 
\end{proof}

\begin{proof}[Proof of Theorem \ref{th-polynomiality-discrete}]
In the proof of  Theorem \ref{th-polynomiality} replace any integral $\int\limits_A f(x)dx$ with a sum 
$\sum\limits_{m \in A \cap M}  f(m)$. In  particular, replace $\vol$ with the number of lattice points. 
For $\tau \in \Sigma$, let $M_1 = \Span(\tau) \cap M$ and $M_2 = \tau^\perp \cap M$. Note that it is possible that 
$M_1 + M_2 \neq M$. Nevertheless $M_1+M_2$ is a subgroup of finite index in $M$. Let $M' \subset M$ 
be a system of coset representatives for $M/ (M_1+M_2)$. Then every $m \in M$ can be uniquely written as 
$m'+m_1+m_2$ where $m' \in M'$, $m_i \in M_i$.  Then similar to the proof of Theorem \ref{th-polynomiality} 
we write 
\begin{align*}
S_\Sigma(\Delta, M) &= \sum_{m \in M} \sum_{\sigma \in \Sigma} (-1)^{\dim \sigma} K_\sigma(m) \, {\mathbf 1}_{T^-_{\Delta, \sigma}}(m) 
\\
&= \sum_{m \in M} \sum_{\sigma \in \Sigma} (-1)^{\dim \sigma} K_\sigma(m) 
\left(\sum_{\tau: \tau \preceq \sigma} (-1)^{\dim \tau} {\mathbf 1}_{C_{\sigma}^{\tau}}(m) \, \gamma_{\Delta, \tau}(m)\right) 
\\
&= \sum_{\tau \in \Sigma} (-1)^{\dim \tau} \sum_{m \in M} 
\left(\sum_{\sigma: \tau \preceq \sigma} (-1)^{\dim \sigma} K_\sigma(m)\, {\mathbf 1}_{C_{\sigma}^{\tau}}(m)\, 
\gamma_{\Delta, \tau}(m)\right) 
\\
&= \sum_{\tau} (-1)^{\dim \tau} \sum_{m' \in M'} 
\left( \sum_{m_2 \in M_2} \sum_{\sigma: \tau \preceq \sigma} (-1)^{\dim \sigma} K_\sigma(m'+m_2)\, 
{\mathbf 1}_{C_{\sigma}^{\tau}}(m'+m_2) \right) 
\cdot 
\left( \sum_{m_1 \in M_1} \gamma_{\Delta, \tau}(m'+m_1) \right).
\end{align*}
One shows that, for fixed $m' \in M'$, the quantity $\sum\limits_{m_2 \in M_2} 
\sum\limits_{\sigma: \tau \preceq \sigma} (-1)^{\dim \sigma} K_\sigma(m'+m_2)$ is equal to $S_{\Sigma/\tau}(0)$ 
with respect to the functions $K_\sigma(m'+x)$ (instead of $K_\sigma(x)$). 
By the discrete version of Lemma \ref{lem-conv-Sigma/tau}, we know that $S_{\Sigma/\tau}(0)$ is convergent.
Let us see that the other term $\sum\limits_{m_1 \in M_1} \gamma_{\Delta, \tau}(m'+m_1)$ depends polynomially 
on $\Delta$. Let $\pi: V \to \Span(\tau)$ be the orthogonal projection. Since $\gamma_{\Delta, \tau}$ is invariant 
in the $\tau^\perp$ direction we have $\gamma_{\Delta, \tau}(m'+m_1) = \gamma_{\Delta, \tau}(\pi(m')+m_1)$.
Now the polynomiality of $\sum\limits_{m_1 \in M}  \gamma_{\Delta, \tau}(\pi(m')+m_1)$ follows from 
Remark \ref{rem-vol-virt-poly} (see also Theorem \ref{th-McMullen} and Remark \ref{rem-McMullen}).  
Thus $S_\Sigma(\Delta, M)$ is a finite sum (over $m' \in M'$) of polynomials and hence a polynomial itself.  
This finishes the proof. 
\end{proof}

\section{Toric varieties}  \label{sec-toric-var}
\subsection{Background on toric varieties}   \label{subsec-toric-background}
In this section we review some basic  facts about toric varieties. Common references on  toric varieties are 
\cite{Fulton, cls}. 
Let $T = T_N \cong (\C^*)^n$ be an algebraic torus of dimension $n$ over $\C$, with character lattice 
$M \cong \Z^n$ and cocharacter lattice $N \cong \Z^n$. We denote the corresponding vector spaces 
$N \otimes_\Z \R$ and $M \otimes_\Z \R$ by $N_\R$ and $M_\R$ respectively.
For $m \in M$ we denote the corresponding character/irreducible representation by $\chi_m: T \to \C^*$.

Let $\sigma \subset N_\R$ be a rational strongly convex polyhedral cone. Recall that $\sigma$ is rational 
if it is generated as a cone by vectors from $N$. 
To $\sigma$ one associates an affine toric variety $U_\sigma$ defined by 
$$U_\sigma = \Spec(\C[\sigma^\vee \cap M].$$  
Here $\C[\sigma^\vee \cap M]$ is the semigroup algebra of the semigroup of all lattice points in the dual cone 
$\sigma^\vee$. If $\tau \preceq \sigma$ then we have natural inclusion $U_\tau \hookrightarrow U_\sigma$.  
The variety $U_0$ associated to the origin is just the algebraic torus $T$ itself. The $M$-grading on the algebra 
$\C[\sigma^\vee \cap M]$ induces a $T$-action on the variety $U_\sigma$ with open orbit $U_0$.

Recall that a fan $\Sigma$ in $N_\R$ is \emph{rational} if all the cones in $\Sigma$ are generated by vectors in $N$. 
Let $X_\Sigma$ be the toric variety corresponding to a complete rational fan $\Sigma$ (see \cite[Chap. 3]{cls} 
for more details). The (abstract) variety $X_\Sigma$ 
is obtained by gluing all the affine toric varieties $U_\sigma$, $\sigma \in \Sigma$, with respect to inclusion maps 
$U_\tau \hookrightarrow U_\sigma$, $\tau \preceq \sigma$.

There  is an inclusion-reversing correspondence between the cones in $\Sigma$ and the $T$-orbits in $X_\Sigma$. 
For $\sigma \in \Sigma$ let the corresponding $T$-orbit be $O_\sigma$.

For a ray $\rho \in \Sigma(1)$ we denote the corresponding $T$-orbit closure $\overline{O}_\rho$ by $D_\rho$. 
The $D_\rho$, $\rho \in \Sigma(1)$, are $T$-invariant prime divisors on $X_\Sigma$. For each ray  
$\rho \in \Sigma(1)$ let $v_\rho \in N$ be the  primitive vector along $\rho$, i.e. shortest lattice vector on 
$\rho$. Let $\xi \in \sigma \cap N$ be a cocharacter. One knows that for $x \in U_0$, $\lim_{t \to 0} \xi(t) \cdot x$ 
exists and is a point in the orbit $O_\sigma$. 

Let us assume that $X_\Sigma$ is a projective variety. This is equivalent to the set $\mathcal{P}(\Sigma)$, 
of polytopes with normal fan $\Sigma$, being nonempty. Let $\Delta \subset M_\R$ be a lattice polytope with 
normal fan $\Sigma$. The faces of $\Delta$ are in one-to-one correspondence with cones in $\Sigma$.  
For $\sigma \in \Sigma$ let $Q_\sigma$ be the corresponding face of $\Delta$. We note that 
$\dim Q_\sigma = \codim \sigma$. The polytope $\Delta$ can be represented as 
\begin{equation}  \label{equ-Delta}
\Delta =  \{ x \in  M_\R :  \langle x, v_\rho \rangle \leqslant -a_\rho, \forall \rho  \in  \Sigma(1)  \},
\end{equation}
where the $a_\rho$ are the support numbers of $\Delta$ (see Section  \ref{sec-prelim}).
Recall that for $\sigma \in \Sigma$ we let $T^+_{\Delta, \sigma}$ (respectively $T^-_{\Delta, \sigma}$) 
be the inward looking  (respectively outward looking) tangent cone of the corresponding face 
$Q_\sigma$ in $\Delta$ (see Equations \eqref{inward-face-cone} and \eqref{outward-face-cone}).

To $\Delta$ one associates a $T$-invariant (Cartier) divisor 
$$D_\Delta = \sum_{\rho \in \Sigma(1)} -a_\rho D_\rho.$$
It can be shown that $D_\Delta$ is an ample divisor. We denote the corresponding line bundle on 
$X_\Sigma$ by $\mathcal{L}_\Delta$. Since $D_\Delta$ is $T$-invariant, the line bundle 
$\mathcal{L}_\Delta$ comes with a natural $T$-linearization. The divisor $D_\Delta$ defines a sheaf of 
rational functions $\cO(D_\Delta)$ by 
\begin{align}   \label{equ-O-Delta}
H^0(U, \cO(D_\Delta)) &= \{ f \in \C(X_\Sigma) : (f) + D_\Delta > 0 \text{ on } U \} \subset \C[U_0],\\
&= \{ f \in \C(X_\Sigma) : \operatorname{ord}_{D_\rho}(f) \geqslant a_\rho,~ \forall \rho \in \Sigma(1) 
\textup{ such that } D_\rho \cap U \neq \emptyset\}.
\end{align} 
In particular, for an open affine chart $U_\sigma$, the subspace $H^0(U_\sigma, \cO(D_\Delta))$ is 
$T$-invariant and hence decomposes into 
one dimensional $T$-modules. Let $m \in M$. One verifies that for any ray $\rho \in \Sigma(1)$, 
the order of zero/pole of the character $\chi_m$,  regarded as a rational function on $U_0 \cong T$, 
along the divisor $D_\rho$ is given by   
$$\operatorname{ord}_{D_\rho}(\chi_m) = - \langle m, v_\rho \rangle.$$ It follows that, for any 
$\sigma \in \Sigma$, the irreducible representation $\chi_m$ appears in $H^0(U_\sigma, \cO(D_\Delta))$ 
if and only if $\langle m, v_\rho \rangle \leqslant -a_\rho$, for all $\rho \in \sigma(1)$. 
Since $\C[U_0]$, the coordinate ring of the algebraic torus, is multiplicity-free as a $T$-module it follows that 
$H^0(U_\sigma, \cO(D_\Delta))$ is also multiplicity-free. 
Thus the $T$-module $H^0(U_\sigma, \cO(D_\Delta))$ decomposes into one dimensional irreducible representation as 
\begin{equation}  \label{equ-H^0-U_sigma}
H^0(U_\sigma, \cO(D_\Delta)) = \bigoplus_{m \in T^+_{\Delta, \sigma} \cap M} \chi_m,
\end{equation}
where as before $T^+_{\Delta, \sigma}$ denotes the inward looking tangent cone of $\Delta$ at the face corresponding to  $\sigma$. 
Similarly, $\chi_m$ appears in the space of global sections  $H^0(X_\Sigma, \cO(D_\Delta))$ if and only if 
$\langle m, v_\rho \rangle \leqslant -a_\rho$, for all $\rho \in \Sigma(1)$ and we have 
\begin{equation}  \label{equ-H^0}
H^0(X_\Sigma, \cO(D_\Delta)) = \bigoplus_{m \in \Delta \cap M} \chi_m.
\end{equation}
This implies that $\dim(H^0(X_\Sigma, \cO(D_\Delta)) = |\Delta \cap M|$, the number of lattice points in $M$.

\subsection{Brianchon-Gram theorem and equivariant Euler characteristic}   \label{subsec-equiv-Euler-char}
Let $\F$ be a $T$-linearized sheaf (of rational functions) on $X_\Sigma$, that is for any $T$-invariant  open set $U$, 
the space of sections $H^0(U, \F)$ is a $T$-module and the restriction maps are $T$-equivariant. For $m \in M$ and 
$V$ a $T$-module let $V_m$ denote the $m$-isotypic component of $V$. 
By the \textit{equivariant Euler characteristic} of $\F$ we mean the function $\chi_T(X_\Sigma, \F): M \to \Z_{\geq 0}$ given by 
$$\chi_T(X_\Sigma, \F)(m) = \sum_{i=0}^n (-1)^i \dim(H^i(X_\Sigma, \F)_m).$$

Let us compute the equivariant Euler characteristic of the $T$-linearized sheaf $\cO(D_\Delta)$. As explained above, 
for each cone $\sigma \in \Sigma$ the $T$-module $H^0(U_\sigma, \cO(D_\Delta))$ decomposes as 
$$H^0(U_\sigma, \cO(D_\Delta)) = \bigoplus_{m \in T^+_{\Delta, \sigma} \cap M} \chi_m.$$
Recall that $T^+_{\Delta, \sigma}$ denotes the inward tangent cone of $\Delta$ at the face corresponding to 
$\sigma$ (see Section \ref{subsec-polytope}).

From above it follows that the equivariant Euler characteristic $\chi_T(X_\Sigma, \cO(D_\Delta))$, 
computed using \v{C}ech cohomology, can be written as:
\begin{equation}  \label{equ-equi-Euler-char-O(D)}
\chi_T(X_\Sigma, \cO(D_\Delta)) = \sum_{\sigma \in \Sigma} (-1)^{\dim(Q_\sigma)} {\mathbf 1}_{T^+_{\Delta, \sigma} \cap M},
\end{equation}
where as usual ${\mathbf 1}_A$ denotes the characteristic function of a set $A$.

One knows that $\cO(D_\Delta)$ is ample and hence $H^i(X_\Sigma, \cO(D_\Delta)) = 0$ for $i > 0$.  
Thus we also obtain 
\begin{equation}    \label{equ-Euler-char-H^0}
\chi_T(X_\Sigma, \cO(D_\Delta))(m) = \dim(H^0(X_\Sigma, \cO(D_\Delta))_m), \quad \forall m \in M. 
\end{equation}
And hence from \eqref{equ-H^0} we have 
\begin{equation}   \label{equ-Euler-char-O(D)}
\chi_T(X_\Sigma, \cO(D_\Delta)) = {\mathbf 1}_{\Delta \cap M}.
\end{equation}
Comparing with \eqref{equ-equi-Euler-char-O(D)} one recovers the Brinachon-Gram theorem (Theorem \ref{th-BG}).

The alternative version of the Brianchon-Gram using outward face cones (Theorem \ref{th-BG-alt}) can also 
be obtained in a similar fashion.
Let $\Delta'$ be the polytope with support numbers $a_\rho + 1$ and $D' = D_{\Delta'} = 
\sum\limits_{\rho \in \Sigma(1)} -(a_\rho + 1) D_\rho$ 
the corresponding Cartier divisor.  
Note that $\langle x, v_\rho \rangle \leqslant -(a_\rho + 1)$ if and only $\langle -x, v_\rho \rangle>a_\rho$. 
Thus for all $m \in M$ we have 
\begin{equation} \label{equ-Euler-char-O(-D')}
\chi_T(X_\Sigma, \cO(-D'))(m) = \sum_{\sigma \in \Sigma} (-1)^{n -  \dim \sigma} 
{\mathbf 1}_{T^-_{\Delta, \sigma} \cap M}(-m)
\end{equation} 
(recall \eqref{outward-face-cone} for defining inequalities of outward tangent cone  $T^-_{\Delta, \sigma}$).
On the other hand, the Khovanskii-Pukhlikov formula for inverse of the polytope $\Delta$ with respect 
to the convolution $*$ (see Section \ref{sec-convex-chains}) tells us that:
\begin{equation}   \label{equ-Ehrhart-recip}
\chi_T(X_\Sigma, \cO(-D'))(m) = (-1)^n \chi_T(X_\Sigma, \cO(D_\Delta))(-m) = (-1)^n {\mathbf 1}_{\Delta \cap  M}(-m),  
\end{equation}
Putting together \eqref{equ-Euler-char-O(-D')} and \eqref{equ-Ehrhart-recip} we obtain 
$$(-1)^n {\mathbf 1}_{\Delta \cap  M} = \sum_{\sigma \in \Sigma} (-1)^{n - \dim \sigma} 
{\mathbf 1}_{T^-_{\Delta, \sigma} \cap M}$$
which  immediately implies Theorem \ref{th-BG-alt}.

\begin{remark}[A symplectic interpretation of the Brianchon-Gram theorem]
We can also give a symplectic geometric interpretation of the Brianchon-Gram theorem, 
namely as an identity between Liouville measures.  
Let $X$ be a symplectic manifold with a Hamiltonian $S^1$-action with moment map $\mu: X \to \R$.  
This means that the Hamiltonian vector field of $\mu$ generates the $S^1$-action.  
Let $\epsilon$ be a regular value of the moment map $\mu$. Then $\mu^{-1}([\epsilon, \infty))$ 
is a manifold with boundary.  The \emph{symplectic cut} $\overline{X}_{\mu \geq \epsilon}$ 
is the manifold obtained by collapsing each $S^1$-orbit in the boundary $\mu^{-1}(\epsilon)$ to a point.

We can decompose $T = (\C^*)^n$ as $T = (S^1)^n \times \R_{>0}^n$. Equip $T$ with the standard symplectic form from $\C^n$.  
Each ray $\rho \in \Sigma(1)$ defines a Hamiltonian function $\mu_\rho: U_0 \to \R$ on $U_0  \cong T$ by 
$$\mu_\rho(x) = |x|^{v_\rho} := |x_1|^{r_1} \cdots |x_n|^{r_n},$$
where $x=(x_1, \ldots, x_n)$ and $v_\rho = (r_1, \ldots, r_n)$.
One verifies that the Hamiltonian vector field of $\mu_\rho$ generates the $\C^*$-action on $T$ corresponding to the 
cocharacter $v_\rho \in N$. Let $\Sigma$ be a smooth fan and let $\Delta$ be a rational polytope with normal fan $\Sigma$ and let 
$a_\rho$, $\rho \in \Sigma(1)$, be its support numbers. Starting with $(\C^*)^n$, doing repeated symplectic cuts with respect to the 
$\mu=\mu_\rho$ and $\epsilon=a_\rho$, $\rho \in \Sigma(1)$, one arrives at the toric variety $X_\Sigma$. One can show that 
the open affine chart $U_\sigma$ is the symplectic manifold obtained by symplectic cuts using rays of $\sigma$.  
Moreover, the image of the moment map of $U_\sigma$ is the inward tangent cone $T^+_{\Delta, \sigma}$.

The Brianchon-Gram equality \eqref{equ-BG} can be thought of as an equality involving pushforwards (to $N_\R = \R^n$) of 
Liouville measures on all the symplectic manifolds $U_{\sigma}$ and $X_\Sigma$. 
\end{remark}

\subsection{Positive part of a toric variety and logarithm map}   \label{subsec-toric-positive}
As before let $X_\Sigma$ be the toric variety associated to a rational fan $\Sigma$ in $N_\R$. 
Take $\sigma \in \Sigma$. By definition the set $U_\sigma(\C)$ of points of $U_\sigma$ defined over 
$\C$ is the set of maximal ideals of the semigroup 
algebra $\C[\sigma^\vee \cap M]$. This set then can be identified with $\Hom(\sigma^\vee \cap M, \C)$ 
where $\Hom$ denotes the semigroup homomorphisms. 
This observation enables us to construct $X_\Sigma^+$, the points of $X_\Sigma$ over the semigroup 
$\R_{\geq 0}$ (see \cite[Section 4.1]{Fulton}). 
We think of $X_\Sigma^+$ as the ``positive'' part of $X_\Sigma(\C)$. It is constructed as follows. For each 
$\sigma \in \Sigma$ let $U_\sigma^+ = \Hom(\sigma^\vee \cap M, \R_{\geq 0})$. Then, as before the $U_\sigma^+$ 
glue together to give $X_\Sigma^+$. One has natural inclusion $X_\Sigma^+ \hookrightarrow X_\Sigma(\C)$.  
Moreover, the absolute value $| \cdot |: \C \to \R_{\geq 0}$ induces a retraction map $X_\Sigma(\C) \to X_\Sigma^+$.  
Let $T_K = (S^1)^n$ denote the usual compact torus which is the maximal compact subgroup of $T$.  
One verifies that the retraction map induces a homeomorphism between the quotient $X_\Sigma(\C) / T_K$ and $X_\Sigma^+$.

Another way to look at $X_\Sigma^+$ is as follows. 
Consider the \emph{logarithm map} 
\[ 
\Log: T_N = (\C^*)^n \longrightarrow N_\R = \Hom(M, \R) 
\] 
defined as follows. 
For $z \in T$ and $m \in M$ let  
\begin{equation}   \label{equ-log-map}
\Log(z)(m) = \log( |\chi_m(z)|).
\end{equation}
In the standard coordinates for $(\C^*)^n$ the logarithm map is given by 
\begin{equation} \label{equ-log-map2}
\Log(z_1, \ldots, z_n) = (\log|z_1|, \ldots, \log|z_n|).
\end{equation}
For each $\sigma \in \Sigma$ the orbit $O_\sigma$ can be identified with $T / T_\sigma$ where 
$T_\sigma$ is the $T$-stabilizer of $O_\sigma$. 
Let  $N_\sigma$ denote the cocharacter lattice of $T_\sigma$. It follows from the definitions that 
$N_\sigma \otimes \R = \Span(\sigma)$. 
The logarithm map then induces a map $\Log_\sigma: T/T_\sigma  \to N_\R / \Span(\sigma)$. In the same way, 
that $X_\Sigma(\C)$ is a disjoint union of the tori $O_\sigma$, $\sigma \in \Sigma$, the positive part $X_\Sigma^+$, is a disjoint union 
of the real vector spaces $N_\R / \Span(\sigma)$, $\sigma \in \Sigma$.

Finally $X_\Sigma^+$ is actually homeomorphic to a polytope (in a non-unique way). Given a polytope 
$\Delta$ with normal fan $\Sigma$, 
one can construct explicitly a $T_K$-invariant continuous map $\mu: X_\Sigma \to \Delta$ such that the induced map 
$\bar{\mu}: X_\Sigma / T_K \to \Delta$ is a homeomorphism and the following diagram is commutative (see \cite[Section 4.2]{Fulton}). 
\begin{equation}  \label{equ-moment-log-diagram}
\begin{tikzcd}
& (\C^*)^n \cong U_0 \arrow[hookrightarrow]{r} \arrow{d}{\Log} \arrow{dl}[swap]{\Log} & 
X_\Sigma(\C) \arrow{d}{\Log} \arrow{dr}{\mu} & \\
N_\R & \R^n \cong U_0^+ \arrow{l}{\cong} \arrow[hookrightarrow]{r} & X_\Sigma^+ \arrow{r}{\bar{\mu}}[swap]{\cong} & \Delta  
\end{tikzcd}
\end{equation}
Moreover, the bottom row gives a homeomorphism between $N_\R$ and the interior $\Delta^\circ$ of $\Delta$.
The map $\mu$ is a special case of the notion of \emph{momentum map} from the theory of 
Hamiltonian group actions in symplectic geometry.

\section{Geometric interpretations of combinatorial truncation}  \label{sec-geo-comb-truncation}
We propose two geometric interpretations of our combinatorial truncation in terms of geometric notions on toric varieties. 
The same ideas should extend to give geometric interpretations of Arthur's truncation and modified kernel. 
We expect that in this case one should replace a toric variety $X_\Sigma$ by Mumford's compactificaiton of a reductive 
algebraic group as in \cite[Section IV.2]{Mumford}.

\subsection{Combinatorial truncation as a complex measure on a toric variety}  \label{sec-measure}
In this section we propose that combinatorial truncation can be interpreted as a ``truncated'' complex measure on 
a projective toric variety, obtained from the data of 
prescribed measures on each torus orbit as well as choice of a polytope normal to the fan which determines 
certain neighborhoods of the torus orbits.

As usual let $X_\Sigma$ be the toric variety associated to a (rational) fan $\Sigma$ in $N_\R$.
Recall that the starting data of combinatorial truncation is a collection of functions $\{K_\sigma: N_\R \to \C: \sigma \in \Sigma\}$, 
where each $K_\sigma$ is invariant in the direction of $\Span(\sigma)$.

As before let $T_K = (S^1)^n$ denote the compact torus in $T = (\C^*)^n$ which is the maximal compact subgroup of $T$.  
Suppose we are given a $T_K$-invariant complex measure $\omega_0 = f_0 d\mu_0$ on $U_0 = T$ 
where $f_0$ is a continuous function 
on $U_0$ and $d\mu_0$ denotes a Haar measure on $U_0$. Moreover, suppose for each $\{0\} \neq \sigma \in \Sigma$ we have a 
$T_K$-invariant complex measure $\omega_\sigma = f_\sigma d\mu_\sigma$ on the torus orbit 
$O_\sigma$, the $T$-orbit in $X_\Sigma$ 
associated to $\sigma$. Here $f_\sigma$ is a continuous function on $O_\sigma$ and 
$d\mu_\sigma$ is the Haar measure on $O_\sigma$ 
induced from $d\mu_0$. Recall that $O_\sigma \cong T / T_\sigma$ is itself isomorphic to a torus, 
where $T_\sigma \subset T$ is the stabilizer of any point in $O_\sigma$. 
Since $\omega_\sigma$, and hence $f_\sigma$, are $T_K$-invariant, the function $f_\sigma$ induces a continuous function 
$k_\sigma: N_\R / \Span(\sigma) \longrightarrow \C$.

The projection $N_\R \to N_\R/\Span{\sigma}$ maps the cone $\sigma$ to $\{0\}$. This gives us an equivariant morphism 
$\pi_\sigma$ from the $T$-toric variety $U_\sigma$ to the $(T/T_\sigma)$-toric variety $O_\sigma$ (see \cite[Sec. 3.3]{cls}).  
We can use $\pi_\sigma: U_\sigma \to O_\sigma$ to extend the measure $\omega_\sigma$ to a measure $\Omega_\sigma$ 
on the affine toric chart $U_\sigma \subset X_\Sigma$ (and in particular, on the open orbit $U_0 \cong T$) by defining  
$$\Omega_\sigma = \pi_\sigma^*(\omega_\sigma).$$
The measure $\Omega_\sigma$ then gives a continuous function $K_\sigma: N_\R \to \C$ 
which is invariant in the direction of $\Span(\sigma)$.

Now fix an inner product $\langle \cdot, \cdot \rangle$ on $N_\R$ and identify $M_\R$ with $N_\R$ via 
$\langle \cdot, \cdot \rangle$.
As usual take a polytope $\Delta \subset M_\R \cong N_\R$ with normal fan $\Sigma$. 
Recall that $\Log: T \to N_\R$ denotes the logarithm map on the torus, which extends to 
$\Log: X_\Sigma \to X_\Sigma^+$ (see \eqref{equ-log-map} and the diagram \eqref{equ-moment-log-diagram}).  
Consider the tangent cone $T^-_{\Delta, \sigma}$.  We regard it as an open subset of $U_0^+ \cong N_\R \cong \R^n$ 
and hence as an open subset of $X_\Sigma^+$.  We have 
$$U_{\Delta, \sigma} = \Log^{-1}(T^-_{\Delta, \sigma}).$$
We can also define the subset $U_{\Delta} \subset U_0$ by 
$$U_{\Delta} = \Log^{-1}(\Delta).$$
We think of $\Omega_\sigma \, {\mathbf 1}_{U_{\sigma, \Delta}}$ as an extension of the measure 
$\omega_\sigma$ to the neighborhood $U_{\Delta, \sigma}$.
Finally, we can define a complex measure $\Omega_\Delta$ on $X_\Sigma$ by 
$$\Omega_\Delta = \sum_{\sigma \in \Sigma} (-1)^{\dim \sigma} \Omega_\sigma ~{\mathbf 1}_{U_{\Delta, \sigma}}.$$
It is a $T_K$-invariant complex measure on $X_\Sigma$ and corresponds to the function 
$k_\Delta$ on $N_\R$. We think of it as a \emph{truncation} of $\omega_0$ with respect to the measures 
$\omega_\sigma$ at infinity. 
From Theorems \ref{th-growth-conv} and \ref{th-polynomiality} we have the following.  
 
\begin{proposition}
Under the assumptions in Theorem \ref{th-growth-conv} on the functions $K_\sigma$, 
the total measure of $\Omega_\Delta$ is finite and is a polynomial in the support numbers of $\Delta$.
\end{proposition} 
 
\begin{remark}
In fact, each tangent cone $T^-_{\Delta, \sigma}$ gives us an open neighborhood of the orbit closure 
$\overline{O}_\sigma$ in $X_\Sigma$.  
To construct this open neighborhood, we complete $T^-_{\Delta, \sigma} \subset N_\R$ to an open subset 
$\tilde{T}_{\Delta, \sigma} \subset X_\Sigma^+$ containing the closure $\overline{O^+_\sigma}$ by 
$$\tilde{T}_{\Delta, \sigma} = \bigcup_{\sigma': \sigma \preceq \sigma'} 
\bigcup_{\tau: \tau \preceq \sigma'} T^-_{Q_{\tau}, \sigma'} ~\subset 
X_\Sigma^+ := \bigsqcup_{\sigma \in \Sigma} N_\R / \Span(\sigma).$$
One verifies that $\tilde{T}_{\Delta, \sigma}$ is indeed an open subset of $X_\Sigma^+$ containing $\overline{O^+_\sigma}$.  
It follows that $\tilde{U}_{\Delta, \sigma} = \Log^{-1}(\tilde{T}_{\Delta, \sigma})$
is an open neighborhood of the orbit closure $\overline{O}_\sigma$ in the toric variety $X_\Sigma$.
We note that $T^-_{\Delta, \sigma}$ is open dense in  $\tilde{T}_{\Delta, \sigma}$ and hence 
for the purposes of truncation it does not matter 
whether we work with $T^-_{\Delta, \sigma}$ or $\tilde{T}_{\Delta, \sigma}$.
\end{remark}

\subsection{Combinatorial truncation as a Lefschetz number} \label{sec-Lefschetz}
In this section we give an interpretation of the combinatorial truncation as a Lefschetz number.

\subsubsection*{Lefschetz number} 
Let $X$ be a topological space such that all its cohomology groups $H^i(X, \R)$ are 
finite dimensional and for some $n \geqslant 0$, $H^i(X, \R) = 0$, $\forall i > n$.
Let $\Phi: X \to X$ be a continuous map. Recall that the \emph{Lefschetz number} of 
$\Phi$ is defined to be 
$$\Lambda(\Phi) = \sum_{i=0}^n (-1)^i ~\operatorname{Tr}(\Phi^*: H^i(X, \R) \to H^i(X, \R)).$$

The Lefschetz number of the identity map is, by definition, equal to the Euler characteristic of $X$.  
The Lefschetz number appears in the Lefschetz fixed point theorem which states that if $X$ is a compact triangulable space and 
$\Lambda(\Phi) \neq 0$, then $\Phi$ has at least one fixed point. 

Let us define an analogue of the notion of Lefschetz number for morphisms of sheaves.  
Let  $\F$ be a sheaf of vector spaces on $X$ such that all the cohomology groups of $(X, \F)$ 
are finite dimensional and for some $n$, $H^i(X, \F) = 0$, $\forall i > n$. 
By a \emph{morphism of sheaves} $\Psi: \F \to \F$ we mean a collection of linear maps 
$\{\Psi_U: \F(U) \to \F(U) : U \subset X \textup{ open }\}$ which are compatible with the restriction maps.  
That is, for $U \subset V$ we have $$\Psi_U \circ \textup{rest}_{V, U} = \textup{rest}_{V, U} \circ \Psi_V.$$
Clearly, $\Psi$ induces linear maps $\Psi^*: H^i(X, \F)  \to H^i(X, \F)$ between the cohomology groups of $(X, \F)$.
Extending the above notion of Lefschetz number we make the following definition. 
\begin{definition}[Lefschetz number for morphisms of sheaves]  \label{def-Lefschetz-sheaf}
The \textit{Lefschetz number} $\Lambda(\Psi, \F)$ is defined to be 
$$\Lambda(\Psi, \F) = \sum_{i=0}^n (-1)^i ~\textup{Tr}(\Psi^*: H^i(X, \F) \to H^i(X, \F)).$$
\end{definition}

\begin{remark}
When $\Psi$ is the identity morphism, i.e. all the maps $\Psi_U$ are identities, then 
$\Lambda(\Psi, \F)$ is just the \emph{Euler characteristic} of the 
sheaf $\F$.
\end{remark}

Let $\mathcal{U}$ be a finite open cover of $X$. Suppose $\mathcal{U}$ is a good open cover with respect to $\F$, that is, 
$\F$ is acyclic on any intersection of the open sets in $\mathcal{U}$. It is a standard result in topology that 
the \v{C}ech cohomology groups of $(\mathcal{U}, \F)$ are independent of the choice of the good open cover and 
coincide with the sheaf cohomology groups of $(X, \F)$. 
 
Suppose the vector spaces in the \v{C}ech cochain complex $C^\bullet(\mathcal{U}, \F)$ are finite dimensional.  
In other words, for any collection of open sets $U_1, \ldots, U_k \in \mathcal{U}$ we have 
$\dim H^0(U_1 \cap \cdots \cap U_k, \F) < \infty$.  
In this case, the Lefschetz number can be computed in terms of the traces of the vector spaces in the cochain complex 
$C^\bullet(\U, \F)$ as well. This straightforward result is sometimes referred to as the Hopf trace formula. 
\begin{proposition}   \label{prop-Hopf}
With assumptions as above, the Lefschetz number can be computed as 
$$\Lambda(\Psi) = \sum_{i=0}^n (-1)^i ~\textup{Tr}(\Psi^*: C^i(\mathcal{U}, \F) \to C^i(\mathcal{U}, \F)),$$ 
where $C^i(\mathcal{U}, \F)$ denotes the vector space of $i$-th \v{C}ech cochains of $\mathcal{U}$ with coefficients in $\F$.    
\end{proposition}

Similarly, suppose $X$ is equipped with a measure and $\F$ a sheaf of $L^2$-functions on $X$ and let 
$\Psi: \F \to  \F$ be a morphism of sheaves. 
Moreover, suppose for every open set $U$, the linear operator $\Psi: \F(U) \to \F(U)$ is a trace class operator 
with kernel function $K_U$.
Then for each $i$, the induced map $\Psi^*: H^i(X, \F) \to H^i(X, \F)$ is also a trace class operator. We denote its kernel by $T_i$.

\begin{definition}[Lefschetz number for morphisms of sheaves of $L^2$-functions]
We define the \textit{Lefschetz number} $\Lambda(\Psi, \F)$ by 
\begin{equation}  \label{equ-Lefschetz-L^2}
\Lambda(\Psi, \F) = \int\limits_X \sum_{i=0}^n (-1)^i T_i(x) \, dx.  
\end{equation}
\end{definition}

As above, let $\mathcal{U}$ be a finite open cover of $X$ which is a good cover with respect to $\F$.  
Suppose for each $i$ the operator $\Psi^*: C^i(\mathcal{U}, \F) \to C^i(\mathcal{U}, \F)$ is trace class with kernel $K_i$.  
Similarly to Proposition \ref{prop-Hopf} the Lefschetz number $\Lambda(\Psi, \F)$ can be computed as 
$$\Lambda(\Psi, \F) = \int\limits_X \sum_{i=0}^n (-1)^i K_i(x) \, dx.$$

The observation in this section is that when $X=X_\Sigma$ is a  toric variety, the Lefschetz number is given 
by a combinatorial truncation $J_\Sigma(\Delta)$. 
As usual let $\Sigma$ be a (rational) fan in $N_\R$ and let $\Delta \in \mathcal{P}(\Sigma)$ be a polytope with 
normal fan $\Sigma$.  As in Section \ref{subsec-toric-background} let  $X_\Sigma$ be the toric variety of the fan 
$\Sigma$ and $\mathcal{O}(D_\Delta)$ be the sheaf of sections of the (Cartier) divisor $D_\Delta$ associated to $\Delta$. 
Let  the $a_\rho$, $\rho \in \Sigma(1)$, be the support numbers of $\Delta$. Let $\Delta'$ be the polytope whose support numbers 
are the $a_\rho-1$. Let $\Psi: \mathcal{O}(-D_{\Delta'}) \to \mathcal{O}(-D_{\Delta'})$ be a morphism of sheaves.

Recall that the characters $\chi_m$, $m \in M$, form a vector space basis for $\C[U_0]$. Moreover, a subset of this basis 
is a basis for $\mathcal{O}(-D_{\Delta'})$. For $m \in M$, let $K_\sigma(m)$ be the $(m, m)$-entry of the matrix of the linear operator 
$\Psi_\sigma: \mathcal{O}(-D_{\Delta'})(U_\sigma) \to \mathcal{O}(-D_{\Delta'})(U_\sigma)$. The following follows from 
Section \ref{subsec-equiv-Euler-char} and in particular \eqref{equ-Euler-char-O(-D')}.

\begin{proposition}[Combinatorial truncation as a Lefschetz number on a toric variety]   \label{prop-toric-Lefschetz}
With notation as above, the Lefschetz number $\Lambda(\Psi, \mathcal{O}(-D_{\Delta'}))$ is equal to the truncated sum 
$S_\Sigma(\Delta, M)$:
$$\Lambda(\Psi, \mathcal{O}(-D_{\Delta'})) = S_\Sigma(\Delta, M) := 
\sum_{m \in M} \sum_{\sigma \in \Sigma} (-1)^{\dim \sigma} K_\sigma(m)~ {\mathbf 1}_{T^-_{\Delta, \sigma} \cap M}(m).$$ 
\end{proposition}

\begin{remark}
The reason for appearance of the polytope $\Delta'$ instead of $\Delta$ is that we defined the outward tangent cones 
$T^-_{\Delta, \sigma}$ using strict inequalities. If we change the convention and use non-strict inequalities in the definition of 
$T^-_{\Delta, \sigma}$, then Proposition \ref{equ-Euler-char-O(-D')} holds with $D$ in place of $D'$.  
\end{remark}

Finally, as a side remark we also mention an example of a presheaf that is reminiscent of Arthur's construction of the kernels 
$K_P$ (see \cite[Section 4]{arthur-clay}).

\begin{example}[A sheaf of $W$-invariant sections on the toric variety of Weyl fan]
Suppose $\Sigma$ is the Weyl fan and hence the Weyl group acts on $\Sigma$.  
Note that by definition $W$ acts on the character lattice $M$. 
For $\sigma \in \Sigma$ let $W_\sigma$ be the $W$-stabilizer of $\sigma$. 
Let $\mathcal{O}(\Delta)$ be the invertible sheaf associated to a 
$W$-invariant polytope $\Delta$. We define the sheaf $\mathcal{O}(\Delta)^W$ by 
$$H^0(U_\sigma, \mathcal{O}(\Delta)^W) := H^0(U_\sigma, \mathcal{O}(\Delta))^{W_\sigma}, \quad \forall \sigma \in \Sigma.$$
Let $\tau \subset \sigma$ be cones in $\Sigma$. Note that $W_\sigma \subset W_\tau$ and hence if 
$f \in H^0(U_\sigma, \mathcal{O}(\Delta))^{W_\sigma}$ then, in general, $f_{|U_\tau}$ may not be $W_\tau$-invariant and 
hence may not lie in $H^0(U_\tau, \mathcal{O}(\Delta))^{W_\tau}$. We remedy this by defining the restriction map 
$i_{\sigma\tau}: H^0(U_\sigma, \mathcal{O}(\Delta)^W) \to H^0(U_\tau, \mathcal{O}(\Delta)^W)$ by: 
$$i_{\sigma\tau}(f) = \sum_{w \in W_\tau/W_\sigma} (w \cdot f)_{|U_\tau}.$$  
Let us verify that the above restriction maps $i_{\sigma \tau}$ give a well-defined pre-sheaf on $X_\Sigma$.  
Suppose we have cones $\gamma \subset \tau \subset \sigma$ in $\Sigma$ with corresponding affine charts 
$U_\gamma \subset U_\tau \subset U_\sigma$. We need to show $i_{\tau \gamma} \circ i_{\sigma \tau} = i_{\sigma \gamma}$.  
Let $f \in H^0(U_\sigma, \mathcal{O}(\Delta)^W)$. We have 
$$
i_{\tau \gamma}(i_{\sigma \tau}(f)) = \sum_{w \in W_\gamma / W_\tau} \sum_{w' \in W_\tau / W_\sigma} (ww') \cdot f.  
$$
As $w$ (respectively, $w'$) runs over a set of representatives for $W_\gamma / W_\tau$ (respectively, $W_\tau / W_\sigma$), 
the product $ww'$ runs over a set of representatives for $W_\gamma / W_\sigma$. This proves the claim.
\end{example}

It is interesting to compute the Euler characteristic and \v{C}ech cohomologies of the above presheaf.


\begin{thebibliography}{10000}

\bibitem[Ag06]{Agapito} 
	J. Agapito. 
	\textit{
	Weighted Brianchon-Gram decomposition. 
	} 
	Canad. Math. Bull. 49 (2006), no. 2, 161--169
  
\bibitem[Ar78]{arthur-duke} 
	J. Arthur. 
	\textit{
	A trace formula for reductive groups. I. Terms associated to classes in $G(\Q)$.
	} 
	 Duke Math. J. 45 (1978), no. 4, 911--952. 
	 
\bibitem[Ar81]{arthur-annals} 
	J. Arthur. 
	\textit{
	The trace formula in invariant form. 
	}
	Ann. of Math. (2) 114 (1981), no. 1, 1--74.	
	
\bibitem[Ar05]{arthur-clay} 
	J. Arthur. 
	An introduction to the trace formula. 
	In 
	\textit{
	Harmonic analysis, the trace formula, and Shimura varieties,
	} 
	1--263, Clay Math. Proc., 4, Amer. Math. Soc., 2005.

\bibitem[BHS09]{Beck-Haase-Sottile}
	M. Beck, C. Haase and F. Sottile.
	\textit{
	Formulas of Brion, Lawrence, and Varchenko on rational generating functions for cones.
	} 
	Math. Intelligencer 31 (2009), no. 1, 9--17.

\bibitem[B37]{Brianchon}
	C. J. Brianchon. 
	\textit{
	Th\'eor\`eme nouveau sur les poly\`edres. 
	} 
	J. Ecole (Royale) Polytechnique 15 (1837), 317--319.

\bibitem[Br88]{Brion}
	M. Brion. 
	\textit{
	Points entiers dans les poly\`edres convexes. 
	} 
	Ann. Sci. \'Ecole Norm. Sup. 21 (1988), no. 4, 653--663.

\bibitem[BV97]{Brion-Vergne} 
	M. Brion and M. Vergne. 
	\textit{
	Residue formulae, vector partition functions and lattice points in rational polytopes. 
	} 
	J. Amer. Math. Soc. 10 (1997), 797--833 

\bibitem[Cass04]{casselman} 
	W. Casselman. 
	Truncation exercises. 
	In 
	\textit{
	Functional analysis VIII.  
	}
	84--104, Various Publ. Ser. (Aarhus), 47, Aarhus Univ., Aarhus, 2004. 

\bibitem[CLS11]{cls} 
	D. A. Cox, J. B. Little, and H. K. Schenck. 
	\textit{
	Toric varieties. 
	} 
	Graduate Studies in Mathematics, 124. 
	Amer. Math. Soc., 2011.

\bibitem[FL11]{Finis-Lapid}
	T. Finis and E. Lapid.
	\textit{
	On the spectral side of Arthur's trace formula—combinatorial setup. 
	}  
	Ann. of Math. (2) 174 (2011), no. 1, 197--223.

\bibitem[FLM11]{Finis-Lapid-Mueller}
	T. Finis and E. Lapid, W. M\"uller.
	\textit{
 	On the spectral side of Arthur's trace formula--absolute convergence. 
	}
	Ann. of Math. (2) 174 (2011), no. 1, 173--195. 

\bibitem[FL16]{Finis-Lapid-Vietnam}
	T. Finis and E. Lapid.
	\textit{
On the continuity of the geometric side of the trace formula. 
	}  
Acta Math. Vietnam. 41 (2016), no. 3, 425--455. 



\bibitem[Fu93]{Fulton}
	W. Fulton.
	\textit{
	Introduction to toric varieties. 
	} 
	Annals of Mathematics Studies, 131,  
	The William H. Roever Lectures in Geometry,  
	Princeton University Press, Princeton, NJ, 1993.
	
\bibitem[GKM97]{gkm} 
	M. Goresky, R. Kottwitz, and  R. MacPherson. 
	\textit{
	Discrete series characters and the Lefschetz formula for Hecke operators. 
	} 
	Duke Math. J. 89 (1997), no. 3, 477--554.
	Correction: Duke Math. J. 92 (1998), no. 3, 665--666. 
	
\bibitem[G1874]{Gram}	
	J. P. Gram. 
	\textit{
	Om rumvinklerne i et polyeder. 
	} 
	Tidsskrift for Math. (Copenhagen) 4 (1874), no. 3, 161--163.	
	
\bibitem[Hass05]{Haase} 
	C. Haase. 
	Polar decomposition and Brion's theorem. 
	In 
	\textit{
	Integer points in polyhedra--geometry, number theory, algebra, optimization. 
	}
	 Contemp. Math., 374, 91--99, Amer. Math. Soc., 2005.


\bibitem[Hoff08]{Hoffmann} 
	W. Hoffmann.
	\textit{
	Geometric estimates for the trace formula. 
	} 
	Ann. Global Anal. Geom. 34 (2008), no. 3, 233--261. 


\bibitem[KKMS73]{Mumford}
	G. Kempf, F. F. Knudsen, D. Mumford, and B. Saint-Donat.
	\textit{
	Toroidal embeddings. I. 
	} 
	Lecture Notes in Mathematics, Vol. 339. Springer-Verlag, Berlin-New York, 1973.

\bibitem[KP93a]{Khov-Pukh1} 
	A. G. Khovanskii and A. V. Pukhlikov. 
	\textit{
	Finitely additive measures of virtual polyhedra. 
	} 
	St. Petersburg Math. J. 4 (1993), no. 2, 337--356

\bibitem[KP93b]{Khov-Pukh2} 
	A. G. Khovanskii and A. V. Pukhlikov. 
	\textit{
	The Riemann-Roch theorem for integrals and sums of quasipolynomials on virtual polytopes. 
	} 
	St. Petersburg Math. J. 4 (1993), no. 4, 789--812.

\bibitem[Kot05]{Kottwitz}
	R. Kottwitz.
	Harmonic analysis on reductive $p$-adic groups and Lie algebras. 
	In 
	\textit{
	Harmonic analysis, the trace formula, and Shimura varieties,
	} 
	393--522, Clay Math. Proc., 4, Amer. Math. Soc., 2005.

\bibitem[Lau96]{laumon-vol1}
	G. Laumon. 
	\textit{
	Cohomology of Drinfeld modular varieties. 
	} 
	Part {I}. Geometry, counting of points and local harmonic analysis. 
	Cambridge Studies in Advanced Mathematics, 41. 
	Cambridge University Press, 1996. 

\bibitem[Lau97]{laumon-vol2}
	G. Laumon. 
	\textit{
	Cohomology of Drinfeld modular varieties. 
	} 
	Part II. Automorphic forms, trace formulas and Langlands correspondence. 
	With an appendix by Jean-Loup Waldspurger. 
	Cambridge Studies in Advanced Mathematics, 56. 
	Cambridge University Press, 1997. 

\bibitem[Law91]{Lawrence}
	J. Lawrence. 
	\textit{
	Polytope volume computation. 
	} 
	Math. Comp., 57(195), 259--271, 1991.

\bibitem[Mc77]{McMullen} 
	P. McMullen
	\textit{
	Valuations and Euler-type relations on certain classes of convex polytopes. 
	} 
	Proc. London Math. Soc. (3) 35 (1977), no. 1, 113--135.

\bibitem[St12]{Stanley}
	R. P. Stanley.
	\textit{
	Enumerative combinatorics.
	}. 
	Volume 1. Second edition. 
	Cambridge Studies in Advanced Mathematics, 49. Cambridge University Press, 2012.

\bibitem[Vr87]{Varchenko}
	A. N. Varchenko. 
	\textit{
	Combinatorics and topology of the arrangement of affine hyperplanes in the real space. 
	} 
	Funktsional. Anal. i Prilozhen. 21 (1987), no. 1, 11--22.

\end{thebibliography}
\end{document}